\documentclass[leqno,a4wide]{amsart}
\usepackage{amssymb,amsmath}
\usepackage{times,bbm,epsfig,float,fancyhdr}
\usepackage[colorlinks = false]{hyperref}

\usepackage{tikz}
\usepackage{wrapfig}
\usepackage{subfigure}
\graphicspath{{figures/}}
\usepackage{algorithm}
\usepackage{algorithmic}

\numberwithin{equation}{section}

\newcommand{\norm}[1]{\ensuremath{\left\|#1\right\|}}

\newcommand{\abs}[1]{\ensuremath{\left|#1\right|}}

\newcommand{\Div}{\nabla\cdot\,}

\newcommand{\eps}{\ensuremath{\varepsilon}}

\newcommand{\Grad}{\ensuremath{\mathrm{\nabla}}}

\newcommand{\n}{\ensuremath{\nu}}

\newcommand{\Set}[1]{\ensuremath{\left\{#1\right\}}}
\newcommand{\R}{\ensuremath{\Bbb{R}}}
\newcommand{\U}{\ensuremath{\Bbb{U}}}
\newcommand{\bS}{\ensuremath{\Bbb{S}}}
\newcommand{\bK}{\ensuremath{\Bbb{K}}}
\newcommand{\bX}{\ensuremath{\Bbb{X}}}
\newcommand{\bB}{\ensuremath{\Bbb{B}}}

\newcommand{\B}{\ensuremath{\mathcal{B}}}

\newcommand{\Iion}{I_{\mathrm{ion}}}

\newcommand{\Om}{\ensuremath{\Omega}}

\newcommand{\Ion}{I_{\mathrm{ion}}}

\newcommand{\dx}{\ensuremath{\, dx}}

\newcommand{\dt}{\ensuremath{\, dt}}

\newcommand{\dW}{\ensuremath{\, dW}}
\newcommand{\ds}{\ensuremath{\, ds}}

\newcommand{\En}{\ensuremath{\mathbf{1}}}
\newcommand{\cI}{\ensuremath{\mathcal{I}}}
\newcommand{\cF}{\ensuremath{\mathcal{F}}}
\newcommand{\cFt}{\ensuremath{\cF_t}}
\newcommand{\tcFt}{\ensuremath{\tilde\cF_t}}
\newcommand{\cS}{\ensuremath{\mathcal{S}}}
\newcommand{\cL}{\ensuremath{\mathcal{L}}}
\newcommand{\cP}{\ensuremath{\mathcal{P}}}
\newcommand{\cB}{\ensuremath{\mathcal{B}}}
\newcommand{\E}{\ensuremath{\Bbb{E}}}

\newcommand{\cX}{\ensuremath{\mathcal{X}}}
\newcommand{\cA}{\ensuremath{\mathcal{A}}}
\newcommand{\cZ}{\ensuremath{\mathcal{Z}}}
\newcommand{\cK}{\ensuremath{\mathcal{K}}}
\newcommand{\loc}{\ensuremath{\text{loc}}}

\newcommand{\vrho}{\ensuremath{\varrho}}
\newcommand{\vphi}{\ensuremath{\varphi}}
\newcommand{\vtheta}{\ensuremath{\vartheta}}
\newcommand{\Span}{\ensuremath{\mathrm{Span}\!}}

\newcommand{\weak}{\rightharpoonup}
\newcommand{\weakstar}{\overset{\star}\rightharpoonup}

\newcommand{\tH}{\ensuremath{H_D^1}}

\newtheorem{thm}{Theorem}[section]
\newtheorem{lem}[thm]{Lemma}
\newtheorem{cor}[thm]{Corollary}

\newtheorem{rem}[thm]{Remark}
\newtheorem{defi}[thm]{Definition}

\allowdisplaybreaks

\begin{document}

\title[Stochastic bidomain model]
{Stochastically forced cardiac bidomain model}

\author[M. Bendahmane]{M. Bendahmane}
\address[Mostafa Bendahmane]
{\newline Institut de Math\'ematiques de Bordeaux UMR CNRS 525
\newline Universit\'e Victor Segalen Bordeaux 2
\newline F-33076 Bordeaux Cedex, France}
\email[]{mostafa.bendahmane@u-bordeaux.fr}

\author[K. H. Karlsen]{K. H. Karlsen}
\address[Kenneth Hvistendahl Karlsen]
{
\newline Department of mathematics
\newline University of Oslo
\newline P.O. Box 1053,  Blindern, N--0316 Oslo, Norway
} 
\email[]{kennethk@math.uio.no}



\thanks{This work was supported by the Research Council 
of Norway (project 250674/F20).}

\date{\today}

\begin{abstract}
The bidomain system of degenerate reaction-diffusion 
equations is a well-established spatial model of electrical activity in cardiac tissue, 
with ``reaction" linked to the cellular action potential 
and ``diffusion" representing current flow between cells.
The purpose of this paper is to introduce  a ``stochastically forced" version 
of the bidomain model that accounts for various random effects. 
We establish the existence of martingale (probabilistic weak) solutions 
to the stochastic bidomain model. The result is proved by means of an auxiliary 
nondegenerate system and the Faedo-Galerkin method. 
To prove convergence of the approximate solutions, we use the stochastic 
compactness method and Skorokhod-Jakubowski a.s.~representations. 
Finally, via a pathwise uniqueness result, we conclude that the martingale 
solutions are pathwise (i.e., probabilistic strong) solutions.
\end{abstract}

\maketitle
\tableofcontents

\section{Introduction}

\subsection{Background}
Hodgkin and Huxley \cite{Hodgkin:1952aa} introduced 
the first mathematical model for the propagation of electrical 
signals along nerve fibers. This model was later 
tweaked to describe assorted phenomena in biology. 
Similar to nerve cells, conduction of electrical signals in cardiac tissue rely 
on the flow of ions through so-called ion channels in the cell membrane. 
This similarity has led to a number of cardiac models based 
on the Hodgkin--Huxley formalism \cite{Clayton:2008aa,Colli-Franzone:2014aa,keenersneyd,Qu:2014aa,Rudy:2006aa,Sundnes:2006aa}.  
Among these is the \emph{bidomain model} \cite{tung}, which is regarded 
as an apt spatial model of the electrical properties of 
cardiac tissue \cite{Colli-Franzone:2014aa,Sundnes:2006aa}. 

The bidomain equations result from the principle of conservation of current 
between the intra- and extracellular domains, followed by a 
homogenization process of the cellular model defined on a periodic 
structure of cardiac tissue (see, e.g., \cite{Colli-Franzone:2014aa}). 
The bidomain model can be viewed as a PDE system, consisting of 
a degenerate parabolic (reaction-diffusion) PDE for the transmembrane potential 
and an elliptic PDE for the extracellular potential. These PDEs are supplemented 
by a nonlinear ODE system for the conduction dynamics of the ion channels.  
There are many membrane models of cardiac cells, differing in their complexity 
and in the level of detail with which they represent the 
biology (see \cite{Clayton:2008aa} for a review). 
Herein we will utilize a simple model for voltage-gated ion channels \cite{MS:Ion}.

The idiom ``bidomain''  reflects that the intra- and 
extracellular tissues are viewed as two superimposed anisotropic 
continuous media, with different longitudinal 
and transversal conductivities. If these conductivities
are equal, then we have the so-called monodomain model (elliptic PDE reduces 
to an algebraic equation). The degenerate structure of the bidomain PDE system 
is due to the anisotropy of cardiac tissue \cite{Bend-Karl:cardiac,Colli4}. 
Solutions exhibit discontinuous-like propagating 
excitation fronts. This, together with strongly varying time 
scales, makes the system difficult to solve by numerical methods.

The bidomain model is a deterministic system. 
This means that at each moment in time, the solution can 
be inferred from the prescribed data. This is at variance with 
several phenomena happening at the microscopic (cellular) 
and macroscopic (heart/torso) scales, where respectively 
channel noise and external random perturbations acting 
in the torso can play important roles. At the macroscopic level, the 
ECG signal, a coarse-grained representation of the electrical activity 
in the heart, is often contaminated by noise. One source for this noise 
is the fluctuating environment of the heart. 
In \cite{Lerma:2007aa}, the authors argue that such randomness 
cannot always be suppressed. Occasionally deterministic 
equations give qualitatively incorrect results, and it is 
important to quantify the nature of the noise and choose an appropriate model incorporating randomness.
 
At the cellular level, the membrane potential is due to disparities 
in ion concentrations (e.g., sodium, calcium, potassium) 
across the cell membrane. The ions move through the cell membrane due to
random transitions between open and close states of the ion channels. 
The dynamics of the voltage potential reflect the aggregated behaviour 
of the individual ion channels, whose conformational changes 
control the conductance of each ionic current.
The profound role of channel noise in excitable cells is summarised 
and discussed in \cite{Goldwyn:2011aa}. Faithful modeling of channel noise 
gives raise to continuous-time Markov chains with 
voltage-dependent transition probabilities. In the limit of infinitely 
many ion channels, these models lead to deterministic
Hodgkin--Huxley type equations. To capture channel noise, 
an alternative (and computationally much simpler) approach is to add 
well-placed stochastic terms to equations of the 
Hodgkin--Huxley type \cite{Goldwyn:2011aa,Laing-Lord:2010}.
Indeed, recent studies (see \cite{Goldwyn:2011aa} for a synthesis) 
indicate that this approach can give an accurate 
reproduction of channel fluctuations. For work specifically devoted to 
cardiac cells, see \cite{Dvir:2013aa,Lerma:2007aa,Qu:2014aa}.

\subsection{Deterministic bidomain equations}
Fix a final time $T>0$ and a bounded open
subset $\Omega \subset \R^3$ representing the 
heart (cf.~Section \ref{sec:model}). 
Roughly speaking, the bidomain equations result from 
applying Ohm's electrical conduction law and the 
continuity equation (conservation of electrical charge) 
to the intracellular and extracellular domains.  
Let $J_i$ and $J_e$ denote, respectively, the 
current densities in the intracellular and extracellular 
domains. Moreover, denote by $I_m$ the membrane 
current per unit volume and by $I_i,I_e$ the 
injected stimulating currents. The continuity equations are
\begin{equation}\label{eq:cont_eqs}
	\nabla \cdot J_i = -I_m + I_i, \quad 
	\nabla \cdot J_e = I_m + I_e.
\end{equation}
The negative sign in the first equation reflects that the 
current leaving the intracellular domain is positive.
We assume that the intracellular and extracellular current densities
can be written in terms of potentials $u_i,u_e$ as follows:
$$
J_i = - M_i \nabla u_i, \qquad 
J_e = - M_e \nabla u_e, 
$$
where $M_i,M_e$ are the intracellular 
and extracellular conductivity tensors. The transmembrane potential $v$ 
is defined as $v:=u_i-u_e$. Hence, the continuity 
equations \eqref{eq:cont_eqs} become
\begin{equation}\label{eq:cont_eqs2}
	-\nabla \cdot \left(M_i \nabla u_i\right) = -I_m + I_i, \quad 
	-\nabla \cdot \left(M_e\nabla u_e\right) = I_m + I_e.
\end{equation}

By adding the equations in \eqref{eq:cont_eqs2}, we obtain
\begin{equation}\label{eq:cont_eqs3}
	-\nabla \cdot \left ( (M_i+M_e)\nabla u_e\right)
	-\nabla \cdot \left (M_i \nabla v\right) 
	= I_i+I_e \quad \text{in $\Omega\times (0,T)$}.
\end{equation}
The membrane current $I_m$ splits into a capacitive 
current $I_c$, since the cell membrane acts as a capacitor,
and an ionic current, due to the flowing of ions through 
different ion channels (and also pumps/exchangers):
\begin{equation}\label{eq:mcurrent}
	I_m= \chi_m\left(I_c + \Iion\right), \qquad 
	I_c= c_m \frac{\partial v}{\partial t},\qquad
	\Iion = \Iion(v,w),
\end{equation}
where $\chi_m$ is the ratio of membrane surface area to 
tissue volume and $c_m>0$ is the (surface) capacitance 
of the membrane per unit area. The (nonlinear) function $\Iion(v,w)$ 
represents the ionic current per unit surface area, which 
depends on the transmembrane potential $v$ and a vector $w$ of ionic 
(recovery, gating, concentrations, etc.)~variables.  
A simplified model, frequently used for analysis, assumes that the 
functional form of $\Iion$ is a cubic polynomial in $v$. 
The ionic variables $w$ are governed by an ODE system,
\begin{equation}\label{eq:ODEintro}
	\frac{\partial w}{\partial t} = H(v,w) 
	\quad \text{in $\Omega\times (0,T)$},
\end{equation}
where, as alluded to earlier, various membrane models exist 
for cardiac cells, giving raise to different choices of $H$ (and $\Iion$). 
Inserting \eqref{eq:mcurrent} into \eqref{eq:cont_eqs2}, we arrive at
\begin{equation}\label{eq:veq1}
	\chi_m c_m \frac{\partial v}{\partial t}
	-\nabla \cdot \left(M_i (\nabla (v+ u_e)\right) +\chi_m \Iion(v,w)=I_i
	\quad \text{in $\Omega\times (0,T)$}.
\end{equation}
The system \eqref{eq:cont_eqs3}, \eqref{eq:ODEintro}, \eqref{eq:veq1} 
is sometimes referred to as the \textit{parabolic-elliptic form} of 
the bidomain model, as it contains a parabolic PDE \eqref{eq:veq1} for the 
transmembrane potential $v$ and an elliptic PDE \eqref{eq:cont_eqs3} 
for the extracellular potential $u_e$. The bidomain equations 
are closed by specifying initial conditions for $v,w$ 
and boundary conditions for $u_i, u_e$. Electrically isolated 
heart tissue, for example, leads to zero flux boundary conditions.   

Herein we will rely on a slightly different form of the 
bidomain model, obtained by inserting \eqref{eq:mcurrent} 
into both equations in \eqref{eq:cont_eqs2}: 
\begin{equation}\label{eq:veq2}
	\begin{split}
		& \chi_m c_m \frac{\partial v}{\partial t}
		-\nabla \cdot \left(M_i \nabla u_i\right) +\chi_m \Iion(v,w)=I_i
		\quad \text{in $\Omega\times (0,T)$},
		\\
		&\chi_m c_m \frac{\partial v}{\partial t}
		+\nabla \cdot \left(M_e \nabla u_e\right) +\chi_m \Iion(v,w)=-I_e
		\quad \text{in $\Omega\times (0,T)$}.
	\end{split}
\end{equation}
Consisting of two (degenerate) parabolic PDEs, the 
system \eqref{eq:ODEintro}, \eqref{eq:veq2} is occasionally referred to 
as the \textit{parabolic-parabolic form} of the bidomain model.  
On the subject of well-posedness, i.e., existence, 
uniqueness, and stability of properly defined solutions, we 
remark that standard theory for parabolic-elliptic 
systems does not apply naturally. 
The main reason is that the anisotropies of the intra- and 
extracellular domains differ, entailing the degenerate structure of the system. 
Moreover, a maximum principle is not available. That being the case, a number of works 
\cite{Andreianov:2010uq,Bend-Karl:cardiac,Boulakia:2008aa,Bourgault:2009aa,Colli-Franzone:2014aa,Colli4,Giga:2018aa,Kunisch:2013aa,Veneroni:2009aa} have recently 
provided well-posedness results for the bidomain model, applying differing 
solution concepts and technical frameworks.

\subsection{Stochastic model \& main results}
The purpose of the present paper is to introduce 
and analyze a bidomain model that accounts for 
random effects (noise), by way of a few well-placed stochastic terms. 
The simplest way to insert randomness is to add Gaussian 
white noise to one or more of the ionic ODEs \eqref{eq:ODEintro}, 
leading to a system of (It\^{o}) stochastic 
differential equations (SDEs):
\begin{equation}\label{eq:introSDE}
	dw  = H(v,w)\dt + \alpha \dW^w,
\end{equation}
where $W^w$ is a cylindrical Wiener process, with noise amplitude $\alpha$. 
Formally, we can think of $\alpha\, dW^w$ as 
$\sum_{k\ge1}\alpha_k \, d W_k^w(t)$, where $\{W_k^w\}_{k\ge 1}$ 
is a sequence of independent 1D Brownian 
motions and $\{\alpha_k\}_{k\ge 1}$ 
is a sequence of noise coefficients. Interpreting $w$ as gating 
variables representing the fraction of open channel 
subunits of varying types, in \cite{Goldwyn:2011aa} this type of 
noise is referred to as \textit{subunit noise}. We will allow for 
subunit noise in our model, assuming for simplicity that the ionic 
variable $w$ is a scalar and that the noise amplitude depends on 
the transmembrane potential $v$, $\alpha=\alpha(v)$ (multiplicative noise).
We will also introduce fluctuations into the bidomain system by replacing 
the PDEs \eqref{eq:veq2} with the (It\^{o}) stochastic 
partial differential equations (SPDEs)
\begin{equation}\label{eq:veq-noise}
	\begin{split}
		& \chi_m c_m dv
		-\nabla \cdot \left(M_i \nabla u_i\right)\dt +\chi_m \Iion(v,w)\dt
		=I_i\dt+\beta \dW^v
		\\
		&\chi_m c_m dv
		+\nabla \cdot \left(M_e \nabla u_e\right)\dt +\chi_m \Iion(v,w)\dt
		=-I_e\dt+\beta \dW^v,
	\end{split}
\end{equation}
where $W^v$ is a cylindrical Wiener process (independent of $W^w$), with 
noise amplitude $\beta$. Adding a stochastic term 
to the equation for the membrane potential $v$ 
is labeled \textit{current noise} in \cite{Goldwyn:2011aa}. 
Current noise represents the aggregated effect of the random 
activity of ion channels on the voltage dynamics. 
Allowing the noise amplitude in \eqref{eq:veq-noise} 
to depend on the membrane voltage $v$, we arrive 
at equations with so-called 
\textit{conductance noise} \cite{Goldwyn:2011aa}. 
The nonlinear term $\Iion(v,w)$ accounts for the total conductances of various ionic currents, and conductance noise pertains 
to adding ``white noise" to the deterministic 
values of the conductances, i.e., 
replacing $\Iion$ by $\Iion+ \hat\beta(v)\, \frac{dW_v}{dt}$, 
for some function $\hat \beta$. Herein we include this 
case by permitting $\beta$ in \eqref{eq:veq-noise} 
to depend on the voltage variable $v$, $\beta=\beta(v)$.

Our main contribution is to establish the existence of properly 
defined solutions to the SDE-SPDE system \eqref{eq:introSDE}, \eqref{eq:veq-noise}.
From the PDE perspective, we are searching for weak solutions 
in a certain Sobolev space ($H^1$). From the probabilistic point of view, we are considering martingale solutions, sometimes also referred to as weak solutions. The notions of weak \& strong probabilistic solutions have 
different meaning from weak \& strong solutions in the PDE literature. 
If the stochastic elements are fixed in advance, we speak 
of a strong (or pathwise) solution. 
The stochastic elements are collected in a stochastic basis 
$\bigl(\Omega, \cF, \Set{\cFt}_{t\in [0,T]}, P,W\bigr)$, 
where $W=(W_w,W_v)$ are cylindrical Wiener processes adapted to the 
filtration $\Set{\cFt}_{t\in [0,T]}$. 
Whenever these elements constitute a part of the unknown solution, the relevant notion is that of a martingale solution. The connection between weak and strong solutions to It\^{o} equations is exposed in the famous 
Yamada-Watanabe theorem, see, e.g., \cite{Ikeda:1981aa,Kurtz:2014aa,Prevot:2007aa}. We reserve the name \textit{weak martingale solution} for solutions that are weak in the PDE sense as well as being probabilistic weak.

We will prove that there exists a weak martingale 
solution to the stochastic bidomain system. Motivated by the 
approach in \cite{Bend-Karl:cardiac} (see also \cite{Boulakia:2008aa}) 
for the deterministic system, we use the Faedo-Galerkin 
method to construct approximate solutions, based on an 
auxiliary nondegenerate system obtained by adding 
$\eps d u_i$ and $- \eps du_e$ respectively to 
the first and second equations in \eqref{eq:veq-noise} ($\eps$ is a 
small positive parameter). The stochastic compactness method is put to 
use to conclude subsequential convergence of the approximate solutions. 

Indeed, we first apply the It\^{o} chain rule to derive some basic
a priori estimates. The combination of multiplicative noise and the 
specific structure of the system makes these estimates 
notably harder to obtain than in the deterministic case. 
The a priori estimates lead to strong compactness of 
the approximations in the $t,x$ variables (in the deterministic 
context \cite{Bend-Karl:cardiac}). In the stochastic setting, there is an 
additional (probability) variable $\omega\in D$ in which strong 
compactness is not expected. Traditionally, one handles this issue by arguing for 
weak compactness of the probability laws of the approximate 
solutions, via tightness and Prokhorov's theorem \cite{Ikeda:1981aa}. 
The ensuing step is to construct a.s.~convergent versions of 
the approximations using the Skorokhod representation theorem.
This theorem supplies new random variables on a new probability 
space, with the same laws as the original variables, converging almost surely. 
Equipped with a.s.~convergence, we are able to 
show that the limit variables constitute a weak martingale solution. 
Finally, thanks to a uniqueness result and the Gy\"ongy-Krylov 
characterization of convergence in probability \cite{Gyong-Krylov}, 
we pass \textit{\`{a} la} Yamada-Watanabe \cite{Ikeda:1981aa} 
from martingale to pathwise (probabilistic strong) solutions. 

Martingale solutions and the stochastic compactness method 
have been harnessed by many authors for 
different classes of SPDEs, see e.g.~\cite{Bensoussan:1995aa,Bessaih:1999aa,DaPrato:2014aa,Debussche:2011aa,Flandoli:2008vn,Flandoli:1995aa,Glatt-Holtz:2014aa,Hausenblas:2013aa,Hofmanova:2013aa,Mohammed:2015ab,Razafimandimby:2015aa,Sango:2010aa} for problems related to fluid mechanics.  An important step in the compactness method 
is the construction of almost surely convergent versions of processes that converge weakly. 
This construction dates back to the work of Skorokhod, for processes taking values in 
a Polish (complete separable metric) space \cite{DaPrato:2014aa,Ikeda:1981aa}. 
The classical Skorokhod theorem is befitting for the transmembrane  variable $v$, but not the 
intracellular and extracellular variables $u_i,u_e$. This fact is a manifestation of the 
degenerate structure of the bidomain system, necessitating the use of 
a Bochner-Sobolev space equipped with the weak topology. 
We refer to Jakubowski \cite{Jakubowski:1997aa} for a recent
variant of the representation theorem that applies to so-called 
quasi-Polish spaces, specifically allowing for separable Banach 
spaces equipped with the weak topology, as well as spaces of weakly continuous 
functions with values in a separable Banach space. 
We refer to \cite{Breit:2016aa,Brzezniak:2013aa,Brzezniak:2013ab,Brzezniak:2011aa,Ondrejat:2010aa,Smith:2015aa} for  works making use of 
Skorokhod-Jakubowski a.s.~representations.  

The remaining part of this paper is organized as follows: The stochastic 
bidomain model is presented in Section \ref{sec:model}. Section \ref{sec:stoch} 
outlines the underlying stochastic framework and list 
the conditions imposed on the ``stochastic" data of the model. 
Solution concepts and the accompanying main results are collected in 
Section \ref{sec:defsol}. The approximate (Faedo-Galerkin) solutions 
are constructed in Section \ref{sec:approx-sol}. In Section \ref{sec:convergence} 
we establish several a priori estimates and prove convergence of the 
approximate solutions, thereby providing an existence result for weak martingale solutions. 
A pathwise uniqueness result is established in Section \ref{sec:uniq}, which 
is then used in Section \ref{sec:pathwise} to upgrade martingale 
solutions to pathwise solutions.

\section{Stochastic bidomain model}\label{sec:model}
The spatial domain of the heart is given by a bounded open set 
$\Omega \subset \R^3$ with piecewise 
smooth boundary $\partial \Om$. This three-dimensional 
slice of the cardiac muscle is viewed as two superimposed
(anisotropic) continuous media, representing 
the intracellular ($i$) and extracellular ($e$) tissues. 
The tissues are connected at each point via the cell membrane. 
In our earlier outline of the (deterministic) bidomain model, 
we saw that the relevant quantities are the \textit{intracellular} and 
\textit{extracellular} potentials 
$$
u_i=u_i(x,t) \quad \text{ and }\quad u_e=u_e(x,t), 
\qquad  (x,t)\in \Om_{T}:=\Om\times(0,T),
$$ 
as well as the \textit{transmembrane  potential} $v:=u_i-u_e$ (defined in $\Om_{T}$).

The conductivities of the intracellular and extracellular 
tissues are encoded in anisotropic matrices 
$M_i=M_i(x)$ and $M_e=M_e(x)$. Cardiac tissue is composed of 
fibers, with the conductivity being higher along the fibers 
than in the cross-fibre direction. The cardiac fibers are organized 
in sheets of varying surface orientation, giving raise to three 
principal directions for conduction: parallel to the fibers, perpendicular to the fibers 
but parallel to the sheet, and perpendicular to the sheet.
When the fibers rotate from bottom to top, we have 
cardiac tissue with rotational anisotropy. 
Without rotation (axisymmetric anisotropy), the conductivity 
tensors take the form
\begin{equation}\label{eq:Mj-def}
	M_j=\sigma_j^t I +\left(\sigma_j^l-\sigma_j^t\right) aa^T,
	\quad j=i,e,
\end{equation}
where $a=a(x)$ is a unit vector giving the fiber direction and 
$\sigma_j^l =\sigma_j^l(x)$, $\sigma_j^t=\sigma_j^t(x)$ are coefficients 
describing respectively the intra- and extracellular conductivities along 
and transversal to the fibre direction. Whenever the fibers are aligned with the 
axes, $M_i$ and $M_e$ are diagonal matrices. For more 
details, see, e.g., \cite{Coli-fibers,Colli-Franzone:2014aa,Sundnes:2006aa}.
In this paper, we do not exploit structural properties of 
cardiac tissue, that is, we assume only that $M_i$ and $M_e$ 
are general (bounded, positive definitive) matrices, cf.~\eqref{matrix} below. 

The stochastic bidomain model contains two 
nonlinearly coupled SPDEs involving the potentials 
$u_i,u_e,v$. These stochastic reaction-diffusion equations are further 
coupled to a nonlinear SDE for the gating (recovery) variable $w$. 
The dynamics of $(u_i,u_e,v,w)$ is governed by the equations
\begin{align}\label{S0}
	\begin{split}
		&\chi_m c_m d v
		-\Div \bigl(M_i \Grad u_i \bigr) \dt
		+\chi_m \Ion(v,w)\dt = I_i\dt+\beta(v) \dW^v 
		\quad \text{in $\Om_{T}$},\\
		& \chi_m c_m d v
		+\Div\bigl(M_e\Grad u_e\bigr) \dt
		+\chi_m \Ion(v,w)\dt  = -I_e\dt+\beta(v) \dW^v
		\quad \text{in $\Om_{T}$},\\
		& d w = H(v,w)\dt + \alpha(v)\dW^w 
		\quad  \text{in $\Om_{T}$},
	\end{split}
\end{align}
where $c_m>0$ is the surface capacitance 
of the membrane, $\chi_m$ is the surface-to-volume ratio, and 
$I_i,I_e$ are stimulation currents. In \eqref{S0}, randomness is 
represented by cylindrical Wiener processes 
$W^v,W^w$ with nonlinear noise amplitudes $\beta,\alpha$ 
(cf.~Section \ref{sec:stoch} for details).

We impose initial conditions on the 
transmembrane potential and the gating variable:
\begin{equation}\label{S-init}
	v(0,x)=v_{0}(x),
	\qquad 
	w(0,x)=w_0(x), 
	\qquad 
	x \in \Om.
\end{equation}
The intra- and extracellular domains are often assumed to 
be electrically isolated, giving raise to zero flux (Neumann type) 
boundary conditions on the potentials $u_i, u_e$ 
\cite{Colli-Franzone:2014aa,Sundnes:2006aa}. 
From a mathematical point of view, Dirichlet and mixed 
Dirichlet-Neumann type boundary conditions are utilized  in \cite{Andreianov:2010uq} 
and \cite{Bend-Karl:cardiac}, respectively. Herein we partition the 
boundary $\partial\Om$ into regular parts $\Sigma_N$ and $\Sigma_D$ 
and impose the mixed boundary conditions ($j=i,e$)
\begin{equation}\label{S-bc}
	\begin{split}
		& \bigl(M_j(x)\Grad u_j)\cdot\n =0 
		\quad \text{on $\Sigma_{N,T}:=\Sigma_N\times(0,T)$},\\
		& u_j=0 
		\qquad \qquad \qquad \,\, \, \,\text{on $\Sigma_{D,T}:=\Sigma_D\times(0,T)$},
\end{split}\end{equation}
where $\n$ denotes the exterior unit normal to the ``Neumann 
part" $\Sigma_N$ of the boundary, which is defined a.e.~with respect to the 
two-dimensional Hausdorff measure $\mathcal{H}^2$ on $\partial\Om$.

Observe that the equations in \eqref{S0} are invariant under the 
change of $u_i$ and $u_e$ into $u_i+k,u_e+k$, for any $k\in \R$. 
Hence, unless Dirichlet conditions are imposed 
somewhere ($\Sigma_D \neq\emptyset$), the bidomain 
system determines the electrical potentials only up to an additive constant. 
To ensure a unique solution in the case $\Sigma_D:=\emptyset$ 
($\partial\Om=\Sigma_N$), we may impose the normalization condition
$$
\int_{\Om} u_e(x,t)\dx=0, \qquad t\in(0,T).
$$
To avoid making this paper too long, we assume that $\Sigma_D \neq\emptyset$. 
Moreover, we stick to homogenous boundary conditions, although we could have 
replaced the right-hand sides of \eqref{S-bc} by 
sufficiently (Sobolev) regular functions.

Regarding the ``membrane" functions $\Ion$ and $H$, we have 
in mind the fairly uncluttered FitzHugh-Nagumo model \cite{FH,Nagumo}. This 
is a simple choice for the membrane kinetics that is 
often used to avoid difficulties arising from a large number of coupling variables. 
The model is specified by
\begin{align*}
	& \Ion(v,w)=-v\left(v-a\right)(1-v) + w,
	\\ &
	H(v,w) =\epsilon (\kappa v-\gamma w),
\end{align*}
where the parameter $a$ represents the threshold for excitation, $\epsilon$ 
represents excitability, and $\kappa, \gamma, \delta$ are parameters that 
influence the overall dynamics of the system.  For background material 
on cardiac membrane models and their general mathematical structure, we refer 
to the books \cite{Colli-Franzone:2014aa,keenersneyd,Sundnes:2006aa}.

In an attempt to simplify the notation, we 
redefine $M_i, M_e$ as
$$
\frac{1}{\chi_m c_m} M_i, \quad 
\frac{1}{\chi_m c_m} M_e, \quad
$$ 
and set
$$
I := \frac{1}{c_m} \Iion, \quad \eta :=\frac{1}{\chi_m c_m} \beta.
$$
We also assume $I_i,I_e \equiv 0$, as these source terms do 
not add new difficulties. The resulting stochastic 
bidomain system becomes
\begin{equation}\label{S1}
	\left \{
	\begin{split}
		&d v -\Div\bigl(M_i \Grad u_i \bigr) \dt
		+ I(v,w)\dt = \eta(v) \dW^v \quad \text{in $\Om_{T}$},\\
		& d v +\Div\bigl(M_e\Grad u_e\bigr) \dt
		+ I(v,w)\dt  =  \eta(v) \dW^v \quad \text{in $\Om_{T}$},\\
		& d w=H(v,w)\dt +\sigma(v) \dW^w \quad  \text{in $\Om_{T}$}, 
	\end{split} \right.
\end{equation}
along with the initial and boundary conditions \eqref{S-init} and \eqref{S-bc}.  
The cylindrical Wiener processes 
$W^v,W^w$ in \eqref{S1} are defined in Section \ref{sec:stoch}.

With regard to the conductivity matrices in \eqref{S1}, we assume
the existence of positive constants $m,M$ such that for $j=i,e$,
\begin{equation}\label{matrix}
	M_j\in L^\infty, 
	\quad 
	m\abs{\xi}^2 \le \xi^\top M_j(x) \, \xi  \le M \abs{\xi}^2, 
	\quad 
	\text{$\forall \xi\in \R^3$, for a.e.~$x$.}
\end{equation}

Motivated by the discussion above on membrane models, we impose 
the following set of assumptions on the functions $I,H$ in \eqref{S1}:
\begin{itemize}
	\item Generalized FitzHugh-Nagumo model {\rm (\textbf{GFHN})}:
	$$
	I(v,w)=I_1(v)+I_2(v)w,\qquad H(v,w)=h(v)+c_{H,1} w,
	$$
	where $I_1,I_2, h\in C^1(\R)$ and for all $v\in \R$,
	\begin{align*}
		& \abs{I_1(v)}\le c_{I,1}\left(1+ \abs{v}^3\right), 
		\quad 
		I_1(v)v \ge \underline{c}_I \abs{v}^4 - c_{I,2} \abs{v}^2,
		\\ & I_2(v)=c_{I,3}+c_{I,4}v, \quad 
		\abs{h(v)}^2\le c_{H,2}\left(1 + \abs{v}^2\right),
	\end{align*}
	for some positive constants $c_{I,1},c_{I,2},c_{I,3},c_{I,4}, c_{H,1},c_{H,2}$ 
	and $\underline{c}_I>0$.\\
\end{itemize}

We end this section with a remark about 
the so-called \textit{monodomain model}.

\begin{rem}
The stochastic bidomain model simplifies 
if the anisotropy ratios $\sigma_i^l/\sigma_i^t$
and $\sigma_e^l/ \sigma_e^t$ are equal, cf.~\eqref{eq:Mj-def}.
Indeed, suppose $\sigma_i^l / \sigma_i^t=\sigma_e^l/\sigma_e^t$ 
and moreover $\sigma_i^t=\lambda \sigma_e^t$ for some 
constant $\lambda>0$ (and thus $\sigma_i^l=\lambda \sigma_e^l$). 
Then, in view of \eqref{eq:Mj-def}, it follows 
that $M_i =\lambda M_e$, and hence the first two 
equations in \eqref{S1} can be combined into a single 
equation; thereby arriving at the stochastic monodomain system
\begin{align}\label{S1-mono}
	\begin{split}
		&d v -\Div\left(M \Grad v \right) \dt
		+ I(v,w)\dt = \eta(v) \, d W^v \quad \text{in $\Om_{T}$},\\
		& d w=H(v,w)\dt +\sigma(v) d W^w \quad  \text{in $\Om_{T}$}, 
	\end{split}
\end{align}
where $M:=\frac{\lambda}{1+\lambda} M_i$. The system \eqref{S1-mono} 
is a significant simplification of the bidomain model \eqref{S1}, and even 
though the assumption of equal anisotropy ratios is very strong, the 
monodomain model is adequate in certain situations \cite{Colli3}.
\end{rem}

\section{Stochastic framework}\label{sec:stoch}
To define the cylindrical Wiener processes $W=W^v,W^w$ and the stochastic 
integrals appearing in \eqref{S1}, we need to briefly recall some basic concepts 
and results from stochastic analysis. For a detailed account we 
refer to \cite{DaPrato:2014aa,Prevot:2007aa} (see also \cite{Ikeda:1981aa,Karatzas}). 

We consider a complete probability space $(D,\cF, P)$, along 
with a complete right-continuous filtration $\Set{\cFt}_{t\in [0,T]}$. 
Without loss of generality, we assume that the $\sigma$-algebra 
$\cF$ is countably generated. Let $\bB$ be a separable Banach 
space (equipped with the Borel $\sigma$-algebra $\cB(\bB)$). Then a $\bB$-valued 
random variable $X$ is a measurable mapping from $(D,\cF, P)$ to 
$(\bB,\cB(\bB))$, $D\ni \omega\mapsto X(\omega)\in \bB$. 
The expectation of a random variable $X$ is denoted 
$\E[X]:=\int_{D} X\, dP$. We use the abbreviation 
a.s.~(almost surely) for $P$-almost every $\omega\in D$.
The collection of all (equivalence classes of) $\bB$-valued random variables 
is denoted by $L^1(D,\cF,P)$. Equipped with the norm
$\norm{X}_{L^1(D,\cF,P)}=\E\left[ \norm{X}_B\right]$ it becomes 
a Banach space. For $p>1$ we  define $L^p(D,\cF,P)$ similarly, with norm 
\begin{align*}
	&\norm{X}_{L^p(D,\cF,P)}
	:=\left(\E\left[ \norm{X}_B^p\right]\right)^{\frac{1}{p}} \quad  (p<\infty),
	\\ & 
	\norm{X}_{L^\infty(D,\cF,P)}:= \sup_{\omega\in D} \norm{X(\omega)}_B.
\end{align*}
In this paper we deal with time-dependent functions and processes, so 
$\bB$ is typically an ``evolutionary" (Bochner) space like 
$\bB=L^{q_t}((0,T);L^{q_x}(\Om))$, $q_t,q_x\in [1,\infty]$.

A stochastic process $X=\Set{X(t)}_{t\in [0,T]}$ is a collection 
of $\bB$-valued random variables $X(t)$.  We write $X(\omega,t)$ 
instead of $X(t)(\omega)$ for $(\omega,t)\in D\times [0,T]$. We 
say that $X$ is \textit{measurable} if the map $X:D\times [0,T]\to \bB$ 
is (jointly) measurable from $\cF \times \cB([0,T])$ to $\cB(\bB)$.
The paths $t \to X (\omega,t)$ of a (jointly) measurable 
process $X$ are automatically Borel measurable functions. 
Of course, it is the use of filtrations and adaptivity that differentiates the 
theory of stochastic processes from the one of functions on product spaces. 
A stochastic process $X$ is \textit{adapted} if $X(t)$ is $\cF_t$ measurable for 
all $t\in [0,T]$. Let $\Set{W_k}_{k=1}^\infty$ 
be a sequence of independent one-dimensional 
Brownian motions adapted to the filtration
$\Set{\cFt}_{t\in [0,T]}$. We refer to
\begin{equation}\label{eq:stochbasis}
	\cS=\left(D,\cF,\Set{\cFt}_{t\in [0,T]},P,\Set{W_k}_{k=1}^\infty\right)
\end{equation}
as a (Brownian) \textit{stochastic basis}. When a filtration is 
involved there are additional notions of measurability (predictable, optional, progressive) 
that occasionally are more convenient to work with. 
Herein we use the (stronger) notion of a predictable process, but we could 
also have used that of a progressive process.  
A \textit{predictable} process is a $\cP_T\times \cB([0,T])$ measurable map
$D\times [0,T]\to \bB$, $(\omega,t)\mapsto X(\omega,t)$, 
where $\cP_T$ is the predictable $\sigma$-algebra 
on $D\times [0, T]$ associated with $\Set{\cFt}_{t\in [0,T]}$, i.e., 
the $\sigma$-algebra generated by all left-continuous 
adapted processes; $\cP_T$ is also generated by the sets
$$
F\times (s,t], \, F\in \cF_s, \, 0\le s<t\le T, \,  \qquad F\times \Set{0},\, F\in \cF_0.
$$
A predictable process is adapted and measurable. Although the 
converse implication is not true, adaptive processes 
with regular (e.g.~continuous) paths are predictable.

Fix a separable Hilbert space $\U$, equipped with 
a complete orthonormal basis $\Set{\psi_k}_{k\ge 1}$. 
We utilize cylindrical Brownian motions $W$ evolving over $\U$. Informally, 
the are defined by $W:=\sum_{k\ge 1} W_k \psi_k$. 
We have to be a bit more precise, however, since
the series does not converge in $\U$:
$$
\E\left[\norm{\sum_{k=1}^n W_k(t) \psi_k}_{\U}^2\right]=
\sum_{k=1}^n \E\left[\left(W_k(t)\right)^2\right]=n\, t \to \infty, 
\quad \text{as $n\uparrow \infty$.}
$$
Let $\bX$ be a separable Hilbert space with inner product $(\cdot,\cdot)_{\bX}$ and 
norm $\norm{\cdot}_{\bX}$. For the bidomain model \eqref{S1}, a natural choice 
is $\bX=L^2(\Om)$. The vector space of all bounded 
linear operators from $\U$ to $\bX$ is denoted $L(\U,\bX)$. We denote by 
$L_2(\U,\bX)$ the collection of Hilbert-Schmidt operators from $\U$ to $\bX$, i.e., 
$R\in L_2(\U,\bX)\Longleftrightarrow R\in L(\U,\bX)$ and
\begin{equation}\label{def:HS-norm}
	\norm{R}_{L_2(\U,\bX)}^2:= \sum_{k\geq 1}
	\norm{R\psi_k}_{\bX}^2<\infty.
\end{equation}
We turn $L_2(\U,\bX)$ into a Hilbert space by 
supplying it with the inner product 
$$
\left ( \hat R,\tilde R \right)_{L_2(\U,\bX)}
=\sum_{k\ge 1} \left(\hat R \psi_k,\tilde R \psi_k\right)_X, 
\qquad \hat R, \tilde R\in L_2(\U,\bX). 
$$
Returning to the convergence of $W=\sum_{k\ge 1} W_k \psi_k$, it is 
always possible to construct an auxiliary Hilbert 
space $\U_0\supset \U$ for which there exists a Hilbert-Schmidt 
embedding $J:\U\to \U_0$.  Indeed, given a sequence 
$\Set{b_k}_{k\ge 1}\in \ell^2$ with $b_k\neq 0$ 
for all $k$ (e.g., $b_k=1/k$), introduce the space
\begin{equation}\label{eq:U0}
	\U_0=\left\{u=\sum_{k\geq 1} a_k \psi_k: 
	\sum_{k\geq 1}a_k^2b_k^2<\infty\right\},
\end{equation}
and equip it with the norm $\norm{u}_{\U_0}:= 
\left( \sum_{k\geq 1}a_k^2b_k^2\right)^{\frac12}$, 
$u=\sum_{k\geq 1} a_k \psi_k$. Clearly, the 
embedding $J:\U\to \U_0$, $u\mapsto u$, 
is Hilbert-Schmidt. Now, on this larger Hilbert 
space $\U_0$, the infinite series 
$\sum_{k\ge 1} W_k \psi_k$ converges; indeed, as $n\to \infty$,
$$
\E\left[\norm{J\sum_{k=1}^n W_k(t) \psi_k}_{\U_0}^2\right]=
\sum_{k=1}^n \E\left[\left(W_k(t)\right)^2\right]
\norm{J\psi_k}_{\U_0}^2 \to t\norm{J}_{L_2(\U,\U_0)}^2.
$$

To summarize, the stochastic process 
\begin{equation}\label{eq:Wdef}
	W(\omega,t,\cdot):=\sum_{k\ge 1} W_k(\omega,t) \psi_k(\cdot)
\end{equation}
is referred to as a cylindrical Brownian motion evolving over $\U$. The
right-hand side of \eqref{eq:Wdef} converges on the Hilbert space $\U_0$,
with the embedding $\U\subset \U_0$ being Hilbert-Schmidt. 
Via standard martingale arguments, $W$ is a.s.~continuous 
with values in $\U_0$, i.e., $W(\omega,\cdot,\cdot)$ belongs 
to $C([0,T];\U_0)$ for $P$-a.e.~$\omega\in D$, and 
also $L^2(D,\cF,P; C([0,T];\U_0))$. Without loss of generality, we assume 
that the filtration $\Set{\cFt}_{t\in [0,T]}$ is generated by 
$W$ and the initial condition. See \cite{DaPrato:2014aa,Prevot:2007aa} for details.

Given a cylindrical Brownian motion $W$, we can define the It\^{o} stochastic 
integral $\int G dW$ as follows \cite{DaPrato:2014aa,Prevot:2007aa}:
\begin{equation}\label{def:sint}
	\int_0^t G dW=\sum_{k=1}^\infty \int_0^t G_k \dW_k, 
	\qquad G_k := G \psi_k,
\end{equation}
provided the integrand $G$ is a predictable $X$-valued 
process satisfying 
$$
G \in L^2\left(D,\cF,P;L^2((0,T);L_2(\U,\bX))\right).
$$
The stochastic integral \eqref{def:sint} is an 
$\bX$-valued square integrable martingale, satisfying the 
Burkholder-Davis-Gundy inequality
\begin{equation}\label{eq:bdg}
	\E\left[ \sup_{t\in [0,T]} \norm{\int_0^t G dW}_{\bX}^p \right]
	\le C\, \E\left[\left(\int_0^T \norm{G}_{L_2(\U,\bX)}^2 \dt\right)^{\frac{p}{2}} \right],
\end{equation}
where $C$ is a constant depending on $p\ge 1$. 
In terms of the basis $\Set{\psi_k}_{k\ge1}$, \eqref{eq:bdg} reads
$$
\E\left[ \sup_{t\in [0,T]} 
\norm{\sum_{k\ge 1} \int_0^t G_k dW_k}_{\bX}^p \right]
\le C\, \E\left[\left(\int_0^T \sum_{k\ge 1}
\norm{G_k}_{\bX}^2 \dt\right)^{\frac{p}{2}} \right].
$$

For the bidomain model \eqref{S1}, we take $\bX=L^2(\Om)$. 
With this choice, we can give meaning to the stochastic terms
$$
\int_{\Om} \left(\, \int_0^t  \beta(v) \dW\right) \vphi \dx, 
\qquad \text{$(\beta,W)=(\eta,W^v)$ or $(\sigma,W^w)$,}
$$
appearing in the weak formulation of \eqref{S1}, with $\vphi\in L^2(\Om)$. 
Since $W=\sum_{k\ge 1} W_k \psi_k$  
is a cylindrical Brownian motion, we can write
\begin{equation}\label{eq:int-W}
	\begin{split} 
		\int_{\Om} \left(\, \int_0^t  \beta(v) \dW\right) \vphi \dx
		& =\int_{\Om} \left(\, \sum_{k\ge 1}
		\int_0^t  \beta_k(v) \dW\right) \vphi \dx
		\\ & = \sum_{k\ge 1}\int_0^t \int_{\Om} \beta_k(v) \vphi \dx \dW_k,
	\end{split}
\end{equation}
knowing that the series converges in $L^2\left(D,\cF,P;C([0,T])\right)$, where 
$\beta_k(v):=\beta(v)\psi_k$ are real-valued functions. 
Sometimes we denote the right-hand side 
by $\int_0^t  \int_{\Om}  \beta(v) \vphi \dx \dW^v$. 

We need to impose conditions on the noise 
amplitudes $\beta=\eta,\sigma$. For each $v\in L^2(\Om)$, 
we assume that $\beta(v):\U\to L^2(\Om)$ is defined by 
$$
\beta(v)\psi_k=\beta_k(v(\cdot)),  \quad k\ge 1,
$$
for some real-valued functions $\beta_k(\cdot):\R\to \R$ that satisfy 
\begin{equation}\label{eq:noise-cond}
	\begin{split}
		& \sum_{k\ge 1} \abs{\beta_k(v)}^2 \le 
		C_\beta \left(1+ \abs{v}^2\right), \qquad \forall v\in \R,
		\\ 
		&\sum_{k\ge 1} \abs{\beta_k(v_1)-\beta_k(v_2)}^2 
		\le C_\beta\abs{v_1-v_2}^2, \qquad \forall v_1, v_2\in \R,
	\end{split}
\end{equation}
for a constant $C_\beta>0$.  
Thus, $\beta$ becomes a mapping from 
$L^2(\Om)$ to $L_2\left(\U,L^2(\Om)\right)$ and
satisfies, via \eqref{def:HS-norm},
$$
\norm{\beta(v)}_{L_2\left(\U,L^2(\Om)\right)}^2
= \sum_{k\ge 1} \norm{\beta_k(v)}_{L^2(\Om)}^2
=\int_{\Om} \sum_{k\ge 1} \abs{\beta_k(v)}^2 \dx<\infty,
$$ 
and similarly for $\norm{\beta(v_1)-\beta(v_2)}_{L_2\left(\U,L^2(\Om)\right)}^2$. 
To be sure, by way of \eqref{eq:noise-cond}, we have 
\begin{equation}\label{eq:noise-cond2}
	\begin{split}
		&\norm{\beta(v)}_{L_2\left(\U,L^2(\Om)\right)}^2
		\le C_\beta \left(1 + \norm{v}_{L^2(\Om)}^2\right), 
		\quad v\in L^2(\Om),
		\\ & 
		\norm{\beta(v_1)-\beta(v_2)}_{L_2\left(\U,L^2(\Om)\right)}^2
		\le C_\beta \norm{v_1 - v_2}_{L^2(\Om)}^2, \quad
		v_1, v_2 \in L^2(\Om).
	\end{split}
\end{equation}

Let $(\beta,W)=(\eta,W^v)$ or $(\sigma,W^w)$.
Given a predictable process 
$$
v\in L^2\left(D,\cF,P;L^2((0,T);L^2(\Om))\right),
$$
the stochastic integral $\int_0^t  \beta(v) \dW^w$ 
is well-defined, taking values in $L^2(\Om)$. 
By \eqref{eq:noise-cond2},
\begin{align*}
	& \E\left[\abs{\int_{\Om} \left(\, \int_0^t \beta(v)  \dW\right) \vphi \dx}^2\right]
	\le \E\left[\norm{ \int_0^t \beta(v)  \dW}_{L^2(\Om)}^2\right]
	\norm{\vphi}_{L^2(\Om)}^2
	\\ & \quad \le C_\vphi\, \E\left[\int_0^T 
	\norm{\beta(v)}_{L_2(\U,L^2(\Om))}^2 \dt\right]
	\le C_{\vphi,\beta} \, \E\left[\int_0^T 
	\left(1+\norm{v}_{L^2(\Om)}^2\right) \dt \right]
	\\ & \quad 
	\le C_{\vphi,\beta,T} \left(1+ 
	\norm{v}_{L^2\left(D,\cF,P;L^2((0,T);L^2(\Om))\right)}^2\right)<\infty,
	\quad \vphi\in L^2(\Om),
\end{align*} 
where we have also used the 
Cauchy-Schwarz and Burkholder-Davis-Gundy inequalities. 
Hence, \eqref{eq:int-W} makes sense.

\begin{rem}
The condition \eqref{eq:noise-cond} on the noise amplitude allows 
for various additive and multiplicative noises, see 
e.g.~\cite[Example 3.2]{{Glatt-Holtz:2008aa}} for a list of representative examples.

It is possible to allow $\beta=\eta,\sigma$ to be time dependent, $\beta=\beta(t,v)$. 
In this case, $\beta(t):L^2(\Om)\to L_2\left(\U,L^2(\Om)\right)$ 
must satisfy \eqref{eq:noise-cond} for a.e.~$t \in [0,T]$, 
with a constant $C_\beta$ that is independent of $t$.
This does not entail additional effort in the proofs, but for 
simplicity of presentation we suppress the time 
dependency throughout the paper. 
\end{rem}

We conclude this section by collecting a few relevant probability tools. 
Let $\bS$ be a separable Banach (or Polish) space. We denote by $\cB(\bS)$ the family 
of the Borel subsets of $\bS$ and by $\cP(\bS)$ the family of all 
Borel probability measures on $\bS$.  Each random variable $X:D\to \bS$ induces 
a probability measure on $\bS$ via the pushforward $X_\# P:=P\circ X^{-1}$. 
A sequence of probability measures $\Set{\mu_n}_{n\ge1}$ on $(\bS, \cB(\bS))$ is 
tight if for every $\epsilon>0$ there is a compact 
set $\bK_\epsilon\subset \bS$ such that $\mu_n(\bK_\epsilon)>1-\epsilon$ for all $n\ge 1$.

We will construct weak martingale solutions by 
applying the stochastic compactness method to a sequence of approximate solutions. 
In one step of the argument, we show tightness
of the probability laws of the approximations. By the Prokhorov 
theorem, this is equivalent to exhibiting weak compactness  of the laws. 
A sequence $\Set{\mu_n}_{n\ge} \subset \cP(\bS)$ is weakly (narrowly) 
convergent to $\mu\in \cP(\bS)$ as $n\to \infty$ if
$$
\lim_{n\to \infty} \int_{\bS} f \, d\mu_n  =  \int_{\bS} f \, d\mu ,
$$
for every bounded continuous functional $f:\bS\to \R$.

\begin{thm}[Prokhorov]\label{thm:prokhorov}
A sequence $\Set{\mu_n}_{n\ge1} \subset \cP(\bS)$ 
of probability measures is tight if and only 
if $\Set{\mu_n}_{n\ge}$ is relatively weakly compact.
\end{thm}

Relating to convergence of the approximate solutions, it is essential that we 
secure strong compactness (a.s.~convergence) in the $\omega$ variable. 
To that end, we are in need of a Skorokhod a.s.~representation 
theorem \cite{Ikeda:1981aa}, delivering a 
new probability space and new random variables, with the same 
laws as the original ones, converging almost surely. 

As alluded to before, our path space is not a Polish space since weak 
topologies in Hilbert and Banach spaces are not metrizable. 
Therefore the original Skorokhod theorem is not applicable; 
instead we must rely on the more recent Jakubowski version \cite{Jakubowski:1997aa} 
that applies to so-called quasi-Polish spaces. ``Quasi-Polish" refers to 
spaces $\bS$ for which there exists a countable family  
\begin{equation}\label{eq:countable}
	\Set{f_\ell:\bS\to [-1,1]}_{\ell\in L}
\end{equation}
of continuous functionals that separate points (of $\bS$) \cite{Jakubowski:1997aa}. 
Quasi-Polish spaces include separable Banach spaces 
equipped with the weak topology, and also spaces of weakly 
continuous functions taking values in some separable Banach space. 
The basic assumption \eqref{eq:countable} gives raise to a mapping 
between $\bS$ and the Polish space $[-1,1]^L$,
\begin{equation}\label{eq:tildef}
	\bS\ni u \mapsto \tilde f(u) 
	= \Set{f_\ell(u)}_{\ell\in L} \in [-1,1]^L,
\end{equation}
which is one-to-one and continuous, but in general $\tilde f$ is not a homeomorphism of $\bS$ 
onto a subspace of $\bS$. However, if we restrict to a $\sigma$-compact subspace of $\bS$, 
then $\tilde f$ becomes a measurable isomorphism \cite{Jakubowski:1997aa}. 
In this paper we use the following form of the 
Skorokhod-Jakubowski theorem \cite{Jakubowski:1997aa}, taken 
from \cite{Brzezniak:2013aa,Ondrejat:2010aa} (see 
also \cite{Brzezniak:2013ab,Brzezniak:2011aa}).

\begin{thm}[Skorokhod-Jakubowski a.s.~representations 
for subsequences]\label{thm:skorokhod}
Let $\bS$ be a topological space for which there exists 
a sequence $\Set{f_\ell}_{\ell\ge 1}$ of continuous functionals
$f_\ell: \bS \to  \R$ that separate points of $\bS$. 
Denote by $\Sigma$ the $\sigma$-algebra 
generated by the maps $\Set{f_\ell}_{\ell\ge 1}$. Then

(1) every compact subset of $\bS$ is metrizable;

(2) every Borel subset of a $\sigma$-compact set in $\bS$ 
belongs to $\Sigma$;

(3) every probability measure supported by a $\sigma$-compact set in $\bS$ 
has a unique Radon extension to the Borel $\sigma$-algebra $\cB(\bS)$;

(4) if $\Set{\mu_n}_{n\ge 1}$ is a tight sequence of probability 
measures on $(\bS,\Sigma)$, then there exist a subsequence 
$\Set{n_k}_{k\ge 1}$, a probability space $(\tilde D,\tilde \cF,\tilde P)$, and  
Borel measurable $\bS$-valued random variables $\tilde X_k$, $\tilde X$, 
such that $\mu_{n_k}$ is the law of $\tilde X_k$ 
and $X_k\to X$ $\tilde P$-a.s.~(in $\bS$). 
Moreover, the law $\mu$ of $\tilde X$ is a Radon measure.
\end{thm}

We will need the Gy\"ongy-Krylov characterization of convergence 
in probability \cite{Gyong-Krylov}. It will be used to upgrade weak martingale solutions 
to strong (pathwise) solutions, via a pathwise uniqueness result.

\begin{lem}[Gy\"{o}ngy-Krylov characterization]\label{lem:krylov}
Let $\bS$ be a Polish space, and let $\{X_n\}_{n\ge 1}$ 
be a sequence of $\bS$-valued random variables on a 
probability space $(D,\cF, P)$. For each $n,m\ge 1$, denote by 
$\mu_{n,m}$ the joint law of $\left(X_n,X_m\right)$, that is,
$$
\mu_{n,m}(A) 
:= P\left(\Set{\omega\in D: 
\left(X_n(\omega), X_m(\omega)\right )\in A}\right), 
\quad A\in \B(\bS\times \bS).
$$
Then $\{X_n\}_{n\ge 1}$ converges in 
probability (and $P$-a.s.~along a subsequence) 
$\Longleftrightarrow$ for any subsequence $\{\mu_{m_k,n_k}\}_{k\ge 1}$ there 
exists a further subsequence that converges weakly to some $\mu \in \cP(\bS)$ 
that is supported on the diagonal: 
$$
\mu\left(\Set{(X,Y)\in \bS \times \bS: X=Y}\right)=1.
$$
\end{lem}

\begin{rem}\label{rem:GyKr}
As a matter of fact, we need access to the ``$\Longleftarrow$" part 
of the Gy\"{o}ngy-Krylov lemma for quasi-Polish spaces $\bS$. Suppose for any 
subsequence $\Set{\left(X_{n_k}, X_{m_k}\right)}_{k\ge 1}$ there 
exists a further subsequence 
$$
\Set{\left(X_{n_{k_j}}, X_{m_{k_j}}\right)}_{j\ge 1}
$$ 
that converges in distribution to $(X,X)$ as $j\to \infty$, for some $X\in \bS$, that is, 
the joint probability laws $\mu_{m_{k_j},n_{k_j}}$ converge weakly 
to some $\mu\in \cP(\bS\times \bS)$ that is supported on the diagonal. 
Recalling the mapping $\tilde f$ between $\bS$ and the Polish 
space $[-1,1]^L$, cf.~\eqref{eq:tildef}, and the 
continuous mapping theorem, it follows that the sequence 
$$
\Set{\left(\tilde f(X_{n_{k_j}}), \tilde f(X_{m_{k_j}})\right)}_{j\ge 1}
$$ 
converges in distribution to $\left(f(X),f(X)\right)$ as $j\to \infty$. 
In view of the Gy\"{o}ngy-Krylov lemma, this implies that the 
sequence $\Set{\tilde f(X_n)}_{n\ge 1}$ converges in 
probability and thus, along a subsequence 
$\Set{\tilde f(X_{n_j})}_{j \ge 1}$, $P$-almost surely. 
Since $\Set{f_\ell}_{\ell\ge 1}$ separate points of $\bS$, it is not difficult to see that 
this implies that $\Set{X_{n_j}}_{j \ge 1}$ converges $P$-a.s.~as well.
\end{rem}

We are going to need the following convergence result  for stochastic 
integrals due to Debussche, Glatt-Holtz, and Temam \cite{Debussche:2011aa}.

\begin{lem}[convergence of stochastic integrals]\label{lem:stoch-conv}
Fix a probability space $(D,\cF,P)$ and 
a separable Hilbert space $\bX$.  For each $n=1,2,\ldots$, 
consider a stochastic basis 
$$
\cS_n=\left(D,\cF,\Set{\cFt^n}_{t\in [0,T]},P,W^n\right)
$$
and a $\Set{\cFt^n}_{t\in [0,T]}$-predictable 
$\bX$-valued process $G^n$ satisfying 
$$
G^n\in L^2((0,T);L_2(\U,\bX)), \quad \text{$P$-almost surely}.
$$ 
Suppose there exist a stochastic basis 
$\cS=\left(D,\cF,\Set{\cFt}_{t\in [0,T]},P,W\right)$ 
and a $\Set{\cFt}_{t\in [0,T]}$-predictable $\bX$-valued process 
$G$ with $G\in L^2((0,T);L_2(\U,\bX))$ $P$-a.s., such that
\begin{equation*}
	\begin{split}
		& W^n \to W \quad \text{in $C([0,T]; \U_0)$, in probability}
		\\ &
		G^n \to G \quad \text{in $L^2((0,T); L_2(\U;\bX))$, in probability}.
	\end{split}
\end{equation*} 
Then
$$
\int_0^t G^n \, dW^n \to \int_0^t G \, dW 
\quad \text{in $L^2((0,T);\bX)$, in probability}.
$$
\end{lem}

Finally, we recall the ``Kolmogorov test" for 
the existence of continuous modifications of real-valued 
stochastic processes.

\begin{thm}[Kolmogorov's continuity theorem]\label{thm:kolm-cont}
Let $X=\Set{X(t)}_{t\in [0,T]}$ be a real-valued stochastic process 
defined on a probability space $(D, \cF, P)$. Suppose 
there are constants $\kappa>1$, $\delta>0$, and $C>0$ such 
that for all $s,t \in [0,T]$,
$$
\E\left[ \abs{X(t) - X(s)}^{\kappa} \right] 
\leq C \abs{t-s}^{1+\delta}.
$$
Then there exists a continuous modification of $X$. The paths of $X$ 
are $\gamma$-H\"older continuous for every 
$\gamma \in \left[0, \frac{\delta}{\kappa}\right)$.
\end{thm}

\section{Notion of solution and main results}\label{sec:defsol}
Depending on the (probabilistic) notion of solution, the initial data \eqref{S-init} 
are imposed differently. For pathwise (probabilistic strong) solutions, we 
prescribe the initial data as random variables 
$v_0,w_0\in L^2(D,\cF,P;L^2(\Om))$.
For martingale (or probabilistic weak) solutions, of which 
the stochastic basis is an unknown component, we prescribe the initial data 
in terms of probability measures $\mu_{v_0},\mu_{w_0}$ on $L^2(\Om)$. 
The measures $\mu_{v_0}$ and $\mu_{w_0}$ should be viewed as 
``initial laws" in the sense that the laws of $v(0)$, $w(0)$ 
are required to coincide with $\mu_{v_0},\mu_{w_0}$, respectively. 

Sometimes we need to assume the existence of 
a number $q_0>\frac92$ such that 
\begin{equation}\label{eq:moment-est}
	\int_{L^2(\Om)} \norm{v}^{q_0}_{L^2(\Om)}\, d\mu_{v_0}(v)<\infty, 
	\quad
	\int_{L^2(\Om)} \norm{w}^{q_0}_{L^2(\Om)}\, d\mu_{w_0}(w)<\infty.
\end{equation}
As a matter of fact, we mostly need \eqref{eq:moment-est} with 
$q_0>2$. One exception occurs in Subsection \ref{subsec:conclude}, where 
we use $q_0>\frac92$ to conclude that the transmembrane 
potential $v$ is a.s.~weakly time continuous, 
cf.~part \eqref{eq:mart-adapt} in the definition below (for 
$w$ this holds with just $q_0>2$).

Let us define precisely what is meant by a solution to 
the stochastic bidomain model. For this, we use the space
$$
\tH(\Om) := 
\,
\text{closure of the set}
\,
\Set{v \in C^\infty(\R^3),\, v \big|_{\Sigma_D}= 0}
\, 
\text{in the $H^1(\Om)$ norm}.
$$
We denote by $(\tH(\Om))^*$ the dual of $\tH(\Om)$, which is 
equipped with the norm
\begin{equation}\label{eq:def-dual-norm}
	\norm{u^*}_{(\tH(\Om))^*}=
	\underset{\norm{\phi}_{\tH(\Om)}\le 1}{\sup_{\phi\in \tH(\Om)}}
	\left \langle u^*,\phi\right\rangle_{(\tH(\Om))^*,\tH(\Om)}.
\end{equation}

\begin{defi}[weak martingale solution] \label{def:martingale-sol}
Let $\mu_{v_0}$ and $\mu_{w_0}$ be probability measures 
on $L^2(\Om)$. A weak martingale solution 
of the stochastic bidomain system \eqref{S1}, with initial-boundary 
data \eqref{S-init}-\eqref{S-bc}, is a collection 
$$
\bigl(\cS,u_i,u_e,v,w\bigr)
$$ 
satisfying the following conditions:
\begin{enumerate}
	\item\label{eq:mart-sto-basis}
	$\cS=\left(D,\cF,\Set{\cFt}_{t\in [0,T]},P,\Set{W_k^v}_{k=1}^\infty,
	\Set{W_k^w}_{k=1}^\infty\right)$ is a stochastic basis;

	\item \label{eq:mart-wiener}
	$W^v:=\sum_{k\ge 1} W^v_k e_k$ and 
	$W^w:=\sum_{k\ge 1} W^w_k e_k$ are 
	two independent \\ cylindrical Brownian motions, adapted 
	to the filtration $\Set{\cFt}_{t\in [0,T]}$;
	
	\item\label{eq:mart-uiue-reg}
	For $P$-a.e.~$\omega\in D$,
	$u_i(\omega),u_e(\omega)\in L^2((0,T);\tH(\Om))$;
	 
	\item\label{eq:mart-vreg}
	For $P$-a.e.~$\omega\in D$, 
	$v(\omega)\in L^2((0,T);\tH(\Om))\cap L^4(\Om_T)$. Moreover, $v=u_i-u_e$;
	
	\item\label{eq:mart-adapt}
	$v,w:D\times [0,T]\to L^2(\Om)$ are 
	$\Set{\cFt}_{t\in [0,T]}$-adapted processes, 
	$\Set{\cFt}_{t\in [0,T]}$-predictable in $(\tH(\Om))^*$, such 
	that for $P$-a.e.~$\omega\in D$,
	$$
	v(\omega),w(\omega)\in 
	L^\infty((0,T);L^2(\Om))\cap C([0,T];(\tH(\Om))^*);
	$$
	
	\item\label{eq:mart-data}
	The laws of $v_0:=v(0)$ and $w_0:=w(0)$ are 
	respectively $\mu_{v_0}$ and $\mu_{w_0}$:
	$$
	P\circ v_0^{-1}=\mu_{v_0}, \qquad
	P\circ w_0^{-1}=\mu_{w_0};
	$$
	
	\item\label{eq:mart-weakform} 
	The following identities hold $P$-almost surely, for any $t \in [0,T]$: 
	\begin{equation}\label{eq:weakform}
		\begin{split}
			& \int_{\Om} v(t) \vphi_i \dx 
			+ \int_0^t \int_{\Om} 
			\Bigl( M_i\Grad u_i \cdot \Grad \vphi_i + I(v,w) \vphi_i \Bigr) \dx\ds
			\\ & \qquad\qquad \qquad \qquad 
			= \int_{\Om} v_0 \, \vphi_i \dx 
			+  \int_0^t \int_{\Om} \eta(v) \vphi_i \dx  \dW^v (s), 
			\\ & 
			\int_{\Om} v(t)  \vphi_e \dx
			+ \int_0^t \int_{\Om} 
			\Bigl (- M_e\Grad u_e\cdot \Grad \vphi_e
			+ I(v,w) \vphi_e \Bigr) \dx\ds 
			\\ & \qquad\qquad \qquad \qquad
			=\int_{\Om} v_0 \vphi_e \dx
			 + \int_0^t \int_{\Om} \eta(v)  \vphi_e \dx \dW^v(s),
			 \\ &
			 \int_{\Om} w(t) \vphi\dx =\int_{\Om} w_0 \vphi \dx
			 +\int_0^t \int_{\Om} H(v,w)\vphi \dx \ds
			 \\ & \qquad\qquad \qquad \qquad\qquad \qquad
			 + \int_0^t \int_{\Om}  \sigma(v) \vphi \dx \dW^w(s),
		\end{split}
	\end{equation}
	for all $\vphi_i,\vphi_e \in \tH(\Om)$ and $\vphi\in L^2(\Om)$. 
\end{enumerate} 
 \end{defi} 
 
 \begin{rem}
In view of the regularity conditions imposed in Definition \ref{def:martingale-sol}, it 
is easily verified that the deterministic integrals in \eqref{eq:weakform} are well-defined. 
The stochastic integrals are well-defined as well; they have been 
given special attention in Section \ref{sec:stoch}, see \eqref{eq:int-W}.
\end{rem}

\begin{rem}\label{rem:weak-L2-cont}
We denote by $C\left([0,T];L^2(\Om)\mathrm{-weak}\right)$
the space of weakly continuous $L^2(\Om)$ functions . 
According to \cite[Lemma 1.4]{Temam:1977aa}, part \eqref{eq:mart-adapt} 
of Definition \ref{def:martingale-sol} implies that 
$$
v(\omega,\cdot,\cdot),w(\omega,\cdot,\cdot)
\in C\left([0,T];L^2(\Om)\mathrm{-weak}\right),
\quad \text{for $P$-a.e.~$\omega\in D$}.
$$ 
\end{rem}

Our main existence result is contained in 
\begin{thm}[existence of weak martingale solution]\label{thm:martingale}
Suppose conditions {\rm (\textbf{GFHN})}, \eqref{matrix}, 
and \eqref{eq:noise-cond} hold. Let $\mu_{v_0}$, $\mu_{w_0}$ 
be probability measures satisfying the moment 
estimates \eqref{eq:moment-est} (with $v_0\sim\mu_{v_0}$, $w_0\sim \mu_{w_0}$).
Then the stochastic bidomain model \eqref{S1}, \eqref{S-init}, \eqref{S-bc} 
possesses a weak martingale solution in the 
sense of Definition \ref{def:martingale-sol}.
\end{thm}

The proof of Theorem \ref{thm:martingale} is divided into a series of steps. 
We construct approximate solutions in Section \ref{sec:approx-sol}, which are 
shown to converge in Section \ref{sec:convergence}.
The convergence proof relies on several uniform a priori estimates 
that are established in Subsections \ref{sec:apriori} and \ref{sec:translation-est}. 
We use these estimates in Subsection \ref{subsec:tight} to conclude that the laws of 
the approximate solutions are tight and that the approximations 
(along a subsequence) converge to a limit. The limit is shown to be 
a weak martingale solution in 
Subsections \ref{subsec:limit} and  \ref{subsec:conclude}.

If the stochastic basis $\cS$ in Definition \ref{def:martingale-sol} is 
fixed in advance (not part of the solution), we speak of a weak solution or 
weak pathwise solution. A weak solution is thus weak in the PDE sense and 
strong in the probabilistic sense. In this case, we prescribe the 
initial data $v_0,w_0$ as random variables relative to $\cS$.

\begin{defi}[weak solution]\label{def:pathwise-sol} 
Let 
$$
\cS=\left(D,\cF,\Set{\cFt}_{t\in [0,T]},P,\Set{W_k}_{k=1}^\infty\right) 
$$
be a fixed stochastic basis and assume that
$$
\text{$v_0,w_0$ are $\cF_0$-measurable, 
with $v_0,w_0 \in L^2(D,\cF,P;L^2(\Om))$.}
$$
A weak solution of the stochastic bidomain 
system \eqref{S1}, with initial-boundary 
data \eqref{S-init}-\eqref{S-bc}, is a collection $U=\bigl(u_i,u_e,v,w\bigr)$ 
satisfying conditions \eqref{eq:mart-uiue-reg}, \eqref{eq:mart-vreg}, 
\eqref{eq:mart-adapt}, and \eqref{eq:mart-weakform} 
in Definition \ref{def:martingale-sol} (relative to $\cS$).
\end{defi}
 
Weak solutions  are said to be unique if, given 
any pair of such solutions $\hat U, \tilde U$ for which 
$\hat U$ and $\tilde U$ coincide a.s.~at $t = 0$,
\begin{equation}\label{eq:path-uniq}
	P\Bigl(\Set{\hat U(t)=\tilde U(t) \, \forall t\in[0,T]} \Bigr)=1.
\end{equation}
We establish pathwise uniqueness by demonstrating 
that $v(t),w(t)$ depend continuously on the 
initial data $v_0,w_0$ in $L^2(D,\cF,P;L^2(\Om))$. 
Moreover, using the Poincar\'e inequality, we conclude as well the pathwise 
uniqueness of $u_i,u_e$.

As alluded to earlier, we use this to ``upgrade" martingale solutions 
to weak (pathwise) solutions, thereby delivering

\begin{thm}[existence and uniqueness of weak solution]\label{thm:pathwise}
Suppose conditions {\rm (\textbf{GFHN})}, \eqref{matrix}, 
and \eqref{eq:noise-cond} hold. Then the stochastic 
bidomain model \eqref{S1}, \eqref{S-init}, \eqref{S-bc} possesses 
a unique weak solution in the sense 
of Definition \ref{def:pathwise-sol}, provided 
the initial data satisfy 
$$
v_0,w_0\in L^{q_0}(D,\cF,P;L^2(\Om)), 
\qquad q_0>9/2.
$$
\end{thm}

Regarding the proof of Theorem \ref{thm:pathwise}, we divide it into two steps. 
A pathwise uniqueness result is established in Section \ref{sec:uniq} by exhibiting 
an $L^2$ stability estimate for the difference between  two solutions. 
We use this result in Section \ref{sec:pathwise} to upgrade martingale 
solutions to pathwise solutions.

\section{Construction of approximate solutions}\label{sec:approx-sol}
In this section we define the Faedo-Galerkin approximations. 
They are based on a \textit{nondegenerate} system introduced below.
In upcoming sections we use these approximations to construct 
weak martingale solutions to the stochastic bidomain model. 

We begin by fixing a stochastic basis 
\begin{equation}\label{eq:fixedS}
	\cS=\left( D, \cF, \Set{\cFt}_{t\in [0,T]}, P, 
	\Set{W_k^v}_{k=1}^\infty, \Set{W_k^w}_{k=1}^\infty\right),
\end{equation}
and $\cF_0$-measurable initial data $v_0,w_0 \in L^2(D;L^2(\Om))$ 
with respective laws $\mu_{v_0},\mu_{w_0}$ on $L^2(\Om)$.
For each fixed $\eps>0$, the nondegenerate system reads
\begin{align}\label{eq:nondegen}
	\begin{split}
		&d v + \eps d u_i -\Div\bigl(M_i \Grad u_i \bigr) \dt
		+ I(v,w)\dt = \eta(v) \, d W^v \quad \text{in $\Om_T$},\\
		& d v - \eps du_e +\Div\bigl(M_e\Grad u_e\bigr) \dt
		+ I(v,w)\dt  =  \eta(v)\, d W^v \quad \text{in $\Om_T$},\\
		& d w=H(v,w)\dt +\sigma(v) d W^w \quad  \text{in $\Om_T$}, 
	\end{split}
\end{align}
with boundary conditions \eqref{S-bc}. 
Regarding \eqref{eq:nondegen}, we must provide initial data 
for $u_i$, $u_e$ (not $v=u_i-u_e$ as in the original problem). 
For that reason, we decompose (arbitrarily) the initial 
condition $v_0$ in \eqref{S-init} as $v_0 = u_{i,0}-u_{e,0}$, for some 
$\cF_0$-measurable random variables $u_{i,0}$ and $u_{e,0}$, 
\begin{equation}\label{eq:uie-init-ass}
	u_{i,0},u_{e,0}\in L^2\left(D,\cF,P;L^2(\Om)\right), 
\end{equation}
such that the law of $u_{i,0}-u_{e,0}$ coincides with $\mu_{v_0}$. 
We replace \eqref{S-init} by
\begin{equation}\label{S-init-new}
	u_j(0,x)=u_{j,0}(x) \quad (j=i,e),
	\qquad 
	w(0,x)=w_0(x), 
	\qquad 
	x \in \Om.
\end{equation}
In some situations, we make use of the strengthened assumption 
\begin{equation}\label{eq:uie-init-ass-q0}
	u_{i,0},u_{e,0},w_0 \in L^{q_0}\left(D,\cF,P;L^2(\Om)\right),
	\quad \text{with $q_0$ defined in \eqref{eq:moment-est}.}
\end{equation}

\begin{rem}
Modulo some obvious changes, the definitions of weak martingale 
and weak (pathwise) solutions to the nondegenerate 
system \eqref{eq:nondegen}-\eqref{S-init-new}-\eqref{S-bc} 
are basically the same as those for the original system.
\end{rem}

To construct and justify the validity of the Faedo-Galerkin approximations, we employ
a classical Hilbert basis, which is orthonormal in $L^2$ and orthogonal in $H^1_D$. 
We refer for example to \cite[Thm.~7.7, p.~87]{Showalter} (see  also \cite{Renardy}) for 
the standard construction of such bases.  We operate with the same basis 
$\Set{e_l}_{l=1}^n$ for all the unknowns $u_i,u_e,v,w$.

We look for a solution to the problem arising 
as the projection of \eqref{eq:nondegen}, \eqref{S-init}, \eqref{S-bc} 
onto the finite dimensional subspace $\bX_n:=\Span\Set{e_l}_{l=1}^n$. 
The (finite dimensional) approximate solutions take the form
\begin{equation}\label{eq:approxsol}
	\begin{split}
		&u_j^n:[0,T]\to \bX_n, \quad
		 u_j^n(t)=\sum_{l=1}^n c_{j,l}^n(t)e_l
		\quad (j=i,e),
		\\ &
		v^n:[0,T]\to \bX_n, \quad
		v^n(t)=\sum_{l=1}^n c_l^n (t)e_l, 
		\quad c_l^n(t)=c_{i,l}^n(t)-c_{e,l}^n(t),
		\\ &
		w^n:[0,T]\to \bX_n, \quad
		w^n(t)=\sum_{l=1}^n a_l^n(t)e_{l}.
	\end{split}
\end{equation}
We pick the coefficients 
\begin{equation}\label{eq:SDE-coeff}
	c_j^n=\Set{c_{j,l}^n}_{l=1}^n \,\, (j=i,e), 
	\quad
	a^n = \Set{a_l^n}_{l=1}^n,
\end{equation}
which are finite dimensional stochastic processes 
relative to \eqref{eq:fixedS}, such 
that ($\ell=1,\ldots, n$)
\begin{equation}
	\label{S1-Galerkin-new}
	\begin{split}
		& \left(dv^n, e_\ell\right)_{L^2(\Om)}
		+\eps_n \left( du_i^n ,e_\ell\right)_{L^2(\Om)}		
		\\ & \qquad
		+ \left( M_i \Grad u_i^n, \Grad e_\ell\right)_{L^2(\Om)} \dt
		+\left(I (v^n,w^n), e_\ell \right)_{L^2(\Om)}\dt
		\\ & \qquad \qquad 
		= \sum_{k=1}^n\left( \eta^n_k(v^n), e_\ell \right)_{L^2(\Om)} 
		\dW^v_k(t), \\
		&  \left(dv^n, e_\ell\right)_{L^2(\Om)}
		- \eps_n \left(du_e^n, e_\ell\right)_{L^2(\Om)}
		\\ & \qquad
		-\left(M_e\Grad u_e^n, \Grad e_\ell\right)_{L^2(\Om)}\dt
		+\left(I(v^n,w^n), e_\ell \right)_{L^2(\Om)} \dt
		\\ &\qquad  \qquad 
		=  \sum_{k=1}^n\left( \eta^n_k(v^n), e_\ell \right)_{L^2(\Om)} 
		\dW^v_k(t), \\
		& \left(dw^n,e_\ell\right)_{L^2(\Om)} 
		= \left( H(v^n,w^n), e_\ell\right)_{L^2(\Om)} \dt
		\\ & \qquad\qquad\qquad\qquad
		+ \sum_{k=1}^n \left (\sigma^n_k(v^n), e_\ell \right)_{L^2(\Om)} 
		\dW^w_k(t),
	\end{split}	
\end{equation}
where $\eps$ in \eqref{eq:nondegen} is taken as
\begin{equation}\label{eq:def-eps}
	\eps=\eps_n:=\frac{1}{n}, \qquad n\ge 1.
\end{equation}

We need to comment on the finite dimensional approximations 
of the stochastic terms utilized in \eqref{S1-Galerkin-new}.
With $(\beta,W)$ denoting $(\eta,W^v)$ 
or $(\sigma,W^w)$, recall that $\beta$ maps 
from $L^2 \left((0,T);L^2(\Om)\right)$ to 
$L^2\left((0,T);L_2\left(\U,L^2(\Om)\right)\right)$,
where $\U$ is equipped with the orthonormal basis 
$\Set{\psi_k}_{k\ge 1}$ (cf.~Section \ref{sec:stoch}). 
Employing the decomposition 
$$
\beta_k(v)=\beta(v)\psi_k,
\qquad
\beta_k(v) 
=\sum_{l\ge 1} \left( \beta_k(v),e_l\right)_{L^2(\Om)}e_l,
$$
we can write
\begin{align*}
	\beta(v) \dW & = \sum_{k\ge 1} \beta_k(v) \dW_k
	\\ &
	= \sum_{k,l\ge 1} \beta_{k,l}(v) e_l \dW_k, 
	\quad
	\beta_{k,l}(v)= \left( \beta_k(v), e_l\right)_{L^2(\Om)}.
\end{align*}
In \eqref{S1-Galerkin-new}, we utilize the 
finite dimensional approximation
\begin{equation}\label{eq:Wn}
	\beta^n(v) \dW^n := \sum_{k,l=1}^n \beta_{k,l}(v) e_l \dW_k
	=\sum_{k=1}^n \beta_k^n(v)\dW_k,
\end{equation}
with $\beta^n$ and $W^n$ then defined by
$$
\beta_k^n(v)=\beta^n(v)\psi_k,
\quad
\beta_k^n(v) = \sum_{l=1}^n \beta_{k,l}(v) e_l, 
\quad 
W^n=\sum_{k=1}^n W_k \psi_k,
$$
where $(\beta^n,W^n)$ denotes $(\eta^n,W^{v,n})$ or $(\sigma^n,W^{w,n})$; 
$W^n$ converges in $C([0,T];\U_0)$ 
for $P$-a.e.~$\omega\in D$ and (by a martingale 
inequality) in $L^2\left(D,\cF,P; C([0,T];\U_0)\right)$.

The initial conditions are 
\begin{equation}\label{Galerkin-initdata}
	\begin{split}
		& u_j^n(0) = u_{j,0}^n :=\sum_{l=1}^n c_{j,l}^n(0)e_l, 
		\quad c_{j,l}^n(0) := \left(u_{j,0}^n,e_l\right)_{L^2(\Om)}, \quad j=i,e,\\
		& v^n(0)=v_0^n :=u_{i,0}^n-u_{e,0}^n, \\
		& w^n(0) = w_0^n :=\sum_{l=1}^n a_l^n(0) e_l, \quad
		a_l^n(0) := \left(w_{0},e_l\right)_{L^2(\Om)}.
	\end{split}
\end{equation}
In \eqref{Galerkin-initdata}, consider for example 
$u_{j,0}^n$. Since $u_{j,0}\in L^2\left(D,\cF,P;L^2(\Om)\right)$, we have 
(by standard properties of finite-dimensional 
projections, cf.~\eqref{eq:proj-prop0}, \eqref{eq:proj-prop2} below)
\begin{equation*}
	\text{$u_{j,0}^n \to u_{j,0}$ in $L^2(\Om)$, $P$-a.s., as $n\to \infty$,}
\end{equation*}
and
$$
\norm{u_{j,0}^n}_{L^2(\Om)}^2 \le C \norm{u_{j,0}}_{L^2(\Om)}^2.
$$
On this account, the dominated convergence theorem implies
\begin{equation}\label{eq:L2-conv-init}
	\text{$u_{j,0}^n \to u_{j,0}$ in $L^2\left(D,\cF,P;L^2(\Om)\right)$, as $n\to \infty$.}
\end{equation}
Similarly, $w_0^n\to w_0,v_0^n\to v_0$ 
in $L^2(\Om)$, $P$-a.s., and thus in $L^2_\omega(L^2_x)$.

For the basis $\Set{e_l}_{l=1}^\infty$, we introduce the projection 
operators (see e.g.~\cite[page 1636]{Brzezniak:2013aa})
\begin{equation}\label{eq:project-op}
	\begin{split}
		&\Pi_n:(\tH(\Om))^*\to \Span\Set{e_l}_{j=1}^{\infty},
		\\ & 
		\Pi_n u^* := \sum_{l=1}^n 
		\left \langle u^*,e_l \right\rangle_{(\tH(\Om))^*,\tH(\Om)}e_l.
	\end{split}
\end{equation}
The restriction of $\Pi_n$ to $L^2(\Om)$ is also denoted by $\Pi_n$:
\begin{align*}
	&\Pi_n: L^2(\Om)\to \Span\Set{e_l}_{j=1}^{\infty},
	\quad
	\Pi_n u := \sum_{l=1}^n 
	\left( u,e_l\right)_{L^2(\Om)} e_l,
\end{align*}
i.e., $\Pi_n$ is the orthogonal projection from $L^2(\Om)$
to  $\Span\Set{e_l}_{j=1}^{\infty}$. We have
\begin{equation}\label{eq:proj-prop0}
	\norm{\Pi_n u}_{L^2(\Om)}
	\le \norm{u}_{L^2(\Om)}, 
	\qquad u\in L^2(\Om).
\end{equation}
Note that we have the following equality for any
$u^*\in (\tH(\Om))^*$ and $u\in \tH(\Om)$:
\begin{equation}\begin{split}\label{eq:proj-prop1}
	&\left(\Pi_n u^*, u\right)_{L^2(\Om)}
	=\left \langle u^*, \Pi_n u
	\right\rangle_{(\tH(\Om))^*,\tH(\Om)}.
\end{split}
\end{equation}
Furthermore, as $n\to \infty$,
\begin{equation}\label{eq:proj-prop2}
	\norm{\Pi_n u -u}_{\tH(\Om)}\to 0, 
	\quad u\in \tH(\Om).
\end{equation}

Using the projection operator \eqref{eq:project-op}, we may 
write \eqref{S1-Galerkin-new} in integrated form 
equivalently as equalities between 
$(\tH(\Om))^*$ valued random variables:
\begin{equation}\label{eq:approx-eqn-integrated}
	\begin{split}
		& v^n(t) + \eps_n u_i^n(t) = v^n_0 + \eps_n u_{i,0}^n 
		+ \int_0^t \Pi_n\left[\Div\bigl(M_i \Grad u_i^n \bigr) 
		- I(v^n,w^n) \right]\ds 
		\\ & \qquad\qquad\qquad\qquad \qquad
		+ \int_0^t \eta^n(v^n) \dW^{v,n}(s)
		\quad \text{in $(\tH(\Om))^*$},
		\\ & v^n(t) - \eps_n u_e^n(t)= v^n_0 - \eps_n u_{e,0}^n
		+\int_0^t \Pi_n\left[-\Div\bigl(M_e \Grad u_e^n \bigr) 
		- I(v^n,w^n) \right]\ds 
		\\ & \qquad\qquad\qquad\qquad \qquad
		+ \int_0^t \eta^n(v^n) \dW^{v,n}(s)
		\quad \text{in $(\tH(\Om))^*$},
		\\ & w^n(t) = w^n_0 + \int_0^t \Pi_n \left(H(v^n,w^n)\right) \ds
		\\ & \qquad\qquad\qquad\qquad \qquad
		+ \int_0^t \sigma^n(v^n)\dW^{w,n}(s)
		\quad \text{in $(\tH(\Om))^*$},
	\end{split}
\end{equation}
where $v^n_0 = u_{i,0}^n-u_{e,0}^n$ and $u_{i,0}^n=\Pi_n u_{i,0}$, 
$u_{e,0}^n=\Pi_n u_{e,0}$, $w^n_0 = \Pi_n w_0$. 

In coming sections we investigate the 
convergence properties of the sequences 
$\Set{u_j^n}_{n\ge 1}$ ($j=i,e$), $\Set{v^n}_{n\ge 1}$, 
$\Set{w^n}_{n\ge 1}$ defined by \eqref{eq:approx-eqn-integrated}. 
Meanwhile, we must verify the existence of a (pathwise) solution to the 
finite dimensional system \eqref{S1-Galerkin-new}.

\begin{lem}\label{lem:fg-solutions}
For each fixed $n\ge1$, the Faedo-Galerkin equations 
\eqref{eq:approxsol}, \eqref{S1-Galerkin-new}, and \eqref{Galerkin-initdata} 
possess a unique global adapted solution $(u_i^n(t), u_e^n(t), v^n(t), w^n(t))$ on $[0,T]$.
Besides, $u_i^n, u_e^n, v^n, w^n$ belong to $C([0,T];\bX_n)$,
and $v^n=u_i^n-u_e^n$.
\end{lem}

\begin{proof}
Using the orthonormality of the basis, \eqref{S1-Galerkin-new} 
becomes the SDE system ($\ell=1,\ldots,n$)
\begin{equation}\label{S1-Galerkin-bis}
	\begin{split}
		d \left(c_\ell^n+\eps_n  c_{i,\ell}^n\right) 
		& = A_{i,\ell} \dt + \Gamma_\ell \dW^{v,n}, \\
		d \left( c_\ell^n-\eps_n c_{e,\ell}^n\right) 
		& =A_{e,\ell} \dt+\Gamma_\ell \dW^{v,n},\\
		d a_\ell^n & = A_{H,\ell} \dt
		+\zeta_\ell \dW^{w,n},
	\end{split}
\end{equation}
for the coefficients $c_j^n=c_j^n(t)$ $(j=i,e)$ and $a^n=a^n(t)$, 
cf.~\eqref{eq:SDE-coeff}, where
\begin{align*}
	& A_{i,\ell}=
	-\int_{\Om} M_i \Grad u_i^n\cdot \Grad e_\ell \dx
	-  \int_{\Om} I(v^n,w^n) e_\ell \dx,
	\\ 
	& A_{e,\ell}=
	\int_{\Om} M_e\Grad u_e^n\cdot \Grad e_\ell\dx
	-\int_{\Om} I(v^n,w^n)e_\ell \dx,
	\\
	& A_{H,\ell}=
	\int_{\Om}H(v^n,w^n) e_\ell \dx,
	\\ 
	& \Gamma_\ell = \Set{\Gamma_{\ell,k}}_{k=1}^n,
	\; \Gamma_{\ell,k}=\int_{\Om} \eta^n_k(v^n) \, e_\ell  \dx,
	\;  \Gamma_\ell \dW^{v,n} 
	=\sum_{k=1}^n \Gamma_{\ell,k}  \dW^v_k,
	\\ 
	& \zeta_\ell = \Set{\zeta_{\ell,k}}_{k=1}^n,
	\; \zeta_{\ell,k}=\int_{\Om} \sigma^n_k(v^n) \, e_\ell  \dx,
	\;  \zeta_\ell \dW^{v,n} 
	=\sum_{k=1}^n \zeta_{\ell,k}  \dW^v_k.
\end{align*}

Adding the first and second equations 
in \eqref{S1-Galerkin-bis} yields ($\ell=1,\dots,n$)
\begin{equation}\label{S1-Galerkin-bis:1}
	\begin{split}
		d c_\ell^n 
		& = \frac{1}{2+\eps_n}\left [A_{i,\ell} + A_{e,\ell}\right] \dt 
		+ \frac{2}{2+\eps_n} \Gamma_\ell \dW^{v,n},
		\\ & =: F_{ie,\ell} \dt+2 G_\ell\dW^{v,n},
	\end{split}
\end{equation}
and plugging \eqref{S1-Galerkin-bis:1} into \eqref{S1-Galerkin-bis} 
we arrive at ($\ell=1,\dots,n$)
\begin{equation}\label{S1-Galerkin-bis:2}
	\begin{split}
		d\left(\sqrt{\eps_n} c_{i,\ell}^n\right) 
		& = \left[\frac{1+\eps_n}{\sqrt{\eps_n}(2+\eps_n)}A_{i,\ell} 
		-\frac{1}{\sqrt{\eps_n}(2+\eps_n)} A_{e,\ell} \right] \dt 
		\\ & \qquad\qquad
		+\frac{\sqrt{\eps_n}}{2+\eps_n}\Gamma_\ell \dW^{v,n}
		=: F_{i,\ell} \dt+\sqrt{\eps_n}G_\ell\dW^{v,n},
		\\
		d\left(\sqrt{\eps_n} c_{e,\ell}^n\right) 
		& = \left[ \frac{1}{\sqrt{\eps_n}(2+\eps_n)} A_{i,\ell} 
		- \frac{1+\eps_n}{\sqrt{\eps_n}(2+\eps_n)} A_{e,\ell}\right] \dt
		\\ & \qquad\qquad 
		-\frac{\sqrt{\eps_n}}{2+\eps_n}\Gamma_\ell \dW^{v,n}
		=: F_{e,\ell} \dt-\sqrt{\eps_n}G_\ell\dW^{v,n},
		\\ 
		d a_\ell^n & = A_{H,\ell} \dt+\zeta_\ell \dW^{w,n}.
	\end{split}
\end{equation}

Let 
$$
C^n=C^n(t)=\Set{c^n(t),\sqrt{\eps_n}c_i^n(t),\sqrt{\eps_n}c_e^n(t),a^n(t)}
$$
be the vector containing all the unknowns in \eqref{S1-Galerkin-bis:1} 
and \eqref{S1-Galerkin-bis:2}. For technical reasons, related 
to \eqref{Carath-F-1} and \eqref{Carath-F-2} below, we 
write the left-hand sides of the first two equations 
in \eqref{S1-Galerkin-bis:2} in terms of the $\eps_n$ scaled 
quantities $\sqrt{\eps_n}c_i^n,\sqrt{\eps_n}c_e^n$. Moreover, 
we view the right-hand sides of all the equations 
as functions of $C^n$ (involving the $\eps_n$ scaled quantities),
which can always be done since $\eps_n>0$ is a fixed number.  
As a result, the constants below may depend on $1/\eps_n$.
Let
$$
F(C^n)=
\Set{
\Set{F_{ie,\ell}(C^n)}_{\ell=1}^n,
\Set{F_{i,\ell}(C^n)}_{\ell=1}^n,
\Set{F_{e,\ell}(C^n)}_{\ell=1}^n,
\Set{A_{H,\ell}(C^n)}_{\ell=1}^n
}
$$
be the vector containing all the drift terms, and 
$$
G(C^n)= 
\Set{
\Set{2G_\ell}_{\ell=1}^n,
\Set{\sqrt{\eps_n}G_\ell}_{\ell=1}^n,
\Set{-\sqrt{\eps_n}G_\ell}_{\ell=1}^n,
\Set{\zeta_\ell}_{\ell=1}^n
}, 
$$
be the collection of noise coefficients. The vector 
$\Set{W^{v,n},W^{v,n},W^{v,n},W^{w,n}}$ is denoted by $W^n$ .
Then \eqref{S1-Galerkin-bis:1} and \eqref{S1-Galerkin-bis:2} take 
the compact form
\begin{equation}\label{eq:SDE}
	dC^n(t) = F(C^n(t))\dt + G(C^n(t)) \dW^n(t), \quad C^n(0)=C_0^n,
\end{equation}
where $C_0^n=\Set{c^n(0),\sqrt{\eps_n}c_i^n(0),\sqrt{\eps_n}c_e^n(0),a^n(0)}$, 
cf.~\eqref{Galerkin-initdata}. 

If $F,G$ are globally Lipschitz continuous, classical SDE 
theory \cite{Karatzas,Prevot:2007aa,Skorokhod} provides the 
existence and uniqueness of a pathwise solution.  
However, due to the nonlinear nature of the ionic models, cf.~(\textbf{GFHN}), the 
global Lipschitz condition does not hold for the SDE system \eqref{eq:SDE}. 
As a replacement, we consider the following two conditions:
\begin{itemize}
	\item (local weak monotonicity) $\forall C_1, C_2 \in \R^{4n}$, 
	$\abs{C_1},\abs{C_2}\leq r$, for any $r>0$,
	\begin{equation}\label{Carath-F-1}
		2 \left(F(C_1)-F(C_2)\right) \cdot \left(C_1-C_2\right)
		+\abs{G(C_1)-G(C_2)}^2 \leq K_r \abs{C_1-C_2}^2,
	\end{equation}
	for some $r$-dependent positive constant $K_r$.

	\item (weak coercivity) $\forall C\in \R^{4n}$, there exists 
	a constant $K>0$ such that
	\begin{equation}\label{Carath-F-2}
		2 F(C) \cdot C+\abs{G(C)}^2 \leq K\left(1+\abs{C}^2\right).
	\end{equation}
\end{itemize}

Below we verify that the coefficients $F, G$ in \eqref{eq:SDE} 
satisfy these two conditions globally (i.e., \eqref{Carath-F-1} holds 
independent of $r$). Then, in view of Theorem 3.1.1 in 
\cite{Prevot:2007aa}, there exists a unique global adapted 
solution to \eqref{eq:SDE}.

Let us verify the weak monotonicity condition. To this end, set
\begin{align*}
	&u_j^n:=u_{j,1}^n-u_{j,2}^n \,\, (j=i,e),\quad
	v_k^n:=u_{i,k}^n-u_{e,k}^n \,\, (k=1,2),  \\
	& v^n:=v_1^n-v_2^n, \quad w^n:=w_1^n-w_2^n,
\end{align*}
where $\left( u_{i,1}^n, u_{e,1}^n,w_1^n\right)$ and 
$\left( u_{i,2}^n, u_{e,2}^n,w_2^n\right)$ 
are arbitrary functions of the form of \eqref{eq:approxsol}, with 
corresponding time coefficients $\left( c_{i,1}^n, c_{e,1}^n,a_1^n\right)$ 
and $\left( c_{i,2}^n, c_{e,2}^n,a_2^n\right)$, respectively. 
Moreover, set $c_1^n:= c_{i,1}^n-c_{e,1}^n$, 
$c_2^n:=  c_{i,1}^n-c_{e,1}^n$, and 
$C_k^n:=\Set{c_k^n,\sqrt{\eps_n}c_{i,k}^n,\sqrt{\eps_n}c_{e,k}^n,a_k^n}$ 
for $k=1,2$. 

We wish to show that 
$$
\cI_F:=\left( F(C_1^n)-F(C_2^n) \right) 
\cdot \left( C_1^n-C_2^n \right)\le K_F \abs{C_1^n-C_2^n}^2,
$$ 
i.e.,  that $F$ is globally one-sided Lipschitz. This requires 
comparing the ``$dt$-terms" in 
\eqref{S1-Galerkin-bis:1} and \eqref{S1-Galerkin-bis:2} 
corresponding to the vectors $C_1^n$ and $C_2^n$, 
resulting in three different types of terms, linked 
to the $M_j$ (diffusion) part, the $I$ (ionic) part, and the $H$ 
(gating) part of the equations, that is, 
$\cI_F = \cI_F^M + \cI_F^I+ \cI_F^H$. 
First,
\begin{align*}
	\cI_F^I & =
	\frac{-2}{2+\eps_n}
	\sum_{l=1}^n
	\int_{\Om} \left( I(v_1^n,w_1^n) -I(v_2^n,w_2^n)\right)e_l\dx 
	\left(c_{1,l}^n-c_{2,l}^n\right)
	\\ & \qquad 
	+ \frac{ -(1+\eps_n) + 1}{\sqrt{\eps_n} (2+\eps_n)}\sum_{l=1}^n 
	\int_{\Om} \left(I(v_1^n,w_1^n)-I(v_2^n,w_2^n)\right)e_l\dx 
	\\ & \qquad\qquad\qquad\qquad\qquad\qquad\quad
	\times 
	\left (\sqrt{\eps_n}c_{i,1,l}^n-\sqrt{\eps_n}c_{i,2,l}^n \right)
	\\ &\qquad 
	+ \frac{-1 + (1+\eps_n)}{\sqrt{\eps_n}(2+\eps_n)}
	\sum_{l=1}^n
	\int_{\Om}  \left(I(v_1^n,w_1^n)-I(v_2^n,w_2^n)\right)e_l\dx
	\\ & \qquad\qquad\qquad\qquad\qquad\qquad\quad
	\times \left(\sqrt{\eps_n}c_{e,1,l}(t)-\sqrt{\eps_n}c_{e,2,l}(t)\right) 
	\\ &
	=-\frac{2}{2+\eps_n}
	\int_{\Om} \left(I(v_1^n,w_1^n)-I(v_2^n,w_2^n)\right)v^n\dx
	\\ &\qquad 
	- \frac{\eps_n}{2+\eps_n}
	\int_{\Om} \left(I(v_1^n,w_1^n)-I(v_2^n,w_2^n)\right)u_i^n\dx
	\\ &\qquad 
	+ \frac{\eps_n}{2+\eps_n}
	\int_{\Om} \left(I(v_1^n,w_1^n)-I(v_2^n,w_2^n)\right)u_e^n\dx
	\\ &
	= - \int_{\Om} \left(I(v_1^n,w_1^n)-I(v_2^n,w_2^n)\right)v^n\dx.
\end{align*}
Similarly, 
\begin{align*}
	\cI_F^H & =\sum_{l=1}^n 
	\int_{\Om} \left( H(v_1^n,w_1^n)-H(v_2^n,w_2^n) \right) e_l\dx
	\left(a_{1,l}-a_{2,l}\right)
	\\ & =\int_{\Om} \left(H(v_1^n,w_1^n)-H(v_2^n,w_2^n)\right)w^n\dx ,
\end{align*}
and therefore $\cI_F^I + \cI_F^H$ becomes 
\begin{equation}\label{eq:HminusI}
	\begin{split}
		&\int_{\Om} \Bigl( \left ( H(v_1^n,w_1^n)-H(v_2^n,w_2^n)\right)w^n
		-\left(I(v_1^n,w_1^n)-I(v_2^n,w_2^n)\right)v^n \Bigr)\dx
		\\ & \qquad 
		\le \tilde K_F^{H,I} \left(\norm{v_1^n-v_2^n}_{L^2(\Om)}^2
		+ \norm{w_1^n-w_2^n}_{L^2(\Om)}^2\right)
		\le K_F^{H,I} \abs{C_1^n-C_2^n}^2,
	\end{split}		
\end{equation}
for some constants $\tilde K_F^{H,I},K_F^{H,I}$. The first inequality 
in \eqref{eq:HminusI} follows as in \cite[page 479]{Bourgault:2009aa} from 
the structural condition {\rm (\textbf{GFHN})}. Finally, we find that
\begin{align*}
	\cI_F^M & =
	\frac{1}{2+\eps_n}
	\int_{\Om} \left(-M_i \Grad U_i^n+ M_e\Grad U_e^n\right)
	\cdot \Grad V^n\dx
	\\ & \quad 
	+ \frac{1}{2+\eps_n}  
	\int_{\Om} \left(-(1+\eps_n) M_i \Grad U_i^n
	- M_e\Grad U_e^n\right)\cdot \Grad U_i^n\dx
	\\ &\quad 
	+ \frac{1}{2+\eps_n}
	\int_{\Om} \left(-M_i \Grad U_i^n-(1+\eps_n) M_e\Grad U_e^n\right)
	\cdot \Grad U_e^n\dx.
\end{align*}
Adding the integrands gives
\begin{align*}
	&\left(-M_i \Grad U_i^n+ M_e\Grad U_e^n\right)\cdot \Grad V^n
	+\left(-(1+\eps_n) M_i \Grad U_i^n
	- M_e\Grad U_e^n\right)\cdot \Grad U_i^n
	\\ & \qquad  
	+\left(-M_i \Grad U_i^n-(1+\eps_n) M_e\Grad U_e^n\right)
	\cdot \Grad U_e^n
	\\ & = -\left(2+\eps_n\right) M_i \Grad U_i^n\cdot \Grad U_i^n 
	-\left(2+\eps_n\right) M_e \Grad U_e^n\cdot \Grad U_e^n,
\end{align*}
and thus, cf.~\eqref{matrix},
$\cI_F^M  = - \sum_{j=i,e} M_j \Grad U_j^n\cdot \Grad U_j^n \le 0$.
Hence, $F$ is globally one-sided Lipschitz.  In view 
of \eqref{eq:noise-cond}, it follows easily that $G$ is globally Lipschitz: 
$$
\abs{G(C_1^n)-G(C_2^n)} \leq K_G \abs{C_1^n-C_2^n},
$$ 
for some constant $K_G$ (depending on $n$). 
Summarizing, condition \eqref{Carath-F-1} holds.

In much the same way, again using assumptions (\textbf{GFHN}) 
and \eqref{eq:noise-cond}, we deduce that 
$$
F(C_1^n) \cdot C_1^n \leq K_F \left(1+\abs{C_1^n}^2\right),
\qquad
\abs{G(C_1^n)}^2 \leq K_G \left(1+\abs{C_1^n}^2\right),
$$
for some constants $K_F, K_G$; that is to say, 
condition \eqref{Carath-F-2} holds.
\end{proof}

\section{Convergence of approximate solutions}\label{sec:convergence}

\subsection{Basic apriori estimates}\label{sec:apriori}

To establish convergence of the Faedo-Galerkin 
approximations, we must supply a series of apriori 
estimates that are independent of the parameter 
$n$ (cf.~Lemma \ref{lem:apriori-est} below).
At an informal level, assuming that the relevant functions are 
sufficiently regular, these estimates are obtained by considering
\begin{equation}\label{eq:vpmu}
	\begin{split}
		& d \left(v + \eps_n u_i\right) = 
		\left[\Div\bigl(M_i \Grad u_i \bigr) - I(v,w) \right]\dt + \eta(v) \dW^v 
		\\ 
		& d \left(v - \eps_n u_e\right) = 
		\left[-\Div\bigl(M_e \Grad u_e \bigr) - I(v,w) \right]\dt + \eta(v)\dW^v,
	\end{split}
\end{equation}
where $\eps_n$ is defined in \eqref{eq:def-eps}, multiplying the first 
equation by $u_i$, the second equation by $-u_e$, and summing the 
resulting equations. For the moment, let us assume 
that the noise $W^v$ is one-dimensional and $\eta(v)$ 
is a scalar function. To proceed we use the stochastic (It\^{o}) product rule. 
Hence, we need access to the equation for $du_i$, which turns out to be
\begin{align*}
	du_i  
	& =\left[ \frac{1+\eps_n}{\eps_n(2+\eps_n)}\Div\bigl(M_i \Grad u_i \bigr)
	+\frac{1}{\eps_n(2+\eps_n)}\Div\bigl(M_e \Grad u_e \bigr)
	-  \frac{1}{2+\eps_n} I(v,w) \right]\dt
	\\ & \qquad \qquad +  \frac{1}{2+\eps_n}\eta(v) \dW^v.
\end{align*}
Note that this equation ``blows up" as $\eps_n\to 0$ (the same 
is true for the $du_e$ equation below). The stochastic product rule gives
\begin{equation}\label{eq:product-ui}
	\begin{split}
		d\left(u_i \left(v  + \eps_n u_i\right) \right)
		& = u_i \, d \left(v + \eps_n u_i\right) 
		+ d u_i \left(v+\eps_n u_i\right) + \frac{1}{2+\eps_n}\eta(v)^2 \dt
		\\ & = \frac{1}{2+\eps_n}\eta(v)^2 \dt
		+\left[u_i \, \Div\bigl(M_i \Grad u_i\bigr) - u_i \, I(v,w) \right]\dt 
		\\ & \qquad  
		+ u_i \,\eta(v) \dW^v
		+\Biggl[\cdots \Biggr]_i\dt 
		+  \frac{1}{2+\eps_n}\left(v+\eps_n u_i\right)\eta(v) \dW^v,
	\end{split}
\end{equation}
where
\begin{align*}
	\Biggl[\cdots \Biggr]_i\dt & =
	\Biggl[ \frac{1+\eps_n}{\eps_n(2+\eps_n)}\left(v+\eps_n u_i\right)
	\Div\bigl(M_i \Grad u_i \bigr)
	\\ & \qquad \qquad 
	+\frac{1}{\eps_n(2+\eps_n)}\left(v+\eps_n u_i\right)\Div\bigl(M_e \Grad u_e \bigr)
	\\ & \qquad \qquad\qquad 
	-  \frac{1}{2+\eps_n} \left(v+\eps_n u_i\right) I(v,w) \Biggr]\dt.
\end{align*}
Similar computations, this time involving the equation
\begin{align*}
	du_e& =\left[ \frac{1}{\eps_n(2+\eps_n)}\Div\bigl(M_i \Grad u_i \bigr)
	+\frac{1+\eps_n}{\eps_n(2+\eps_n)}\Div\bigl(M_e \Grad u_e \bigr)
	+  \frac{1}{2+\eps_n} I(v,w) \right]\dt
	\\ & \qquad \qquad -  \frac{1}{2+\eps_n}\eta(v) \dW^v,
\end{align*}
yield
\begin{equation}\label{eq:product-ue}
	\begin{split}
		d\left(-u_e \left(v  - \eps_n u_e\right) \right)
		& = -u_e \, d \left(v - \eps_n u_e\right) 
		- d u_e \left(v-\eps_n u_e\right) + \frac{1}{2+\eps_n}\eta(v)^2 \dt
		\\ & = \frac{1}{2+\eps_n}\eta(v)^2 \dt
		+ \left[u_e\, \Div\bigl(M_e \Grad u_e \bigr) + u_eI(v,w) \right]\dt 
		\\ & \qquad 
		- u_e\,\eta(v) \dW^v
		+\Biggl[\cdots \Biggr]_e\dt
		+\frac{1}{2+\eps_n}\left(v-\eps_n u_e\right) \eta(v) \dW^v,
	\end{split}
\end{equation}
where
\begin{align*}
	\Biggl[\cdots \Biggr]_e\dt
	& =\Biggl[- \frac{1}{\eps_n(2+\eps_n)}
	\left(v-\eps_n u_e\right)\Div\bigl(M_i \Grad u_i \bigr)
	\\ & \qquad \qquad 
	-\frac{1+\eps_n}{\eps_n(2+\eps_n)}\left(v-\eps_n u_e\right)
	\Div\bigl(M_e \Grad u_e \bigr)
	\\ & \qquad \qquad \qquad
	- \frac{1}{2+\eps_n}\left(v-\eps_n u_e\right) I(v,w) \Biggr]\dt.
\end{align*}
After some computations we find that
\begin{align*}
	\Biggl[\cdots \Biggr]_i\dt +\Biggl[\cdots \Biggr]_e\dt
	= \Biggl[ 2u_i\, \Div\bigl(M_i \Grad u_i \bigr)
	+2u_e\, \Div\bigl(M_e \Grad u_e \bigr)
	- 2v\, I(v,w) \Biggr]\dt
\end{align*}
and
\begin{align*}
	& u_i \, \eta(v) \dW^v+ \frac{1}{2+\eps_n}\left(v+\eps_n u_i\right)\eta(v) \dW^v
	\\ & \quad \quad
	- u_e\, \eta(v) \dW^v + \frac{1}{2+\eps_n}\left(v-\eps_n u_e\right) \eta(v) \dW^v
	= 2v \, \eta(v) \dW^v.
\end{align*}
Whence, adding \eqref{eq:product-ui} and \eqref{eq:product-ue},
\begin{align*}
	&d\left(v^2  + \eps_n u_i^2 + \eps_n u_e^2 \right)
	= d\left(u_i \left(v  + \eps_n u_i\right) \right)
	+d\left(-u_e \left(v  - \eps_n u_e\right) \right)
	\\ & \qquad 
	=  \Biggl [\frac{2}{2+\eps_n}\eta(v)^2 
	+2 u_i \, \Div\bigl(M_i \Grad u_i\bigr)
	\\ & \qquad \qquad \qquad\qquad
	+ 2 u_e\, \Div\bigl(M_e \Grad u_e \bigr) 
	-2 v\, I(v,w) \Biggr]\dt +2 v\, \eta(v) \dW^v.
\end{align*}
Adding to this the equation for $d w^2$, resulting 
from \eqref{eq:nondegen} and It\^{o}'s formula, the estimates in 
Lemma \ref{lem:apriori-est} below appear once we 
integrate in $x$ and $t$, make use of spatial integration by parts, the 
boundary conditions \eqref{S-bc}, and properties of the nonlinear 
functions $I,H$ implying \eqref{eq:nonlinear-est} below. 
Arguing at the level of finite dimensional approximations, we 
now convert the computations outlined above into a rigorous proof.

\begin{lem}\label{lem:apriori-est}
Suppose conditions {\rm (\textbf{GFHN})}, \eqref{matrix}, 
\eqref{eq:noise-cond}, and \eqref{eq:uie-init-ass} hold. Let 
$$
u_i^n(t),u_e^n(t),v^n(t),w^n(t), \quad t\in [0,T],
$$
satisfy \eqref{S1-Galerkin-new}, \eqref{eq:def-eps}, 
\eqref{eq:Wn}, \eqref{Galerkin-initdata}. 
There is a constant $C>0$, independent of $n$, such that
\begin{align}
	\label{Gal:est1}
	& \E \left[ \norm{v^n(t)}_{L^2(\Om)}^2\right]
	+\E \left[ \norm{w^n(t)}_{L^2(\Om)}^2\right]
	\\ & \qquad\qquad\qquad\quad\,
	+ \sum_{j=i,e} \E \left[ \norm{\sqrt{\eps_n}u_j^n(t)}_{L^2(\Om)}^2 \right]
	\le C, \qquad \forall t\in [0,T]; \notag
	\\
	\label{Gal:est2}
	& \sum_{j=i,e} \E \left[\int_0^T \int_{\Om}\abs{\Grad u_j^n}^2 \dx \dt \right]
		+ \E \left[\int_0^T \int_{\Om} \abs{v^n}^4 \dx\dt \right] \le C;
	\\
	\label{Gal:est2-new}
	& \sum_{j=i,e} \E \left[ \int_0^T \int_{\Om} \abs{u_j^n}^2 \dx \dt \right]
	\le C;
	\\
	\label{Gal:est3}
	& \E \left[ \sup_{t\in [0,T]}\norm{v^n(t)}_{L^2(\Om)}^2\right]
	+\E \left[ \sup_{t\in [0,T]} \norm{w^n(t)}_{L^2(\Om)}^2\right]
	\\ & \qquad\qquad\qquad\qquad\qquad\,
	+ \sum_{j=i,e} \E \left[ \sup_{t\in [0,T]} 
	\norm{\sqrt{\eps_n}u_j^n(t)}_{L^2(\Om)}^2 \right]
	\le C. \notag
\end{align}
\end{lem}

\begin{proof}
Motivated by the discussion above, we wish to compute  
\begin{equation}
	\begin{split}\label{eq:L2est-tmp}
		& d \int_{\Om} \left(v^n\right)^2+\eps_n \left(u_i^n\right)^2
		+\eps_n \left(u_e^n\right)^2 \dx
		\\ & \qquad 
		 = d \int_{\Om}  u_i^n \left(v^n+\eps_n u_i^n \right)\dx
		+
		d\int_{\Om}  -u_e^n \left(v^n -\eps_n u_e^n\right)\dx
		\\ & \qquad  = 
		\sum_{\ell=1}^n d\left(c_{i,\ell}^n \left(c_\ell^n + \eps_n c_{i,\ell}^n \right)\right)
		+ \sum_{\ell=1}^n d\left(-c_{e,\ell}^n\left(c_\ell^n-\eps_n c_{e,\ell}^n \right)\right),
	\end{split}
\end{equation}
where we have used \eqref{eq:approxsol} and the orthonormality of the basis. 

First, in view of \eqref{S1-Galerkin-bis} and \eqref{S1-Galerkin-bis:2}, the 
stochastic product rule implies ($\ell=1,\ldots,n$)
\begin{equation}\label{eq:product-cil}
	\begin{split}
		& d\left(c_{i,\ell}^n(c_\ell^n  + \eps_n c_{i,\ell}^n) \right)
		\\& \quad = \left( c_{i,\ell}^n d \left(c_{\ell}^n + \eps_n c_{i,\ell}^n\right) \right)
		+ \left(d c_{i,\ell}^n \left(c_{\ell}^n+\eps_n c_{i,\ell}^n\right)\right) 
		\\ & \qquad \qquad
		+ \frac{1}{2+\eps_n}\sum^n_{k=1}
		\left(\int_{\Om} \eta_k^n(v^n)e_l\dx \right)^2 \dt
		\\ & \quad = 
		\frac{1}{2+\eps_n}\sum^n_{k=1}
		\left(\int_{\Om} \eta_k^n(v^n)e_l\dx \right)^2 \dt
		\\ & \qquad\qquad
		+\int_{\Om} \left( M_i \Grad u^n_i\cdot \Grad e_\ell
		-I(v^n,w^n) e_\ell \right) \dx \,c_{i,\ell}^n \,\dt 
		\\ &  \qquad \qquad 
		+ \sum^n_{k=1}\int_{\Om}  \eta^n_k(v^n) e_\ell \dx \, c_{i,\ell}^n \dW^{v,n}
		+\Biggl[\cdots \Biggr]_i\dt 
		\\ & \qquad \qquad 
		+  \frac{1}{2+\eps_n}\sum^n_{k=1}\int_{\Om}\eta^n_k(v^n) e_\ell \dx 
		\, \left(c_{\ell}^n+\eps_n c_{i,\ell}^n\right) \dW^{v,n},
	\end{split}
\end{equation}
where
\begin{align*}
	\Biggl[\cdots \Biggr]_i\dt & =
	\Biggl[ \frac{1+\eps_n}{\eps_n(2+\eps_n)}
	\int_{\Om} M_i \Grad u^n_i \cdot  \Grad e_\ell \dx
	\left(c_{\ell}^n+\eps_n c_{i,\ell}^n\right)
	\\ & \qquad \quad 
	+\frac{1}{\eps_n(2+\eps_n)} 
	\int_{\Om} M_e \Grad u^n_e \cdot \Grad e_\ell \dx
	\left(c_{\ell}^n+\eps_n c_{i,\ell}^n\right)
	\\ & \qquad \quad \quad 
	-\frac{1}{2+\eps_n} 
	\int_{\Om} I(v^n,w^n) e_\ell \dx
	\left(c_{\ell}^n+\eps_n c_{i,\ell}^n\right) \Biggr]\dt.
\end{align*}

Similar computations give ($\ell=1,\ldots,n$)
\begin{equation}\label{eq:product-cel}
	\begin{split}
		& d\left(-c_{e,\ell}^n, (c_{\ell}^n  + \eps_n c_{e,\ell}^n) \right)
		\\& \quad  = \left(-c_{e,\ell}^n d \left(c_{\ell}^n - \eps_n c_{\ell}^n\right) \right) 
		- \left(d c_{e,\ell}^n \left(c_{\ell}^n-\eps_n c_{e,\ell}^n\right) \right) 
		\\ & \qquad \qquad
		+\frac{1}{2+\eps_n}\sum^n_{k=1}
		\left(\int_{\Om} \eta_k^n(v^n)e_l\dx \right)^2 \dt
		\\ & \quad = \frac{1}{2+\eps_n}\sum^n_{k=1}
		\left(\int_{\Om} \eta_k^n(v^n)e_l\dx \right)^2 \dt
		\\ & \qquad\qquad
		+  \int_{\Om} (M_e \Grad u^n_e \cdot \Grad e_\ell 
		+ I(v^n,w^n)e_\ell)\dx \,c_{e,\ell}^n\dt 
		\\ & \qquad \qquad
		- \sum^n_{k=1}\int_{\Om} \eta^n_k(v^n) e_l\dx\, c_{e,\ell} \dW^{v,n}
		+\Biggl[\cdots \Biggr]_e\dt
		\\ & \qquad \qquad 
		+\frac{1}{2+\eps_n}\sum^n_{k=1}\int_{\Om}  
		\eta^n_k(v^n)e_\ell \dx 
		\left(c_{\ell}^n -\eps_n c_{e,\ell}^n \right) \dW^{v,n},
	\end{split}
\end{equation}
where
\begin{align*}
	\Biggl[\cdots \Biggr]_e\dt
	& =\Biggl[- \frac{1}{\eps_n(2+\eps_n)}
	\int_{\Om}M_i \Grad u^n_i \cdot \Grad e_\ell \dx
	\left(c_{\ell}^n-\eps_n c_{e,\ell}^n \right)
	\\ & \qquad \qquad 
	-\frac{1+\eps_n}{\eps_n(2+\eps_n)}
	\int_{\Om} M_e \Grad u^n_e \cdot \Grad e_\ell 
	\left(c_{\ell}^n-\eps_n c_{e,\ell}^n\right)\dx
	\\ & \qquad \qquad \qquad 
	- \frac{1}{2+\eps_n} \int_{\Om} I(v^n,w^n) e_\ell \dx
	\left(c_{\ell}^n-\eps_n c_{e,\ell}^n\right)\Biggr]\dt.
\end{align*}

Combining \eqref{eq:L2est-tmp}, \eqref{eq:product-cil}, \eqref{eq:product-cel} 
we arrive eventually at
\begin{equation}\label{eq:v-product}
	\begin{split}
		& d \int_{\Om}  \abs{v^n}^2 
		+\eps_n \abs{u_i^n}^2 + \eps_n \abs{u_e^n}^2 \dx
		\\ & \; = 
		\Biggl[ 
		- 2\int_{\Om}  M_i \Grad u_i^n \cdot \Grad u_i^n \dx 
		- 2\int_{\Om}  M_e \Grad u_e^n\cdot \Grad u_e^n \dx 
		- 2\int_{\Om} v^n I(v^n,w^n) \dx
		\\ &\qquad\quad  
		+ \frac{2}{2+\eps_n}
		\sum_{k,l=1}^n \left(\int_{\Om} \eta_k^n(v^n)e_l\dx \right)^2 \dt 
		\Biggr] \dt+ 2\int_{\Om} v^n \eta^n(v^n) \dx \dW^{v,n}.
	\end{split}
\end{equation}
Similarly, in view of \eqref{eq:approxsol} and \eqref{S1-Galerkin-bis:2}, 
It\^{o}'s lemma gives
\begin{equation}\label{eq:w-ito}
	\begin{split}
		d\int_{\Om} \abs{w^n}^2 \dx
		& =\left[\,  2\int_{\Om} w^n H(v^n,w^n) \dx 
		+ \sum_{k,l=1}^n \left(\int_{\Om} \sigma_k^n(v^n)e_l\dx \right)^2 \right]\dt
		\\ & \qquad \qquad
		+2\int_{\Om} w^n \sigma^n(v^n) \dW^{w,n}.
      \end{split}
\end{equation}
After integration in time, adding \eqref{eq:v-product} 
and \eqref{eq:w-ito} delivers 
\begin{equation}\label{eq:v+w-1}
	\begin{split}
		& \frac12\norm{v^n(t)}_{L^2(\Om)}^2
		+\sum_{j=i,e}\frac12\norm{\sqrt{\eps_n}u_j^n(t)}_{L^2(\Om)}^2
		+\frac12\norm{w^n(t)}_{L^2(\Om)}^2.
		\\ &  \qquad \quad
		+\sum_{j=i,e} \int_0^t \int_{\Om}  M_j\Grad u_j^n
		\cdot \Grad u_j^n \dx \ds 
		\\ & 
		= \frac12\norm{v^n(0)}_{L^2(\Om)}^2
		+\sum_{j=i,e}\frac12\norm{\sqrt{\eps_n}u_j^n(0)}_{L^2(\Om)}^2
		+\frac12\norm{w^n(0)}_{L^2(\Om)}^2
		\\ & \qquad\quad
		+\int_0^t\int_{\Om}  \bigl( w^n H(v^n,w^n) - v^n I(v^n,w^n)\bigr) \dx \ds 
		\\ & \qquad \quad
		+\frac{1}{2+\eps_n}
		\sum_{k,l=1}^n 
		\int_0^t\left(\int_{\Om} \eta_k^n(v^n)e_l\dx \right)^2\ds
		\\ & \qquad \quad
		+\frac{1}{2}
		\sum_{k,l=1}^n 
		\int_0^t\left(\int_{\Om} \sigma_k^n(v^n)e_l\dx \right)^2\ds
		\\
		&  \qquad \quad
		+\int_0^t\int_{\Om} v^n \eta^n(v^n) \dx \dW^{v,n}(s)
		\\ &  \qquad \quad
		+\int_0^t\int_{\Om} w^n \sigma^n(v^n) \dx \dW^{w,n}(s),
      \end{split}
\end{equation}
for any $t\in [0,T]$. By (\textbf{GFHN}) and 
repeated applications of Cauchy's inequality,
\begin{equation}\label{eq:nonlinear-est}
	w H(v,w) - vI(v,w) \le -C_1 \abs{v}^4 
	+C_2\left(\abs{v}^2 + \abs{w}^2\right) + C_3,
\end{equation}
for some constants $C_1>0$ and $C_2,C_3\ge 0$. Recalling that 
$\Set{e_l}_{l\ge 1}$ is a basis for $L^2(\Om)$,
\begin{equation}\label{eq:qvar-est}
	\begin{split}
		& \sum_{k,l=1}^n 
		\int_0^t\left(\int_{\Om} \eta_k^n(v^n)e_l\dx \right)^2\ds
		+\sum_{k,l=1}^n 
		\int_0^t\left(\int_{\Om} \sigma_k^n(v^n)e_l\dx \right)^2\ds 
		\\ & \quad \le
		\int_0^t\int_{\Om}\sum_{k=1}^n\abs{\eta_k(v^n)}^2\dx \ds
		+\int_0^t\int_{\Om} \sum_{k=1}^n\abs{\sigma_k(v^n)}^2\dx \ds
		\\ &\quad \quad 
		\overset{\eqref{eq:noise-cond}}{\le}  
		C_4\left(\int_0^t\int_{\Om}  \abs{v^n}^2 \dx \ds + t \abs{\Om}\right),
	\end{split}	
\end{equation}
for some constant $C_4>0$. Using \eqref{eq:nonlinear-est}, 
\eqref{eq:qvar-est}, and \eqref{matrix} in \eqref{eq:v+w-1}, we obtain 
\begin{equation}\label{eq:v+w-2}
	\begin{split}
		& \frac12\norm{v^n(t)}_{L^2(\Om)}^2
		+\sum_{j=i,e}\frac12\norm{\sqrt{\eps_n}u_j^n(t)}_{L^2(\Om)}^2
		+\frac12\norm{w^n(t)}_{L^2(\Om)}^2.
		\\ & \quad
		+ m \sum_{j=i,e} \int_0^t \int_{\Om}  \abs{\Grad u_j^n}^2 \dx \ds 
		+C_1 \int_0^t \int_{\Om}  \abs{v}^4 \dx \ds 
		\\ &
		\le \frac12\norm{v^n(0)}_{L^2(\Om)}^2
		+\sum_{j=i,e}\frac12\norm{\sqrt{\eps_n}u_j^n(0)}_{L^2(\Om)}^2
		+\frac12\norm{w^n(0)}_{L^2(\Om)}^2
		\\ & \quad
		+(C_3+C_4) t \abs{\Om}
		\\ & \quad 
		+(C_2+C_4)\int_0^t \norm{v^n(s)}_{L^2(\Om)}^2 \ds 
		+ C_2\int_0^t \norm{w^n(s)}_{L^2(\Om)}^2 \ds
		\\ &  \quad
		+\int_0^t\int_{\Om} v^n \eta^n(v^n) \dx \dW^{v,n}(s)
		+\int_0^t\int_{\Om} w^n \sigma^n(v^n) \dx \dW^{w,n}(s).
      \end{split}
\end{equation}
Since $\E\left[\int_0^T\abs{f(t)}^2 \dt\right] <\infty$ 
for $f=\int_{\Om} v^n \eta^n(v^n)\dx$ and $f=\int_{\Om} w^n \sigma^n(v^n)\dx$, the 
martingale property of stochastic integrals ensures that
the expected value of each of the last two terms in \eqref{eq:v+w-2} is zero. 
Hence, taking the expectation in \eqref{eq:v+w-2}, keeping in 
mind \eqref{eq:uie-init-ass} and using Gr{\"o}nwall's inequality, we conclude that 
\eqref{Gal:est1} and \eqref{Gal:est2} hold. 

The refinement of \eqref{Gal:est1} into \eqref{Gal:est3} comes 
from a martingale inequality. Indeed, taking the $\sup$ over $[0,T]$ and 
subsequently applying $\E[\cdot]$ in \eqref{eq:v+w-2}, it follows that
\begin{equation}\label{Esup1}
	\begin{split}
		&\E\left[ \sup_{t\in [0,T]}\norm{v^n(t)}_{L^2(\Om)}^2\right]
		+\sum_{j=i,e}\E\left[\sup_{t\in [0,T]} 
		\norm{\sqrt{\eps_n}u_j^n(t)}_{L^2(\Om)}^2\right]
		\\ & \qquad  
		+\E\left[\sup_{t\in [0,T]} \norm{w^n(t)}_{L^2(\Om)}^2\right] 
		\le C_5\left(1 + \Gamma_\eta+\Gamma_\sigma\right),
	\end{split}
\end{equation}
where $C_5$ is a constant independent of $n$ and 
\begin{align*}
	\Gamma_\eta & := \E\left[\, \sup_{t\in [0,T]} \abs{\int_0^t\int_{\Om} 
	v^n \eta^n(v^n) \dx \dW^{v,n}(s)}\, \right], 
	\\ 
	\Gamma_\sigma & := \E\left[\, \sup_{t\in [0,T]} \abs{\int_0^t\int_{\Om} 
	w^n \sigma^n(v^n) \dx \dW^{w,n}(s)} \, \right].
\end{align*}
To arrive at \eqref{Esup1} we have used \eqref{eq:uie-init-ass}, \eqref{Gal:est1}.

We use the Burkholder-Davis-Gundy inequality to handle the last two terms. 
To be more precise, using \eqref{eq:bdg}, the Cauchy-Schwarz inequality, 
the assumption \eqref{eq:noise-cond} on $\eta$, Cauchy's 
inequality ``with $\delta$", and \eqref{Gal:est1}, we obtain
\begin{equation}\label{Esup2}
	\begin{split}
		\Gamma_\eta & \le C_6 \E \left[ \left(\int_0^T \sum_{k=1}^n 
		\abs{\int_{\Om} v^n \eta^n_k(v^n) \dx}^2 \dt \right)^{\frac12}\right]
		\\ &  \le   C_6 \E \left[ \left(\int_0^T  
		\left(\int_{\Om} \abs{v^n}^2\dx\right) 
		\left(\sum_{k=1}^n\int_{\Om} \abs{\eta^n_k(v^n)}^2 
		\dx \right)\dt\right)^{\frac12} \right]
		\\ &  \le C_6 \E \left[ \left(\sup_{t\in [0,T]}\int_{\Om} \abs{v^n}^2\dx\right)^{\frac12}
		\left( \int_0^T\sum_{k=1}^n\int_{\Om} \abs{\eta^n_k(v^n)}^2 \dx\dt \right)^{\frac12}\right]
		\\ &  \le \delta \E \left[\sup_{t\in [0,T]}\int_{\Om} \abs{v^n}^2\dx\right]
		+C_7(\delta)\E\left[ \int_0^T\sum_{k=1}^n\int_{\Om} \abs{\eta^n_k(v^n)}^2 \dx\dt\right]
		\\ &  \le \delta \E \left[\sup_{t\in [0,T]}\int_{\Om} \abs{v^n}^2\dx\right]
		+C_8\E\left[ \int_0^T\int_{\Om} \abs{v^n}^2 \dx\dt + T \abs{\Om}\right]
		\\ &  \le \delta \E\left[ \sup_{t\in [0,T]} \norm{v^n(t)}_{L^2(\Om)}^2\right]+C_9,
	\end{split}
\end{equation}
for any $\delta>0$. Similarly, using 
\eqref{eq:noise-cond} and \eqref{Gal:est1},
\begin{equation}\label{Esup3}
	\Gamma_\sigma \le \delta \E\left[\sup_{t\in [0,T]} 
	\norm{w^n(t)}_{L^2(\Om)}^2\right] + C_{10}.
\end{equation}
Combining \eqref{Esup1}, \eqref{Esup2} and \eqref{Esup3}, 
with $\delta>0$ small, the desired estimate \eqref{Gal:est3} follows. 

Finally, let us prove \eqref{Gal:est2-new}. By the Poincar\'e inequality, there 
is a constant $C_{11} > 0$, depending on $\Om$ but 
not $n, \omega$ and $t$, such that for each fixed $(\omega,t)\in D\times [0,T]$,
$$
\norm{u_e^n(\omega,t,\cdot)}_{L^2(\Om)}^2
\le C_{11} \norm{\nabla u_e^n(\omega,t,\cdot)}_{L^2(\Om)}^2.
$$
Hence, by \eqref{Gal:est2},
\begin{equation}\label{Gal:est2-new-tmp}
	\E \left[ \int_0^T \norm{u_e^n(\omega,t,\cdot)}_{L^2(\Om)}^2\dt\right]
	\le C_{12}.
\end{equation}
Since $v^n$ ($=u_i^n-u_e^n)$ complies with \eqref{Gal:est1}, it follows 
that also $u_i^n$ satisfies \eqref{Gal:est2-new-tmp}.
\end{proof}

In view of the $n$-independent estimates in Lemma \ref{lem:apriori-est}, 
passing if necessary to a proper subsequence, we can assume that the 
following (weak) convergences hold as $n\to\infty$:
\begin{equation}\label{eq:weakconv1}
	\left \{
	\begin{split}
		& u^n_j \weak u_j  \quad  \text{in} \quad 
		L^2\left(D,\cF,P;L^2((0,T);\tH(\Om))\right), \, j=i,e,
		\\&
		\eps_n u^n_j \to 0  \quad  \text{in} \quad 
		L^2\left(D,\cF,P;L^2((0,T);L^2(\Om))\right), \,j=i,e,
		\\& 
		v^n \weak v \quad  \text{in} \quad 
		L^2\left(D,\cF,P;L^2((0,T);\tH(\Om))\right),
		\\ & 
		v^n \weakstar v \quad  \text{in} \quad 
		L^2\left(D,\cF,P;L^\infty((0,T);L^2(\Om))\right),
		\\ & 
		v^n \weak v \quad  \text{in} \quad 
		L^4\left(D,\cF,P;L^4(\Om_T))\right),
		\\ & 
		w^n \weakstar w \quad  \text{in} \quad 
		L^2\left(D,\cF,P;L^\infty((0,T);L^2(\Om))\right).
	\end{split}\right.
\end{equation}

The next result, a consequence of Lemma \ref{lem:apriori-est} and 
a martingale inequality, supplies high-order moment estimates, useful when
converting a.s.~convergence into $L^2$ convergence.

\begin{cor}\label{cor:Lq0-est}
In addition to the assumptions in Lemma \ref{lem:apriori-est}, 
suppose \eqref{eq:uie-init-ass-q0} holds with 
$q_0$ defined in \eqref{eq:moment-est}. 
There exists a constant $C>0$, independent of $n$, such that
\begin{equation}\label{eq:Lq0-est}
	\begin{split}
		&\E\left[ \, \sup_{0\le t\le T}\norm{v^n(t)}_{L^2(\Om)}^{q_0}\, \right] 
		+\sum_{j=i,e}\E\left[ \, \sup_{0\le t\le T}
		\norm{\sqrt{\eps_n}u_i^n(t)}_{L^2(\Om)}^{q_0}\, \right]
		\\ &\quad \qquad
		+\E\left[ \,\sup_{0\le t\le T} \norm{w^n(t)}_{L^2(\Om)}^{q_0}\, \right]
		\le C.
	\end{split}
\end{equation}
Moreover, 
\begin{equation*}
	\sum_{j=i,e}
	\E\left[ \norm{\Grad u_j^n}_{L^2((0,T)\times \Om)}^{q_0}\right] 
	+\E\left[ \norm{v^n}_{L^4((0,T)\times \Om)}^{2q_0}\right]  \le C.
\end{equation*}
\end{cor}

\begin{proof}
In view of \eqref{eq:v+w-2}, we have the following 
estimate for any $(\omega,t)\in D \times [0,T]$: 
\begin{align*}
	& \sup_{0\le \tau\le t}\norm{v^n(\tau)}_{L^2(\Om)}^2 
	+ \sum_{j=i,e}\sup_{0\le \tau\le t}\norm{\sqrt{\eps_n}u_j^n(\tau)}_{L^2(\Om)}^2
	\\ &\quad \qquad\quad
	+ \sup_{0\le\tau\le t} \norm{w^n(\tau)}_{L^2(\Om)}^2
	\\ & \quad \le
	\norm{v^n(0)}_{L^2(\Om)}^2 
	+ \sum_{j=i,e} \norm{\sqrt{\eps_n}u_j^n(0)}_{L^2(\Om)}^2
	+ \norm{w^n(0)}_{L^2(\Om)}^2
	\\ & \quad \qquad \quad
	+ C_1(1+t) + C_1\int_0^t \norm{v^n(s)}_{L^2(\Om)}^2 \ds
	+ C_1\int_0^t \norm{w^n(s)}_{L^2(\Om)}^2 \ds
	\\ & \quad \qquad\quad
	+C_1\sup_{0\le \tau\le t}\abs{\int_0^{\tau}\int_{\Om} 
	v^n \eta^n(v^n) \dx \dW^{v,n}(s)}
	\\ &  \quad\qquad\quad
	+C_1 \sup_{0\le \tau\le t}\abs{\int_0^{\tau}
	\int_{\Om} w^n \sigma^n(v^n) \dx \dW^{w,n}(s)},
\end{align*}
for some constant $C_1$ independent of $n$. 

We raise both sides of this inequality to the 
power $q_0/2$, take the expectation, and apply 
several elementary inequalities, eventually arriving at
\begin{equation}\label{eq:pmoment}
	\begin{split}
		&\E\left[ \,\sup_{0\le \tau\le t}
		\norm{v^n(\tau)}_{L^2(\Om)}^{q_0}\right] 
		+\sum_{j=i,e}\E\left[ \, \sup_{0\le \tau\le t}
		\norm{\sqrt{\eps_n}u_i^n(\tau)}_{L^2(\Om)}^{q_0}\right]
		\\ &\quad \qquad
		+ \E\left[ \,\sup_{0\le\tau\le t} 
		\norm{w^n(\tau)}_{L^2(\Om)}^{q_0}\right]
		\\ & \quad \le 
		C_2\E\left[ \norm{v^n(0)}_{L^2(\Om)}^{q_0}\right] 
		+C_2\sum_{j=i,e}\E\left[ 
		\norm{\sqrt{\eps_n}u_i^n(0)}_{L^2(\Om)}^{q_0}\right]
		\\ & \quad \qquad
		+ C_2\E\left[\norm{w^n(0)}_{L^2(\Om)}^{q_0}\right]
		+C_2 \left (1+t\right )^{\frac{q_0}{2}}
		\\ & \quad \qquad
		+ C_2\int_0^t \norm{v^n(s)}_{L^2(\Om)}^{q_0} \ds
		+C_2\int_0^t \norm{w^n(s)}_{L^2(\Om)}^{q_0} \ds
		+ \Gamma_\eta+\Gamma_\sigma,
	\end{split}
\end{equation}
where
\begin{align*}
	\Gamma_\eta & := \E\left[\, \sup_{0\le \tau\le t}\abs{\int_0^{\tau}
	\int_{\Om} v^n \eta^n(v^n) \dx \dW^{v,n}(s)}^{\frac{q_0}{2}}\,\right],
	\\ 
	\Gamma_\sigma & := \E\left[\, \sup_{0\le \tau\le t}\abs{\int_0^{\tau}
	\int_{\Om} w^n \sigma^n(v^n) \dx \dW^{w,n}(s)}^{\frac{q_0}{2}}\, \right].
\end{align*}

Arguing as in \eqref{Esup2}, using a martingale inequality (and some
elementary inequalities),
\begin{equation}\label{Esup-vq0}
	\begin{split}
		\Gamma_\eta & \le C_3 \E \left[ \left(\int_0^t \sum_{k=1}^n 
		\abs{\int_{\Om} v^n \eta^n_k(v^n) \dx}^2 \ds\right)^{\frac{q_0}{4}} \right]
		\\ & \le C_3 \E \left[\left( \int_0^t  
		\left(\int_{\Om} \abs{v^n}^2\dx\right) 
		\left(\sum_{k=1}^n\int_{\Om} \abs{\eta^n_k(v^n)}^2 \dx 
		\right)\ds\right)^{\frac{q_0}{4}} \right] \\ & \le C_3 \E \left[
		\left(\sup_{\tau\in [0,t]}\int_{\Om} \abs{v^n}^2\dx\right)^{\frac{q_0}{4}}
		\left( \int_0^t\sum_{k=1}^n\int_{\Om} 
		\abs{\eta^n_k(v^n)}^2 \dx\ds \right)^{\frac{q_0}{4}}\right]
		\\ & \le \delta
		\E \left[\left(\sup_{\tau\in [0,t]}\int_{\Om} 
		\abs{v^n}^2\dx\right)^{\frac{q_0}{2}}\right]
		\\ & \qquad\qquad
		+C_4(\delta)\E\left[\left( \int_0^t\sum_{k=1}^n\int_{\Om} 
		\abs{\eta^n_k(v^n)}^2 \dx\ds \right)^{\frac{q_0}{2}}\right]
		\\ & \le \delta 
		\E \left[ \sup_{\tau\in [0,t]} \norm{v^n(\tau)}_{L^2(\Om)}^{q_0}\right]
		+C_5 \E \left[\int_0^t \norm{v^n(s)}_{L^2(\Om)}^{q_0}\ds \right]+C_6,
	\end{split}
\end{equation}
for any $\delta>0$, where we have used that
\begin{align*}
	\E\left[\left( \int_0^t\sum_{k=1}^n\int_{\Om} 
	\abs{\eta^n_k(v^n)}^2 \dx\ds \right)^{\frac{q_0}{2}}\right]
	&\overset{\eqref{eq:noise-cond}}{\le} 
	\hat C_4\E\left[ \left(\int_0^t\int_{\Om} \abs{v^n}^2 \dx\ds 
	+ t\abs{\Om} \right)^{\frac{q_0}{2}}\right]
	\\ & \le \hat C_5
	\E \left[\int_0^t \int_{\Om} \abs{v^n}^{q_0}\dx\ds \right] + \hat C_6.
\end{align*}
Similarly, relying on \eqref{eq:noise-cond},
\begin{equation}\label{Esup-wq0}
	\begin{split}
		\Gamma_\sigma &\le \delta 
		\E \left[ \sup_{\tau\in [0,t]} \norm{w^n(\tau)}_{L^2(\Om)}^{q_0}\dx \, \right]
		+C_7 \E \left[\int_0^t \norm{v^n}_{L^2(\Om)}^{q_0}\ds \right] + C_8.
	\end{split}
\end{equation}

With $\delta$ chosen small, combining \eqref{Esup-vq0} 
and \eqref{Esup-wq0} in \eqref{eq:pmoment} gives
\begin{equation}\label{eq:pmoment2}
	\begin{split}
		&\E\left[ \,\sup_{0\le \tau\le t}\norm{v^n(\tau)}_{L^2(\Om)}^{q_0} \, \right] 
		+\sum_{j=i,e}\E\left[ \, \sup_{0\le \tau\le t}
		\norm{\sqrt{\eps_n}u_i^n(\tau)}_{L^2(\Om)}^{q_0} \, \right]
		\\ &\quad \qquad
		+\E\left[ \,\sup_{0\le\tau\le t} \norm{w^n(\tau)}_{L^2(\Om)}^{q_0} \, \right]
		\\ & \quad \le 
		C_9\E\left[ \norm{v^n(0)}_{L^2(\Om)}^{q_0}\right] 
		+C_9\sum_{j=i,e}\E\left[ \norm{\sqrt{\eps_n}u_i^n(0)}_{L^2(\Om)}^{q_0}\right]
		\\ & \quad \qquad
		+ C_9\E\left[\norm{w^n(0)}_{L^2(\Om)}^{q_0}\right] +C_9
		+C_9\int_0^t \E \left[\norm{v^n(s)}_{L^2(\Om)}^{q_0} \ds\right],
	\end{split}
\end{equation}
for some constant $C_9>0$ independent of $n$. Set
\begin{align*}
	\Gamma(t) & :=
	\E\left[ \,\sup_{0\le \tau\le t}\norm{v^n(\tau)}_{L^2(\Om)}^{q_0}\, \right] 
	+\sum_{j=i,e}\E\left[ \, \sup_{0\le \tau\le t}
	\norm{\sqrt{\eps_n}u_i^n(\tau)}_{L^2(\Om)}^{q_0}\, \right]
	\\ &\quad \qquad
	+\E\left[ \,\sup_{0\le\tau\le t} \norm{w^n(\tau)}_{L^2(\Om)}^{q_0}\, \right],
\end{align*}
and note that \eqref{eq:pmoment2} reads
$$
\Gamma(t) \le C_9 \Gamma(0)+C_9+C_9\int_0^t \Gamma(s)\ds , \qquad t\in [0,T].
$$
Now an application of Gr\"onwall's inequality 
yields the desired result \eqref{eq:Lq0-est}.

Finally, we can use \eqref{eq:v+w-2}, \eqref{Esup-vq0}, 
\eqref{Esup-wq0}, and \eqref{eq:Lq0-est} to conclude that
$$
\sum_{j=i,e}
\E\left[ \abs{\int_0^t \int_{\Om}  \abs{\Grad u_i^n}^2 \dx \ds}^{\frac{q_0}{2}}\right] 
+ \E\left[\abs{\int_0^t \int_{\Om}  \abs{v^n}^4 \dx \ds}^{\frac{q_0}{2}}\right] 
\le C_{10},
$$
and \eqref{eq:Lq0-est} follows.
\end{proof}

\subsection{Temporal translation estimates}\label{sec:translation-est}
To secure strong $L^2_{t,x}$ compactness of the 
Faedo-Galerkin solutions, via a standard Aubin-Lions-Simon 
compactness lemma, we need to come up with $n$-independent 
temporal translation estimates. 

\begin{lem}\label{Space-Time-translate}
Suppose conditions {\rm (\textbf{GFHN})}, \eqref{matrix}, 
\eqref{eq:noise-cond}, and \eqref{eq:uie-init-ass} hold. 
Let 
$$
u_i^n(t),u_e^n(t),v^n(t),w^n(t),  \quad t\in [0,T],
$$
satisfy \eqref{S1-Galerkin-new}, \eqref{eq:def-eps}, 
\eqref{eq:Wn}, \eqref{Galerkin-initdata}.
With $u^n=v^n$ or $w^n$, there is a 
constant $C>0$, independent of $n$, such that
\begin{equation}\label{Time-translate}
	\E \left[\, \sup_{0\le \tau\le \delta} \int_0^{T-\tau}\int_{\Om} 
	\abs{u^n(t+\tau,x)-u^n(t,x)}^2 \dx \dt \, \right] 
	\le C \delta^{\frac14}, 
\end{equation}
for any sufficiently small $\delta>0$.
\end{lem}

\begin{proof}
We assume that $v^n, u_i^n, u_e^n, w^n$ and 
$\eta^n,\sigma^n$ have been extended by zero 
outside the time interval $[0, T]$. Recalling 
\eqref{eq:approxsol} (i.e., $v^n=u_i^n-u_e^n$), it follows that
\begin{align*}
	\Gamma_{ie}(t) & := \int_{\Om} \abs{v^n(t+\tau,x)-v^n(t,x)}^2 \dx
	\\ & \qquad\quad + \eps_n \sum_{j=i,e} \int_{\Om} 
	\abs{u_j^n(t+\tau,x)-u_j^n(t,x)}^2 \dx\\
	& = \int_{\Om} \left(u_i^n(t+\tau,x)-u_i^n(t,x)\right) 
	\left( \, \int_t^{t+ \tau} d \left(v^n (s,x)+\eps_n u_i(s,x)\right) \right)\dx	
	\\ & \qquad \quad
	- \int_{\Om} \left(u_e^n(t+\tau,x)-u_e^n(t,x)\right) 
	\left(\, \int_t^{t+ \tau} d \left(v^n (s,x)-\eps_n u_e(s,x)\right) \right)\dx	.
\end{align*}
In view of \eqref{S1-Galerkin-bis}, see also \eqref{eq:approx-eqn-integrated},
\begin{align*}
	\Gamma_{ie}(t) 
	& =-\sum_{j=i,e} \int_{\Om} \left( \, \int_t^{t+ \tau}
	M_j(x) \Grad u_j^n(s,x)\ds \right)\cdot 
	\Grad \left(u_j^n(t+\tau,x)-u_j^n(t,x)\right) \dx
	\\ &\qquad \quad 
	- \int_{\Om} \left( \, \int_t^{t+ \tau} I\left(v^n(s,x),w^n(s,x)\right)\ds\right)
	\left(v^n(t+\tau,x)-v^n(t,x)\right)\dx 
	\\ & \qquad \quad
	+\int_{\Om} \left(\, \int_t^{t+ \tau} 
	\eta^n(v^n(s,x)) \dW^{v,n}(s)\right)\left(v^n(t+\tau,x)-v^n(t,x)\right) \dx.
\end{align*}
Similarly, using the equation for $w^n$, cf.~\eqref{S1-Galerkin-bis} 
and also \eqref{eq:nondegen},
\begin{align*}
	 \Gamma_w(t) & :=\int_{\Om} \abs{w^n(t+\tau,x)-w^n(t,x)}^2 \dx
	 \\ & 
	 = \int_{\Om}\left(\, \int_t^{t+ \tau} H\left(v^n(s,x),w^n(s,x)\right)\ds\right) 
	 \left(w^n(t+\tau,x)-w^n(t,x)\right) \dx
	 \\ & \qquad 
	 + \int_{\Om}\left(\, \int_t^{t+ \tau}\sigma^n(v^n(s,x))\dW^{v,n}(s)\right)
	 \left(w^n(t+\tau,x)-w^n(t,x)\right)\dx.
\end{align*}
Integrating over $t\in (0,T-\tau)$ and summing the resulting equations gives
\begin{equation}\label{eq:Iiew}
	\int_0^{T-\tau} \Gamma_{ie}(t) \dt + \int_0^{T-\tau} \Gamma_{w}(t) \dt 
	= \Gamma_1+\Gamma_2+\Gamma_3+\Gamma_4+\Gamma_5,
\end{equation}
where
\begin{align*}
	\Gamma_1 & := -\sum_{j=i,e}\int_0^{T-\tau}
	\int_{\Om}\left(\, \int_{t}^{t+ \tau} 
	M_j(x) \Grad u_j^n(s,x)\ds\right) 
	\\ & \qquad\qquad\qquad\qquad\qquad\qquad 
	\cdot \Grad \left(u_j^n(t+\tau,x)-u_j^n(t,x)\right)\dx\dt
	\\ \Gamma_2 & := -\int_0^{T-\tau}\int_{\Om}
	\left(\, \int_{t}^{t+ \tau} I\left(v^n(s,x),w^n(s,x)\right)\ds\right) 
	\\ & \qquad\qquad\qquad\qquad\qquad\qquad 
	\times \left(v^n(t+\tau,x)-v^n(t,x)\right) \dx\dt
	\\ \Gamma_3 & := \int_0^{T-\tau}\int_{\Om}
	\left(\, \int_{t}^{t+ \tau}H\left(v^n(s,x),w^n(s,x)\right)\ds \right)
	\\ & \qquad\qquad\qquad\qquad\qquad\qquad 
	\times \left(w^n(t+\tau,x)-w^n(t,x)\right) \dx\dt
	\\  \Gamma_4 & := \int_0^{T-\tau}\int_{\Om}
	\left(\, \int_t^{t+ \tau} \eta^n(v^n(s,x))\dW^{v,n}(s)\right)
	\\ & \qquad\qquad\qquad\qquad\qquad\qquad 
	\times \left(v^n(t+\tau,x)-v^n(t,x)\right) \dx\dt
	\\ \Gamma_5& :=\int_0^{T-\tau} \int_{\Om}
	\left(\, \int_t^{t+ \tau} \sigma^n(v^n(s,x)) \dW^{v,n}(s)\right)
	\\ & \qquad\qquad\qquad\qquad\qquad\qquad 
	\times \left(w^n(t+\tau,x)-w^n(t,x)\right)\dx \dt.
\end{align*}
We examine these six terms separately. 
For the $\Gamma_1$ term, noting that
$$
\abs{\int_{t}^{t+ \tau} 
M_j(x) \Grad u_j^n(s,x)\ds}^2\le M\tau \int_{t}^{t+ \tau} 
\abs{\Grad u_j^n(s,x)}^2\ds,
$$
thanks to \eqref{matrix}, we obtain
\begin{align*}
	\abs{ \Gamma_1} & \le \sqrt{M \tau} \sum_{j=i,e} 
	\left ( \int_0^{T-\tau}  
	\int_t^{t+ \tau} \int_{\Om} 
	\abs{\Grad u_j^n(s,x)}^2\dx \ds \dt \right)^{\frac{1}{2}}
	\\ & \quad \quad\quad\quad
	\times \left ( \int_0^{T-\tau} 
	\int_{\Om} \abs{\Grad \left(u_j^n(t+\tau,x)-u_j^n(t,x)\right)}^2 
	\dx \dt\right)^{\frac{1}{2}},
\end{align*}
using Cauchy-Schwarz's inequality. 
Hence, by Young's inequality and \eqref{Gal:est2},
\begin{equation}\label{est-I1}
	\E\left[\sup_{0\le \tau\le \delta}\abs{ \Gamma_1}\right] 
	\le C_1\sqrt{\delta}, 
\end{equation}
for some constant $C_1>0$ independent of $n$. 

Next, take notice of the bound
\begin{equation}\label{eq:I-bound}
	\begin{split}
		&\abs{\int_{t}^{t+ \tau} I\left(v^n(s,x),w^n(s,x)\right)\ds}^{\frac43} 
		\le \tau^{\frac13} \int_{t}^{t+ \tau}
		\abs{I\left(v^n(s,x),w^n(s,x)\right)}^{\frac43}\ds
		\\ & \qquad\quad 
		\le  C_2 \tau^{\frac13} \int_{t}^{t+\tau} 
		\left(1 +\abs{v(s,x)}^4+\abs{w(s,x)}^2\right)\ds,
	\end{split}
\end{equation}
where we have used the inequality
\begin{equation}\label{eq:I43-est}
	\abs{I(v,w)}^{\frac43} \le C_2 \left(1 +\abs{v}^4+\abs{w}^2\right), 
\end{equation}
resulting from {\rm (\textbf{GFHN})} and Young's inequality. 
Due to \eqref{eq:I-bound}, \eqref{Gal:est1} and \eqref{Gal:est2},
\begin{align*}
	\abs{ \Gamma_2}
	& \le C_3 \tau^{\frac14}
	\left ( \int_0^{T-\tau}  
	\int_t^{t+ \tau} \int_{\Om} 
	\left(1 +\abs{v(s,x)}^4+\abs{w(s,x)}^2\right)
	\dx \ds \dt \right)^{\frac{3}{4}}
	\\ & \quad \quad\quad\quad
	\times \left ( \int_0^{T-\tau} 
	\int_{\Om} \abs{v^n(t+\tau,x)-v^n(t,x)}^4 
	\dx \dt \right)^{\frac{1}{4}},
\end{align*}
and for this reason, in view of Young's inequality 
and \eqref{Gal:est2},
\begin{equation}\label{est-I3}
	\E\left[\sup_{0\le \tau\le \delta}\abs{ \Gamma_2}\right] 
	\le C_4\delta^{\frac14}.
\end{equation}

Similarly, since $\abs{H(v,w)}^2 \le C_5 
\left(1 +\abs{v}^4+\abs{w}^2\right)$, cf.~{\rm (\textbf{GFHN})}, we obtain
\begin{equation}\label{est-I4}
	\E\left[\sup_{0\le \tau\le \delta}\abs{ \Gamma_3}\right] 
	\le C_6 \delta^{\frac12}.
\end{equation}

Finally, we treat the stochastic terms. 
By the Cauchy-Schwarz inequality,
\begin{align*}
	\abs{ \Gamma_4}
	& \le \left (\int_0^T \int_{\Om} 
	\sup_{0\le \tau\le \delta}
	\abs{\int_t^{t+ \tau} \eta^n(v^n(s,x)) \dW^{v,n}(s)}^2
	\dx \dt \right)^{\frac{1}{2}}
	\\ &  \quad\quad\quad
	\times 
	\left ( \int_0^T \sup_{0\le \tau\le \delta}
	\int_{\Om} \abs{v^n(t+\tau,x)-v^n(t,x)}^2 
	\dx \dt \right)^{\frac{1}{2}}.
\end{align*}
Applying $\E[\cdot]$ along with the Cauchy-Schwarz 
inequality, we gather the estimate
\begin{equation}\label{est-I5}
	\begin{split}
		\E\left[\sup_{0\le \tau\le \delta}\abs{ \Gamma_4} \right]
		&\le  \left(\E\left [\int_0^T  
		\int_{\Om}
		\sup_{0\le \tau\le \delta}
		\abs{\int_t^{t+ \tau} \eta^n(v^n(s,x)) \dW^{v,n}(s)}^2
		\dx \dt \right]\right)^{\frac{1}{2}}
		\\ &  \quad
		\times \left (\E \left[\sup_{0\le \tau\le \delta} 
		\int_0^{T-\tau} \int_{\Om} \abs{v^n(t+\tau,x)-v^n(t,x)}^2 
		\dx \dt \right] \right)^{\frac{1}{2}}
		\\ & \le C_7 \left ( \E \left[\,  \int_0^T
		\int_t^{t+ \delta} \sum_{k=1}^n\int_{\Om} 
		\abs{\eta^n_k(v^n(s,x))}^2 \dx \ds \dt \right] \right)^{\frac{1}{2}}
		\\ & \le 
		C_8 \left ( \E \left[\,  \int_0^T\int_t^{t+ \delta}\int_{\Om} 
		\left(1+\abs{v^n(s,x)}^2\right) \dx \ds \dt \right] \right)^{\frac{1}{2}}
		\le  C_9 \delta^{\frac12},
	\end{split}
\end{equation}
where we have also used the Burkholder-Davis-Gundy 
inequality \eqref{eq:bdg} and \eqref{eq:noise-cond}, \eqref{Gal:est3}.

Similarly, 
\begin{equation}\label{est-I6}
	\E\left[\sup_{0\le \tau\le \delta}\abs{ \Gamma_5} \right]
	\le C_{10} \delta^{\frac12}.
\end{equation}

Collecting the previous estimates \eqref{est-I1}, \eqref{est-I3},  \eqref{est-I4}, 
\eqref{est-I5}, and \eqref{est-I6} we readily conclude from \eqref{eq:Iiew} that 
the time translation estimate \eqref{Time-translate} holds.
\end{proof}

\begin{rem}
The proof of \eqref{Time-translate}, cf.~estimate \eqref{est-I4}, 
reveals that the amount of time continuity of $w^n$ is actually 
better than stated; it is of order $\delta^{\frac12}$.
\end{rem}

\subsection{Tightness and a.s.~representations}\label{subsec:tight}
To justify passing to the limit in the nonlinear terms in \eqref{eq:nondegen}, we must 
show that $\Set{v^n}_{n\ge1}$ converges strongly, thereby 
upgrading the weak $L^2$ convergence in \eqref{eq:weakconv1}.  
Strong $(t,x)$ convergence is a result of the spatial $\tH$ 
bound \eqref{Gal:est2} and the time translation estimate \eqref{Time-translate}. 

On the other hand, to secure strong (a.s.)~convergence in the probability 
variable $\omega\in D$ we must invoke some nontrivial results of 
Skorokhod, linked to tightness of probability measures 
and a.s.~representations of random variables.  
Actually, there is a complicating factor at play here, namely that the
sequences $\Set{u_i^n}_{n\ge1}$, $\Set{u_e^n}_{n\ge1}$ only converge 
weakly in $(t,x)$ because of the degenerate structure of the bidomain model. 
As a result, we must turn to the Skorokhod-Jakubowski 
representation theorem \cite{Jakubowski:1997aa}, 
which applies to separable Banach spaces equipped with the weak topology and 
other so-called quasi-Polish spaces. At variance with the original Skorokhod 
representations on Polish spaces, the flexibility of the Jakubowski version 
comes at the expense of having to pass to a subsequence (which may be satisfactory 
in many situations). We refer to \cite{Breit:2016aa,Brzezniak:2013aa,Brzezniak:2013ab,Brzezniak:2011aa,Ondrejat:2010aa,Smith:2015aa} for works 
making use of Skorokhod-Jakubowski a.s.~representations.

Following \cite{Bensoussan:1995aa,Mohammed:2015ab} (for example), the 
aim is to establish tightness of the probability measures (laws) generated 
by the Faedo-Galerkin solutions $\Set{\left(U^n,W^n,U_0^n\right)}_{n\ge 1}$, where
\begin{equation}\label{eq:def-U-W-U0}
	U^n=u_i^n,u_e^n,v^n,w^n, \quad
	W^n = W^{v,n},W^{w,n}, \quad
	U_0^n = u_{i,0}^n,u_{e,0}^n,v_0^n,w_0^n.
\end{equation}
Accordingly, we choose the following path space for these measures:
\begin{align*}
	\cX & :=\Bigl [\left(L^2((0,T);\tH(\Om_T))\text{--weak}\right)^2
	\times L^2(\Om_T)
	\times L^2((0,T);(\tH(\Om_T))^*)\Bigr]
	\\ & \qquad\quad
	\times\Bigl [\left(C([0,T]; \U_0)\right)^2\Bigr]
	\times \Bigl [\left(L^2(\Om)\right)^4\Bigr],
	\\ & =: \cX_U\times \cX_W\times \cX_{U_0},
\end{align*}
where $\U_0$ is defined in \eqref{eq:U0} and the tag ``--weak" 
signifies that the space is equipped with the weak topology. 
The $\sigma$-algebra of Borel subsets of $\cX$ 
is denoted by $\cB(\cX)$. We introduce the 
$\left(\cX,\cB(\cX)\right)$-valued measurable mapping $\Phi_n$ 
defined on $\left(D,\cF,P\right)$ by
$$
\Phi_n(\omega)
= \left(U^n(\omega),W^n, U_0^n(\omega)\right).
$$
On $\left(\cX,\cB(\cX)\right)$, we define the 
probability measure (law of of $\Phi_n$)
\begin{equation}\label{eq:Ln-law}
	\cL_n(\cA)=P\left(\Phi_n^{-1}(\cA)\right), 
	\qquad \cA\in \cB(\cX).
\end{equation}
We denote by $\cL_{u_i^n}, \cL_{u_e^n}$ the respective laws of 
$u_i^n, u_e^n$ on $L^2((0,T);\tH(\Om_T))\text{--weak}$, 
with similar notations for the laws of $v^n$ on $L^2(\Om_T)$, 
$w^n$ on $L^2((0,T);(\tH(\Om))^*)$, $W^{v,n},W^{w,n}$ 
on $C([0,T]; \U_0)$, and $u_{i,0}^n,u_{e,0}^n,v_0^n,w_0^n$ 
on $L^2(\Om)$. Hence,
$$
\cL_n = \cL_{u_i^n}\times \cL_{u_e^n}\times \cL_{v^n}\times \cL_{w^n}
\times  \cL_{u_{i,0}^n}\times \cL_{u_{e,0}^n}
\times\cL_{v_0^n}\times\cL_{w_0^n}.
$$

Inspired by \cite{Bensoussan:1995aa}, for any two sequences of 
positive numbers $r_m,\nu_m$ tending to 
zero as $m\to \infty$, we introduce the set
\begin{align*}
	\cZ^v_{r_m,\nu_m}
	:=\Biggl\{ & u \in L^\infty\left((0,T);L^2(\Om)\right)\cap
	L^2((0,T);\tH(\Om)):
	\\ & \qquad
	\sup_{m\ge 1} \frac{1}{\nu_m}\sup_{0\le \tau\le r_m} 
	\norm{u(\cdot+\tau)-u}_{L^2((0,T-\tau);L^2(\Om))}<\infty \Biggr\}.
\end{align*}
Then $\cZ^v_{r_m,\nu_m}$ is a Banach 
space under the natural norm
\begin{align*}
	\norm{u}_{\cZ^v_{r_m,\nu_m}}
	:= &  \norm{u}_{L^\infty\left((0,T);L^2(\Om)\right)}
	+\norm{u}_{L^2((0,T);\tH(\Om))}
	\\ & \qquad 
	+ \sup_{m\ge 1} \frac{1}{\nu_m}\sup_{0\le \tau\le r_m} 
	\norm{u(\cdot+\tau)-u}_{L^2((0,T-\tau);L^2(\Om))}.
\end{align*}
Moreover, $\cZ^v_{r_m,\nu_m}$  is compactly 
embedded in $L^2(\Om_T)$, which is a 
consequence of an Aubin-Lions-Simon lemma. 
Suppose $X_1\subset X_0$ are two Banach spaces, where 
$X_1$ is compactly embedded in $X_0$. 
Let $\cZ\subset L^p((0,T);X_0)$, where $1\le p\le \infty$. 
Simon \cite{Simon:1987vn} provides several results 
ensuring the compactness of $\cZ$ in $L^p((0,T);X_0)$ (and in 
$C([0,T];X_0)$ if $p=\infty$). For example, by assuming that $\cZ$ 
is bounded in $L^1_{\loc}((0,T);X_1)$ and 
$\norm{u(\cdot+\tau)-u}_{L^p((0,T-\tau);X_0)}\to 0$ as $\tau\to 0$, 
uniformly for $u\in \cZ$ \cite[Theorem 3]{Simon:1987vn}. Another 
result \cite[Theorem 5]{Simon:1987vn} concerns functions taking values in 
a third Banach space $X_{-1}$, where $X_1\subset X_0\subset X_{-1}$ 
and $X_1$ is still compactly embedded in $X_0$.
Compactness of $\cZ$ in $L^p((0,T);X_0)$ follows once we know that 
$\cZ$ is bounded in $L^p((0,T);X_1)$ and 
$\norm{u(\cdot+\tau)-u}_{L^p((0,T-\tau);X_{-1})}\to 0$ as $\tau\to 0$, 
uniformly for $u\in \cZ$. 

The space $\cZ^v_{r_m,\nu_m}$ is relevant 
for $v^n$, while for $w^n$ we utilize
\begin{align*}
	\cZ^w_{r_m,\nu_m}
	:=\Biggl\{ & u \in L^\infty\left((0,T);L^2(\Om)\right):
	\\ & \qquad
	\sup_{m\ge 1} \frac{1}{\nu_m}\sup_{0\le \tau \le r_m} 
	\norm{u(\cdot+\tau)-u}_{L^2((0,T-\tau);(\tH(\Om))^*)}<\infty \Biggr\},
\end{align*}
with a corresponding natural norm $\norm{u}_{\cZ^{w}_{r_m,\nu_m}}$. 
Besides, $\cZ^v_{r_m,\nu_m}$  is compactly 
embedded in $L^2((0,T);(\tH(\Om))^*)$.

\begin{lem}[tightness of laws \eqref{eq:Ln-law} for the 
Faedo-Galerkin approximations]\label{lem:tight}
Equipped with the estimates in Lemmas \ref{lem:apriori-est} 
and \ref{Space-Time-translate}, the laws 
$\Set{\cL_n}_{\ge1}$ is tight on $\left(\cX,\cB(\cX)\right)$.
\end{lem}

\begin{proof}
Given any $\delta>0$, we need to produce compact sets
\begin{align*}
	& \cK_{0,\delta}\subset L^2((0,T);\tH(\Om))\text{--weak},
	\\ & \cK_{1,\delta}\subset L^2(\Om_T), \quad
	\cK_{2,\delta} \subset L^2((0,T);(\tH(\Om))^*),
	\\ &
	\cK_{3,\delta} \subset C([0,T];\U_0),
	\quad
	\cK_{4,\delta} \subset L^2(\Om),
\end{align*}
such that, with 
$\cK_{\delta}=\left(\cK_{0,\delta}\right)^2
\times \cK_{1,\delta}\times \cK_{2,\delta}\times \left(\cK_{3,\delta}\right)^2
\times \left(\cK_{4,\delta}\right)^4$,
$$
\cL_n\left(\cK_{\delta}\right)
=P\left(\Set{\omega\in D:\Phi_n(\omega) 
\in \cK_{\delta}} \right)> 1-\delta.
$$
This inequality follows if we can show that
\begin{align}
	& \cL_{u^n}\left(\cK_{0,\delta}^c\right)
	=P\left(\Set{\omega\in D: u^n(\omega)\notin \cK_{0,\delta}} \right)
	\le \frac{\delta}{10},
	\quad u^n=u_i^n, u_e^n,
	\label{eq:law0}
	\\ & 
	\cL_{v^n}\left(\cK_{1,\delta}^c\right)
	=P\left(\Set{\omega\in D: v^n(\omega)\notin \cK_{1,\delta}} \right)
	\le \frac{\delta}{10},
	\label{eq:law1}
	\\ &
	\cL_{w^n}\left(\cK_{2,\delta}^c\right)=
	P\left(\Set{\omega\in D: w^n(\omega)\notin \cK_{2,\delta}} \right) 
	\le \frac{\delta}{10}, 
	\label{eq:law2}
	\\ &
	\cL_{W^n}\left(\cK_{3,\delta}^c\right)
	=P\left(\Set{\omega\in D: W^n(\omega)\notin \cK_{3,\delta}} \right) 
	\le \frac{\delta}{10}, 
	\quad W^n=W^{v,n},W^{w,n},
	\label{eq:law3}
	\\ &
	\cL_{u_0^n}\left(\cK_{4,\delta}^c\right)
	=P\left(\Set{\omega\in D: U_0^n(\omega)\notin \cK_{4,\delta}} \right) 
	\le \frac{\delta}{10}, 
	\quad U_0^n=u_{i,0}^n,u_{e,0}^n,v_0^n,w_0^n.
	\label{eq:law4}
\end{align}

By weak compactness of bounded sets in $L^2((0,T);\tH(\Om))$, the set
$$
\cK^{0,\delta} := \Set{u: \norm{u}_{L^2((0,T);\tH(\Om))}\le R_{0,\delta}},
$$
is a compact subset of $L^2((0,T);\tH(\Om))\text{--weak}$, where 
$R_{0,\delta}>0$ is to be determined later. 
Recalling the Chebyshev inequality for a 
nonnegative random variable $\xi$,
\begin{equation}\label{eq:Chebyshev}
	P\left(\Set{\omega\in D: \xi(\omega)\ge R} \right)
	\le \frac{E\left[\xi^k\right]}{R^k},
	\qquad R,k>0,
\end{equation}
it follows that
\begin{align*}
	P\left(\Set{\omega\in D: u^n(\omega)\notin\cK^{0,\delta}} \right)
	& = P\left(\Set{\omega\in D: 
	\norm{u^n(\omega)}_{L^2((0,T);\tH(\Om))}>R_{0,\delta}}\right)
	\\ & \le \frac{1}{R_{0,\delta}}
	\E\left[\norm{u^n(\omega)}_{L^2((0,T);\tH(\Om))}\right]
	\le \frac{C}{R_{0,\delta}}.
\end{align*}
To derive the last inequality we used the Cauchy-Schwarz 
inequality and then \eqref{Gal:est2}. Clearly, we can 
choose $R_{0,\delta}>0$ such that \eqref{eq:law0} holds.

We fix two sequences $\Set{r_m}_{m=1}^\infty$, 
$\Set{\nu_m}_{m=1}^\infty$ of positive numbers 
numbers tending to zero as $m\to \infty$ (independently of $n$), such that
\begin{equation}\label{eq:r-nu}
	\sum_{m=1}^\infty r_m^{\frac18}/\nu_m<\infty,
\end{equation}
and define
$$
\cK^{1,\delta} := 
\Set{u: \norm{u}_{\cZ^v_{r_m,\nu_m}}\le R_{1,\delta}},
$$
for a number $R_{1,\delta}>0$ that will be determined later. 
Evidently, in view of an Aubin-Lions-Simon 
lemma, $\cK^{1,\delta}$ is a compact subset of $L^2(\Om_T)$. 
We have
\begin{align*}
	&P\left(\Set{\omega\in D: v^n(\omega)\notin\cK^{1,\delta}} \right)
	\\ & \quad
	\le P\left(\Set{\omega\in D: 
	\norm{v^n(\omega)}_{L^\infty\left((0,T);L^2(\Om)\right)}>R_{1,\delta}}\right)
	\\ & \quad\qquad 
	+P\left(\Set{\omega\in D: 
	\norm{v^n(\omega)}_{L^2((0,T);\tH(\Om))}>R_{1,\delta}}\right)
	\\ & \quad\qquad +P\left(\Set{\omega\in D: 
	\sup_{0\le \tau \le r_m} 
	\norm{v^n(\cdot+\tau)-v^n}_{L^2((0,T-\tau);L^2(\Om))}
	>R_{1,\delta} \, \nu_m}\right)
	\\ & \quad 
	 =:P_{1,1}+P_{1,2}+P_{1,3} \quad \text{(for any $m\ge 1$)}.
\end{align*}
Again by the Chebyshev inequality \eqref{eq:Chebyshev}, we infer that
\begin{align*}
	P_{1,1} & \le \frac{1}{R_{1,\delta}}
	\E \left[\norm{v^n(\omega)}_{L^\infty\left((0,T);L^2(\Om)\right)}\right]
	\le \frac{C}{R_{1,\delta}}, 
	\\ P_{1,2} & \le \frac{1}{R_{1,\delta}}
	\E\left[\norm{v^n(\omega)}_{L^2((0,T);\tH(\Om))}\right]
	\le \frac{C}{R_{1,\delta}},
	\\ P_{1,3} & \le \sum_{m=1}^\infty \frac{1}{R_{1,\delta} \, \nu_m}
	\E\left[\, \sup_{0\le \tau \le r_m} 
	\norm{v^n(\cdot+\tau)-v^n}_{L^2((0,T-\tau);L^2(\Om))}\right]
	\\ & \le \frac{C}{R_{1,\delta}}\sum_{m=1}^\infty \frac{r_m^{\frac18}}{\nu_m},
\end{align*}
where we have used \eqref{Gal:est2}, \eqref{Gal:est3}, and \eqref{Time-translate}. 
On the grounds of this and \eqref{eq:r-nu}, we can choose $R_\delta$ 
such that \eqref{eq:law1} holds.

Similarly, with sequences $\Set{r_m}_{m=1}^\infty$, 
$\Set{\nu_m}_{m=1}^\infty$ as above, define
$$ 
\cK^{2,\delta} := \Set{u: \norm{u}_{\cZ^w_{r_m,\nu_m}}\le R_{2,\delta}},
$$
for a number $R_{2,\delta}>0$ to be determined later.  By an 
Aubin-Lions-Simon lemma, $\cK^{2,\delta}$ is 
a compact subset of $L^2((0,T);(\tH(\Om))^*)$. We have
\begin{align*}
	&P\left(\Set{\omega\in D: w^n(\omega)\notin\cK^{2,\delta}} \right)
	\\ & \quad
	\le P\left(\Set{\omega\in D: 
	\norm{w^n(\omega)}_{L^\infty\left((0,T);L^2(\Om)\right)}>R_{2,\delta}}\right)
	\\ & \qquad\quad 
	+P\left(\Set{\omega\in D: 
	\sup_{0\le \tau \le r_m} 
	\norm{w^n(\cdot+\tau)-w^n}_{L^2((0,T-\tau);(\tH(\Om))^*)}
	>R_{2,\delta} \, \nu_m}\right)
	\\ & \quad 
	 =:P_{2,1}+P_{2,2}
	 \quad \text{(for any $m\ge 1$)},
\end{align*}
where, using \eqref{eq:Chebyshev} and \eqref{Gal:est3} as before,
\begin{align*}
	P_{2,1} & \le \frac{1}{R_\delta}
	\E \left[\norm{w^n(\omega)}_{L^\infty\left((0,T);L^2(\Om)\right)}\right]
	\le \frac{C}{R_{2,\delta}}, 
\end{align*}
and, via \eqref{Time-translate} and \eqref{eq:r-nu},
$$
P_{2,2} \le \sum_{m=1}^\infty \frac{1}{R_{2,\delta} \, \nu_m}
\E\left[\, \sup_{0\le \tau \le r_m} 
\norm{w^n(\cdot+\tau)-w^n}_{L^2((0,T-\tau);(\tH(\Om))^*)}\right]
\le \frac{C}{R_{2,\delta}}.
$$
Consequently, we can choose $R_{2,\delta}$ such 
that \eqref{eq:law2} holds.

Recall that the finite dimensional approximations 
$W^n=W^{v,n},W^{w,n}$, cf.~\eqref{eq:Wn}, are 
$P$-a.s.~convergent in $C([0,T];\U_0)$ as $n\to \infty$, and 
hence the laws $\cL_{W^n}$ converge weakly. Thanks to 
Theorem \ref{thm:prokhorov}, this entails the 
tightness of $\Set{\cL_{W^n}}_{n\ge 1}$, i.e., 
for any $\delta>0$, there exists a compact 
set $\cK_{3,\delta}$ in $C([0,T];\U_0)$ such 
that \eqref{eq:law3} holds. Similarly, as the finite dimensional 
approximations $u_{i,0}^n, u_{e,0}^n, v_0^n, w_0^n$, 
cf.~\eqref{Galerkin-initdata}, are $P$-a.s.~convergent in 
$L^2(\Om)$, the laws $\cL_{U_0^n}$ converge 
weakly ($\cL_{v_0^n}\weak \mu_{v_0}$, $\cL_{w_0^n}\weak \mu_{w_0}$). 
Hence, \eqref{eq:law4} follows. 
\end{proof}

\begin{lem}[Skorokhod-Jakubowski a.s.~representations]\label{lem:as-represent}
Passing to a subsequence (not relabeled), there exist a 
new probability space $(\tilde{D},\tilde{\cF}, \tilde{P})$ 
and new random variables
$$
\left(\tilde U^n, \tilde W^n, \tilde U_0^n\right),
\quad 
\left (\tilde U, \tilde W, \tilde U_0\right) ,
$$
where
\begin{equation}\label{eq:def-tvar}
	\begin{split}
		& \tilde U^n=\tilde u_i^n, \tilde u_e^n,\tilde v^n,\tilde w^n, \quad
		\tilde  W^n = \tilde W^{v,n},\tilde W^{w,n}, \quad
		\tilde U_0^n = \tilde u_{i,0}^n,\tilde u_{e,0}^n,\tilde v_0^n,\tilde w_0^n,
		\\ & 
		\tilde U=\tilde u_i, \tilde u_e,\tilde v,\tilde w, \quad
		\tilde  W = \tilde W^v,\tilde W^w, \quad
		\tilde U_0 = \tilde u_{i,0},\tilde u_{e,0},\tilde v_0,\tilde w_0,
	\end{split}
\end{equation}
with respective (joint) laws $\cL_n$ and $\cL$, such that the following 
strong convergences hold $\tilde P$-a.s.~as $n\to \infty$:
\begin{equation}\label{eq:strong-conv1}
	\begin{split}
		& \tilde v^n \to \tilde v  \quad  
		\text{in $L^2((0,T);L^2(\Om))$},
		\\ &  
		\tilde w^n \to \tilde w \quad  
		\text{in $L^2((0,T);(\tH(\Om))^*)$},
		\\ &
		\tilde W^{v,n} \to \tilde W^v, 
		\, 
		\tilde W^{w,n} \to \tilde W^w \quad  
		\text{in $C([0,T]; \U_0)$,}
		\\ & 
		\tilde u_{i,0}^n \to \tilde u_{i,0}, \,
		\tilde u_{e,0}^n\to \tilde u_{e,0}, \,
		\tilde v_0^n\to \tilde v_0, \, 
		\tilde w_0^n \to \tilde w_0
		\quad  \text{in $L^2(\Om)$.} 
	\end{split}
\end{equation}
Moreover, the following weak convergences hold 
$\tilde P$-a.s~as $n\to \infty$:
\begin{equation}\label{eq:weak-conv1}
	\tilde u^n_i \weak \tilde u_i, \,
	\tilde u^n_e \weak \tilde u_e  \quad  
	\text{in $L^2((0,T);\tH(\Om))$.}
\end{equation}
\end{lem}

\begin{proof}
Thanks to the Skorokhod-Jakubowski representation 
theorem (Theorem \ref{thm:skorokhod}), there exist
a new probability space $(\tilde{D},\tilde{\cF}, \tilde{P})$ 
and new $\cX$-valued random variables
\begin{equation}\label{def:tPhi}
	\begin{split}
		&\tilde \Phi_n = \left(\tilde u_i^n,\tilde u_e^n, \tilde v^n, 
		\tilde w^n,\tilde W^{v,n}, \tilde W^{w,n}, 
		\tilde u_{i,0}^n, \tilde u_{e,0}^n,\tilde v_0^n, \tilde w_0^n\right),
		\\ &
		\tilde \Phi = \left(\tilde u_i,\tilde u_e, \tilde v,\tilde w,
		\tilde W^v, \tilde W^w, 
		\tilde u_{i,0}, \tilde u_{e,0},\tilde v_0,\tilde w_0\right) 
	\end{split}
\end{equation}
on $(\tilde{D},\tilde{\cF}, \tilde{P})$, such that the law of 
$\tilde \Phi_n$ is $\cL_n$ and as $n\to \infty$,
\begin{equation}\label{eq:tPhi-conv}
	\tilde \Phi_n\to \tilde \Phi \quad  
	\text{$\tilde P$-almost surely (in $\cX$).}
\end{equation}
To be more accurate, the Skorokhod-Jakubowski theorem implies 
\eqref{def:tPhi}, \eqref{eq:tPhi-conv} along a subsequence, but (as usual) we do 
not to relabel the involved variables. 
Inasmuch as \eqref{eq:tPhi-conv} is a repackaging of 
\eqref{eq:strong-conv1}, \eqref{eq:weak-conv1}, 
this concludes the  proof.
\end{proof}

\begin{rem}
As mentioned before, since our path space $\cX$ is not 
a Polish space, we use Skorokhod-Jakubowski a.s.~representations 
\cite{Jakubowski:1997aa} instead of the classical Skorokhod 
theorem \cite{DaPrato:2014aa,Ikeda:1981aa}. 
For a proof that $L^2((0,T);\tH(\Om_T))\text{\rm --weak}$ (and thus $\cX$) is 
covered by the Skorokhod-Jakubowski theorem, see for 
example \cite[page 1645]{Brzezniak:2013aa}. 
\end{rem}

\begin{lem}[a priori estimates]
The a priori estimates in Lemma \ref{lem:apriori-est} 
continue to hold for the new random variables $\tilde u_i^n,\tilde u_e^n, 
\tilde v^n, \tilde w^n$ on $(\tilde{D},\tilde{\cF}, \tilde{P})$, that is,
\begin{equation}\label{eq:apriori-est-tilde}
	\left \{
	\begin{split}
		& \norm{\tilde u^n_j}_{L^2\left(\tilde D,\tilde \cF,\tilde P;L^2((0,T);\tH(\Om))\right)}
		\le C, \quad j=i,e,
		\\ &
		\norm{\sqrt{\eps_n} \tilde u^n_j}_{L^2
		\left(\tilde D,\tilde \cF,\tilde P;L^\infty((0,T);L^2(\Om))\right)}
		\le C, \quad j=i,e, 
		\\& 
		\norm{\tilde v^n}_{L^2\left(\tilde D,\tilde \cF, \tilde P;L^2((0,T);\tH(\Om))\right)}
		\le C,
		\\ & 
		\norm{\tilde v^n}_{L^2\left(\tilde D,\tilde \cF, \tilde P;L^\infty((0,T);L^2(\Om))\right)}
		\le C,
		\\ & 
		\norm{\tilde v^n}_{L^4\left(\tilde D,\tilde \cF,\tilde P;L^4(\Om_T)\right)}
		\le C,
		\\ & 
		\norm{\tilde w^n}_{L^2\left(\tilde D,\tilde \cF,\tilde P;L^\infty((0,T);L^2(\Om))\right)}
		\le C,
	\end{split}\right.
\end{equation}
for some $n$-independent constant $C>0$. The same applies 
to the estimates in Corollary \ref{cor:Lq0-est}, provided 
\eqref{eq:uie-init-ass-q0} holds. Namely,
\begin{align}
		& \norm{\left(\sqrt{\eps_n} \tilde u^n_i,
		\sqrt{\eps_n} \tilde u^n_e,\tilde v^n,\tilde w^n\right)}_{L^{q_0}
		\left(\tilde D,\tilde \cF,\tilde P;L^\infty((0,T);L^2(\Om))\right)}\le C,
		\label{eq:apriori-est-q0-tilde}
		\\  & \norm{\left(\Grad \tilde u_i^n,\Grad \tilde u_e^n
		\right)}_{L^{q_0}\left(\tilde D,\tilde \cF,\tilde P;L^2((0,T)\times \Om)\right)},
		\, 
		\norm{\tilde v^n}_{L^{2q_0}\left(\tilde D,\tilde \cF,\tilde P;
		L^4((0,T)\times \Om)\right)}\le C.\label{eq:apriori-est-q0-tilde2}
\end{align}
\end{lem}

\begin{proof}
Since the laws of $v^n$ and $\tilde v^n$ coincide 
and $\abs{\cdot}^2:=\norm{\cdot}_{L^\infty((0,T);L^2(\Om))}^2$ 
is bounded continuous on $B:=L^\infty((0,T);L^2(\Om))$ 
(so $\abs{\cdot}^2$ is measurable and $B$ is a Borel set in $\cX$), 
\begin{align*}
	\tilde \E \left[ \norm{\tilde v^n(t)}_{L^\infty((0,T);L^2(\Om))}^2\right]
	& =\int_B \abs{v}^2 \,d \cL_{\tilde v^n}(v)
	=\int_B \abs{v}^2 \,d \cL_{v^n}(v)
	\\ & 
	=\E \left[ \norm{v^n(t)}_{L^\infty((0,T);L^2(\Om))}^2\right]
	\overset{\eqref{Gal:est3}}{\le} C,
\end{align*}
where $\tilde \E[\cdot]$ is the expectation operator 
with respect to $(\tilde P,\tilde D)$; hence the fourth 
estimate in \eqref{eq:apriori-est-tilde} holds. As a matter of fact, by equality 
of the laws, all the estimates in Lemma \ref{lem:apriori-est} and 
Corollary \ref{cor:Lq0-est} hold for the corresponding ``tilde" functions 
defined on $(\tilde{D},\tilde{\cF}, \tilde{P})$. 
\end{proof}

Let us introduce the following stochastic basis 
linked to $\tilde \Phi_n$, cf.~\eqref{def:tPhi}:
\begin{equation}\label{eq:stoch-basis-n}
	\begin{split}
		& \tilde \cS_n=
		\left(\tilde D,\tilde \cF, \Set{\tilde \cF_t^n}_{t\in [0,T]},\tilde P,
		\tilde W^{v,n},\tilde W^{w,n}\right),
		\\ & 
		\tcFt^n = \sigma\left(\sigma \bigl(\tilde \Phi_n\big |_{[0,t]}\bigr)
		\bigcup  \bigl\{N \in \tilde \cF: \tilde P (N)=0\bigr\}\right), \qquad t\in [0,T];
	\end{split}
\end{equation} 
thus $\Set{\tcFt^n}_{n\ge 1}$ is the smallest 
filtration making all the relevant processes \eqref{def:tPhi} adapted.  By equality of the laws 
and \cite{DaPrato:2014aa}, $\tilde W^{v,n}$ and $\tilde W^{w,n}$ are 
cylindrical Brownian motions, i.e., there exist sequences 
$\Set{\tilde W_k^{v,n}}_{k\ge 1}$ and $\Set{\tilde W_k^{w,n}}_{k\ge 1}$ of 
mutually independent real-valued Wiener processes adapted to 
$\Set{\tcFt^n}_{t\in [0,T]}$ such that 
$$
\tilde W^{v,n}=\sum_{k\ge 1}\tilde W_k^{v,n} \psi_k, 
\qquad
\tilde W^{w,n}=\sum_{k\ge 1}\tilde W_k^{w,n} \psi_k,
$$
where $\Set{\psi_k}_{k\ge 1}$ is the basis of $\U$ 
and each series converges in $\U_0\supset \U$ (cf.~Section \ref{sec:stoch}). 
Below we need the $n$-truncated sums
\begin{equation}\label{eq:truncated-sums}
	\tilde W^{v,(n)} = \sum_{k=1}^n\tilde W_k^{v,n} \psi_k,
	\qquad  \tilde W^{w,(n)} = \sum_{k=1}^n\tilde W_k^{w,n} \psi_k,
\end{equation}
which converge respectively to $\tilde W^v$, $\tilde W^w$ in $C([0,T];\U_0)$, 
$\tilde P$-a.s., cf.~\eqref{eq:strong-conv1}.

We must show that the Faedo-Galerkin equations 
hold on the new probability space $(\tilde{D},\tilde{\cF}, \tilde{P})$. 
To do that, we use an argument of Bensoussan \cite{Bensoussan:1995aa}, 
developed originally for the stochastic Navier-Stokes equations. 
For other possible methods leading to the construction of martingale 
solutions, see for example \cite[Chap.~8]{DaPrato:2014aa} and \cite{Ondrejat:2010aa}.

\begin{lem}[Faedo-Galerkin equations]\label{lem:approx-eqn-tilde}
Relative to the stochastic basis $\tilde \cS_n$ in \eqref{eq:stoch-basis-n}, 
the functions $\tilde U^n$, $\tilde W^n$, $\tilde U_0^n$ defined 
in \eqref{eq:def-tvar} satisfy the following equations $\tilde P$-a.s.:
\begin{equation}\label{eq: approx-eqn-tilde}
	\begin{split}
		& \tilde v^n(t) + \eps_n \tilde  u_i^n(t) 
		= \tilde v^n_0 + \eps_n \tilde u_{i,0}^n 
		+ \int_0^t \Pi_n\left[\Div\bigl(M_i \Grad \tilde u_i^n \bigr) 
		- I(\tilde v^n,\tilde w^n) \right]\ds 
		\\ & \qquad\qquad\qquad\qquad \quad\,
		+ \int_0^t \eta^n(\tilde v^n) \, d\tilde W^{v,(n)}(s)
		\quad \text{in $(\tH(\Om))^*$},
		\\ & \tilde v^n(t) - \eps_n \tilde u_e^n(t)
		= \tilde v^n_0 - \eps_n \tilde u_{e,0}^n
		+\int_0^t \Pi_n\left[-\Div\bigl(M_e \Grad \tilde u_e^n \bigr) 
		- I(\tilde v^n, \tilde w^n) \right]\ds 
		\\ & \qquad\qquad\qquad\qquad \quad\,
		+ \int_0^t \eta^n(\tilde v^n) \, d\tilde W^{v,(n)}(s)
		\quad \text{in $(\tH(\Om))^*$},
		\\ & \tilde w^n(t) = \tilde w^n_0 
		+ \int_0^t H( \Pi_n \tilde v^n, \Pi^w_n\tilde w^n)\ds
		\\ & \qquad\qquad\qquad\qquad \quad\,
		+ \int_0^t \sigma^n(\tilde v^n)\, d\tilde W^{w,(n)}(s)
		\quad \text{in $(\tH(\Om))^*$},
	\end{split}
\end{equation}
for each $t\in [0,T]$, where $\eps_n$ is specified in \eqref{eq:def-eps}
and $\tilde W^{v,(n)}$, $\tilde W^{w,(n)}$ are defined in \eqref{eq:truncated-sums}.
Moreover, 
\begin{equation}\label{eq:tvn=tuin-tuen}
	\tilde v^n = \tilde u_i^n - \tilde u_e^n,
	\quad \text{$d\tilde P\times dt\times dx$ 
	a.e.~in $\tilde D\times (0,T)\times \Om$},
\end{equation}
and (by construction) $\tilde U^n$, $\tilde W^n$ are 
continuous, adapted (and thus  predictable) processes. 
Finally, the laws of $\tilde v_0^n$ and $\tilde w_0^n$ coincide with 
the laws of $\Pi_n  v_0$ and $\Pi_n  w_0$, respectively, 
where $v_0\sim\mu_{v_0}$, $w_0\sim\mu_{w_0}$ 
(see Definition \ref{def:martingale-sol}). 
\end{lem}

\begin{proof}
We establish the first equation in \eqref{eq: approx-eqn-tilde}, with 
the remaining ones following along the same lines. 
In accordance with Lemma \ref{lem:fg-solutions} and 
\eqref{eq:def-U-W-U0}, recall that $(U^n,W^n,U_0^n)$
is the continuous adapted solution to the Faedo-Galerkin 
equations \eqref{eq:approx-eqn-integrated} relative to $\cS$, cf.~\eqref{eq:fixedS}. 

Let us introduce the $(\tH(\Om))^*$ valued 
stochastic processes
\begin{align*}
	\cI_n(\omega,t) & := \left(v^n(t)-v^n_0\right) + \eps_n \left(u_i^n(t) -u_{i,0}^n\right)
	\\ & \qquad
	- \int_0^t \Pi_n\left[\Div\bigl(M_i \Grad u_i^n \bigr)
	- I(v^n,w^n) \right]\ds 
	- \int_0^t \eta^n(v^n) \dW^{v,n}(s),
	\\  \tilde \cI_n(\omega,t) & := \left(\tilde v^n(t)-\tilde v^n_0\right) 
	+ \eps_n \left(\tilde u_i^n(t) - \tilde u_{i,0}^n\right),
	\\ & \qquad
	- \int_0^t \Pi_n\left[\Div\bigl(M_i \Grad \tilde u_i^n \bigr)
	- I(\tilde v^n,\tilde w^n) \right]\ds 
	- \int_0^t \eta^n(\tilde v^n) \, d\tilde W^{v,(n)}(s),
\end{align*}
and the real-valued random variables, cf.~\eqref{eq:def-dual-norm},
$$
I_n(\omega)  := \norm{\cI_n}_{L^2((0,T);(\tH(\Om))^*)}^2, 
\quad 
\tilde I_n(\omega) := \norm{\tilde \cI_n}_{L^2((0,T);(\tH(\Om))^*)}^2.
$$

Note that $I_n = 0$ $P$-a.s.~and so $\E[I_n] = 0$. If we could write 
$I_n=L_n(\Phi_n)$ for a (deterministic) bounded continuous 
functional $L_n(\cdot)$ on $\cX$, cf.~\eqref{def:tPhi}, then by 
equality of the laws, also $\tilde \E[\tilde I_n] = 0$ and the result would follow. 
However, this is not directly achievable since the stochastic integral 
is not a deterministic function of $W^{v,n}$. Hence, certain modifications are needed to 
produce a workable proof \cite{Bensoussan:1995aa}. 
First of all, we do not consider $I_n$ but rather 
the bounded map $I_n/(1+I_n)$. Noting that $\E[I_n] = 0$ implies
\begin{equation}\label{eq:In-zero}
	\E \left[ \frac{I_n}{1+I_n}\right]=0,
\end{equation}
the goal is to show that
\begin{equation}\label{eq:tIn-zero}
	\tilde\E \left[ \frac{\tilde I_n}{1+\tilde I_n} \right]=0,
\end{equation}
from which the first equation in \eqref{eq: approx-eqn-tilde} follows.

Recall that, cf.~\eqref{eq:Wn},
$$
\int_0^t \eta^n(v^n) \dW^{v,n}(s)
= \sum_{k=1}^n \int_0^t \eta_k^n(v^n) \dW_k^v(s).
$$
Let $\vrho_\nu(t)$  be a standard mollifier and define (for $k=1,\ldots,n$)
$$
\eta_k^{n,\nu}:= (\eta_k^n(v^n)) \underset{(t)}{\star} 
\vrho_\nu, \qquad \nu>0.
$$
By properties of mollifiers, 
$$
\norm{\eta_k^{n,\nu}}_{L^2\left(D,\cF,P;L^2((0,T);L^2(\Om))\right)}
\le \norm{\eta_k^n(v^n)}_{L^2\left(D,\cF,P;L^2((0,T);L^2(\Om))\right)}
$$
and
\begin{equation}\label{eq:etak-reg-conv}
	\eta_k^{n,\nu} \to \eta_k^n \quad 
	\text{in $L^2\left(D,\cF,P;L^2((0,T);L^2(\Om))\right)$ as $\nu\to 0$.} 
\end{equation}
We define $\tilde \eta_k^{n,\nu}$ similarly (with 
$v^n$ replaced by $\tilde v^n$). 

An ``integration by parts" reveals that
$$
\int_0^t  \eta_k^{n,\nu} \dW_k^v(s)
=\left( \eta_k^{n,\nu}\right)(t) \, W_k^v(t)-
\int_0^t W_k^v(s) \frac{\partial}{\partial s} 
\left(\eta_k^{n,\nu}\right) \ds,
$$
i.e., thanks to the regularization of $\eta^n_k(v^n)$ in the $t$ variable, 
$\int_0^t \eta_k^{n,\nu} \dW_k^v(s)$ can be viewed 
as a (deterministic) functional of $W_k^v$.

Denote by $I_n^\nu$, $\tilde I_n^\nu$ the random variables corresponding 
to $I_n$, $\tilde I_n$ with $\eta^n_k(v^n)$, $\eta^n_k(\tilde v^n)$ 
replaced by $\eta^{n,\nu}_k$, $\tilde \eta^{n,\nu}_k$, respectively, 
and note that now 
$$
\frac{I_n^\nu}{1+I_n^\nu} = L_{n,\nu}(\Phi_n), 
\qquad 
\frac{\tilde I_n^\nu}{1+\tilde I_n^\nu}
= L_{n,\nu}(\tilde \Phi_n),
$$
for some bounded continuous functional $L_{n,\nu}(\cdot)$ 
on $\cX$. By equality of the laws,
\begin{equation}\label{eq:equal-laws-nu}
	\begin{split}
		\tilde\E \left[ \frac{\tilde I_n^\nu}{1+\tilde I_n^\nu} \right]
		=\int_{\cX} L_{n,\nu}(\Phi) \,d \tilde \cL_n(\Phi)
		&=\int_{\cX} L_{n,\nu}(\Phi) \,d \cL_n(\Phi)
		=\E \left[ \frac{I_n^\nu}{1+I_n^\nu}\right].
	\end{split}
\end{equation}
One can check that
\begin{equation}\label{eq:In-Innu-conv}
	\begin{split}
		&\E \left[ \abs{\frac{ I_n}{1+I_n} - \frac{I_n^\nu}{1+I_n^\nu}} \right]
		\le \E \left[ \abs{I_n- I_n^\nu} \right]
		\\ & \qquad \le C \left(\E \left [\int_0^T \sum_{k=1}^n 
		\norm{\eta_k^n(v^n)-\eta_k^{n,\nu}}_{L^2(\Om)}^2\dt \right]\right)^{\frac12}
		\overset{\eqref{eq:etak-reg-conv}}{\longrightarrow} 0 \quad \text{as $\nu\to 0$,}
	\end{split}
\end{equation}
and similarly
\begin{equation}\label{eq:tIn-tInnu-conv}
	\begin{split}
		\tilde \E \left[ \abs{\frac{\tilde I_n}{1+\tilde I_n}
		-\frac{\tilde I_n^\nu}{1+ \tilde I_n^\nu}} \right]
		\! \le \! C \left(\E \left [\int_0^T \sum_{k=1}^n 
		\norm{\eta_k^n(\tilde v^n)
		-\tilde \eta_k^{n,\nu}}_{L^2(\Om)}^2\dt \right]\right)^{\frac12} 
		\overset{\nu\downarrow 0}{\to} 0.
	\end{split}
\end{equation}
Combining \eqref{eq:equal-laws-nu}, \eqref{eq:In-Innu-conv}, 
\eqref{eq:tIn-tInnu-conv}, \eqref{eq:In-zero} we arrive at \eqref{eq:tIn-zero}.

Finally, let us prove \eqref{eq:tvn=tuin-tuen}. 
By construction, $v^n=u_i^n-u_e^n$ and so
$$
\norm{v^n - 
(u_i^n-u_e^n)}_{L^2\left(D,\cF,P;L^2((0,T);L^2(\Om))\right)}=0.
$$
For $\Phi\in \cX$, define
$$
L(\Phi) =
\frac{\norm{v - (u_i-u_e)}_{L^2((0,T);L^2(\Om))}^2}{1
+\norm{v - (u_i-u_e)}_{L^2((0,T);L^2(\Om))}^2}.
$$
Since $L(\cdot)$ is a bounded continuous functional on $\cX$ 
and the laws $\cL_n$, $\tilde \cL_n$ are equal,
$$
\tilde \E \left[ L(\tilde \Phi_n)\right]
=\E\left[ L(\Phi_n)\right]
\le \norm{v^n - 
(u_i^n-u_e^n)}_{L^2\left(D,\cF,P;L^2((0,T);L^2(\Om))\right)}^2=0,
$$
i.e., $L(\tilde \Phi_n)=0$ $\tilde P$-a.s.~and 
thus, via \eqref{eq:apriori-est-tilde},
$$
\norm{\tilde v^n - (\tilde u_i^n- \tilde u_e^n)}_{
L^2\left(\tilde D,\tilde \cF,\tilde P;L^2((0,T);L^2(\Om))\right)}=0.
$$
This concludes the proof of the lemma.
\end{proof}

\subsection{Passing to the limit}\label{subsec:limit}
We begin by turning the probability space $(\tilde{D},\tilde{\cF}, \tilde{P})$, 
cf.~\eqref{def:tPhi} and \eqref{eq:tPhi-conv}, into a stochastic basis, 
\begin{equation}\label{eq:stoch-basis-new}
	\tilde \cS=\left(\tilde D,\tilde \cF,\Set{\tcFt}_{t\in [0,T]},\tilde P,
	\tilde W^v, \tilde W^w\right),
\end{equation}
by supplying the natural filtration $\Set{\tcFt}_{t\in [0,T]}$, i.e., the 
smallest filtration with respect to which all the relevant processes are adapted, viz.
\begin{equation}\label{eq:filtration-new}
	\tcFt= \sigma\left(\sigma \bigl(\tilde \Phi\big |_{[0,t]}\bigr)
	\bigcup  \bigl\{N \in \tilde \cF: \tilde P (N)=0\bigr\}\right), 
	\quad t\in [0,T].
\end{equation}

Lemma \ref{lem:approx-eqn-tilde} shows that $\tilde U^n, \tilde W^n, \tilde U_0^n$ 
satisfy the Faedo-Galerkin equations \eqref{eq:approx-eqn-integrated}; 
hence, they are worthy of being referred to as ``approximations".  
The next two lemmas summarize the relevant convergence properties 
satisfied by these approximations.

\begin{lem}[weak convergence]\label{lem:tlimits}
There exist functions $\tilde u_i, \tilde u_e, \tilde v, \tilde w$, with
\begin{align*}
	& \tilde u_i, \tilde u_e, \tilde v \in 
	L^2\left(\tilde D,\tilde \cF,\tilde P;L^2((0,T);\tH(\Om))\right),
	\, \text{$\, \tilde v = \tilde u_i - \tilde u_e$,}
	\\ & \tilde v, \, \tilde w \in 
	L^2\left(\tilde D,\tilde \cF,\tilde P;L^\infty((0,T);L^2(\Om))\right),
	\, \tilde v \in L^4\left(\tilde D,\tilde \cF,\tilde P;L^4(\Om_T)\right), 
\end{align*}
such that as $n\to \infty$, passing to a subsequence if necessary,
\begin{equation}\label{eq:tilde-weakconv}
	\left \{
	\begin{split}
		& \tilde u^n_i \weak \tilde u_i, \, 
		\tilde u^n_e \weak \tilde u_e  \quad  
		\text{in $L^2\left(\tilde D,\tilde \cF,\tilde P;L^2((0,T);\tH(\Om))\right)$},
		\\&
		\eps_n \tilde u^n_i \to 0, \, 
		\eps_n \tilde u^n_e \to 0  \quad 
		\text{in $L^2\left(\tilde D,\tilde \cF,\tilde P;L^2((0,T);L^2(\Om))\right)$},
		\\& 
		\tilde v^n \weak \tilde v \quad  
		\text{in $L^2\left(\tilde D,\tilde \cF,\tilde P;L^2((0,T);\tH(\Om))\right)$},
		\\ & 
		\tilde v^n \weakstar \tilde v \quad  
		\text{in $L^2\left(\tilde D,\tilde \cF,\tilde P;L^\infty((0,T);L^2(\Om))\right)$},
		\\ & 
		\tilde v^n \weak \tilde v \quad  
		\text{in $L^4\left(\tilde D,\tilde \cF,\tilde P;L^4(\Om_T)\right)$},
		\\ & 
		\tilde w^n \weakstar \tilde w \quad  
		\text{in $L^2\left(\tilde D,\tilde \cF,\tilde P;L^\infty((0,T);L^2(\Om))\right)$}.
	\end{split}\right.
\end{equation}
\end{lem}

\begin{proof}
The claims in \eqref{eq:tilde-weakconv} follow from the 
estimates in \eqref{eq:weakconv1} and the sequential Banach-Alaoglu theorem. 
The relation 
$$
\tilde v_i = \tilde u_i - \tilde u_e, 
\quad \text{$d\tilde P\times dt\times dx$ 
a.e.~in $\tilde D\times (0,T)\times \Om$},
$$
is a consequence of \eqref{eq:tvn=tuin-tuen} and the weak 
convergences in $L^2_{\omega,t,x}$ of $\tilde v^n, \tilde u_i^n, \tilde u_e^n$.
The limit functions $\tilde u_i, \tilde u_e, \tilde v, \tilde w$ 
are easily identified with the a.s.~representations 
in Lemma \ref{lem:as-represent}.
\end{proof}

As a result of \eqref{eq:apriori-est-q0-tilde}, we 
can upgrade a.s.~to $L^2$ convergence.

\begin{lem}[strong convergence]
As $n\to \infty$, passing to a subsequence if necessary, the following 
strong convergences hold:
\begin{equation}\label{eq:strong-conv2}
	\begin{split}
		& \tilde v^n \to \tilde v \quad  
		\text{in $L^2\left(\tilde D,\tilde \cF,\tilde P;L^2((0,T);L^2(\Om))\right)$},
		\\ &  
		\tilde w^n \to \tilde w \quad  
		\text{in $L^2\left(\tilde D,\tilde \cF, \tilde P;L^2((0,T);(\tH(\Om))^*)\right)$},
		\\ &
		\tilde W^{v,n} \to \tilde W^v, \, \tilde W^{w,n} \to \tilde W^w 
		\quad  \text{in $L^2\left(\tilde D,\tilde \cF, \tilde P; C([0,T]; \U_0)\right)$.}
		\\ & 
		\tilde u_{i,0}^n \to \tilde u_{i,0}, \,
		\tilde u_{e,0}^n\to \tilde u_{e,0}, \, 
		\tilde v_0^n \to \tilde v_0,\, 
		\tilde w_0^n \to \tilde w_0
		\quad  \text{in $L^2\left(\tilde D,\tilde \cF, \tilde P;L^2(\Om)\right)$.}
	\end{split}
\end{equation}

\end{lem}

\begin{proof}
The proof merges the a.s.~convergences in \eqref{eq:strong-conv1}, 
the high-order moment estimates in \eqref{eq:apriori-est-q0-tilde}, 
and Vitali's convergence theorem. To justify the first claim 
in \eqref{eq:strong-conv2}, for example, we consider the estimate
$$
\tilde \E \left[ \norm{\tilde v^n(t)}_{L^\infty((0,T);L^2(\Om))}^{q_0}\right]
\le C, \qquad q_0>2,
$$ 
see \eqref{eq:apriori-est-q0-tilde}. 
From this we infer the equi-integrability (w.r.t.~$\tilde P$) of
$$
\Set{\norm{\tilde v^n(t)}_{L^2((0,T);L^2(\Om))}^2}_{n\ge1}.
$$
Accordingly, the first claim in \eqref{eq:strong-conv2} follows from 
the $\tilde P$-a.s.~convergence in \eqref{eq:strong-conv1} and 
Vitali's convergence theorem, with the remaining 
claims following along similar lines. Regarding the third claim, note also that 
for $\tilde W^n= \tilde W^{v,n}$ or $\tilde W^{w,n}$,
\begin{equation}\label{eq:W-bound}
	\tilde \E \left [ \norm{\tilde W^n}_{C([0,T]; \U_0)}^q\right] 
	= \E \left [ \norm{W^n}_{C([0,T]; \U_0)}^q\right] \le C_T, 
	\qquad \forall q\in [1,\infty),
\end{equation}
which follows from equality of the laws and a martingale inequality.
\end{proof}

For each $n\ge1$, $\tilde W^{v,n}$ and $\tilde W^{w,n}$
are (independent) cylindrical Wiener processes with 
respect to the stochastic basis $\tilde \cS_n$, see \eqref{eq:stoch-basis-n}.  
Since $\tilde W^{v,n}\to \tilde W^v$, 
$\tilde W^{w,n} \to \tilde W^w$ in the sense of \eqref{eq:strong-conv1} 
or \eqref{eq:strong-conv2}, it is more or less obvious 
that also the limit processes $\tilde W^v$, $\tilde W^w$ 
are cylindrical Wiener processes. Indeed, we have

\begin{lem}\label{lem:wiener-tilde}
The a.s.~representations $\tilde W = \tilde W^v,\tilde W^w$ from 
Lemma \ref{lem:as-represent} are (independent) cylindrical 
Wiener processes with respect to sequences 
$$
\Set{\tilde W_k^v}_{k\ge 1}, \quad \Set{\tilde W_k^w}_{k\ge 1}
$$ 
of mutually independent real-valued Wiener processes adapted 
to the natural filtration $\Set{\tilde \cFt}_{t\in [0,T]}$, 
cf.~\eqref{eq:stoch-basis-new} and \eqref{eq:filtration-new}, such that
$$
\tilde W^v=\sum_{k\ge 1}\tilde W_k^v \psi_k, 
\quad 
\tilde W^v=\sum_{k\ge 1}\tilde W_k^w \psi_k.
$$
\end{lem}

\begin{proof}
The proof is standard, see e.g.~\cite[Lemma 9.9]{Ondrejat:2010aa}
or \cite[Proposition 4.8]{Debussche:2016aa}. 
To be more precise, according to the martingale characterization 
theorem \cite[Theorem 4.6]{DaPrato:2014aa}, 
we must show that $\tilde W^v, W^w$ are $\{\tilde \cF_t\}$-martingales. 
Recall that an integrable, adapted process 
$\Set{M(t)}_{t\in [0,T]}$ on $(D,\cF,\{\cF_t\},P)$ is called a martingale if 
$$
\E\left[ M(t) \big| \cF_s\right] = M(s), \quad \text{$P$-a.s.},
$$
for all $t,s\in [0,T]$ with $s\le t$. This requirement is equivalent to
$$
\int_D \En_A \Bigl(M(t)-M(s)\Bigr) \, dP=0, \quad 
A\in \cF_s, \, s,t\in [0,T],\, s\le t. 
$$
Consequently, to conclude the proof of the lemma, we need to verify that
$$
\tilde \E \left[ \En_A  \left( \tilde W(t)-\tilde W(s) \right)\right]
=0, \qquad \tilde W=\tilde W^v, \tilde W^w,
$$
for all $A\in \tilde \cF_s$, $s,t\in [0,T]$, $s\le t$. With $\tilde \Phi$ defined 
in \eqref{def:tPhi}, it is sufficient to show that
$$
\tilde \E \left[ L_s(\tilde \Phi)  \left( \tilde W(t)-\tilde W(s) \right)\right]=0,
\qquad \tilde W=\tilde W^v, \tilde W^w,
$$ 
for all bounded continuous functionals $L_s(\Phi)$ on $\cX$ depending 
only on the values of $\Phi$ restricted to $[0,s]$. Since the laws of 
$\Phi_n$ and $\tilde \Phi_n$ coincide, cf.~\eqref{def:tPhi}, 
\begin{equation}\label{eq:mart-tmp1}
	\tilde \E \left[ L_s(\tilde \Phi_n)  \left( \tilde W^n(t)-\tilde W^n(s) \right)\right]
	= \E \left[ L_s(\Phi_n)  \left( W^n(t)- W^n(s) \right)\right]=0,
\end{equation}
where the last equality is a result of the $\{ \cF_t^n\}$-martingale 
property of $W^n=W^{v,n},W^{w,n}$. 
By \eqref{eq:strong-conv1}, \eqref{eq:W-bound}, 
and Vitali's convergence theorem, we can pass to 
the limit in \eqref{eq:mart-tmp1} as $n\to \infty$. 
This concludes the proof of the lemma.
\end{proof}

Given the above convergences,  the final step is to 
pass to the limit in the Faedo-Galerkin equations. 
The next lemma shows that the Skorokhod-Jakubowski representations 
satisfy the weak form \eqref{eq:weakform} of the stochastic bidomain system.

\begin{lem}[limit equations]\label{eq:limit-eqs-tilde}
Let $\tilde U, \tilde W, \tilde v_0, \tilde w_0$ be the a.s.~representations 
constructed in Lemma \ref{lem:as-represent}, and $\tilde \cS$ 
the accompanying stochastic basis defined in 
\eqref{eq:stoch-basis-new}, \eqref{eq:filtration-new}, so that 
$\tilde v, \tilde w, \tilde W^v, \tilde W^w$ become $\{\tcFt\}$-adapted processes. 
Then the following equations hold $\tilde P$-a.s., for a.e.~$t\in [0,T]$:
\begin{equation}\label{eq:weakform-tilde}
	\begin{split}
		& \int_{\Om} \tilde v(t) \vphi_i \dx 
		+ \int_0^t \int_{\Om} 
		\Bigl( M_i\Grad \tilde u_i \cdot \Grad \vphi_i 
		+ I(\tilde v,\tilde w) \vphi_i \Bigr) \dx\ds
			\\ & \qquad\quad
		= \int_{\Om} \tilde v_0 \, \vphi_i \dx 
		+ \int_0^t \int_{\Om} \eta(\tilde v) \vphi_i \dx\, d\tilde W^v (s), 
		\\ & 
		\int_{\Om} \tilde v(t)  \vphi_e \dx
		+ \int_0^t \int_{\Om} 
		\Bigl (- M_e\Grad \tilde u_e\cdot \Grad \vphi_e
		+ I(\tilde v,\tilde w) \vphi_e \Bigr) \dx\ds 
			\\ & \qquad \quad 
		=\int_{\Om} \tilde v_0 \vphi_e \dx
		+ \int_0^t \int_{\Om} \eta(\tilde v)  \vphi_e \dx \, d\tilde W^v(s),
		\\ &
		\int_{\Om} \tilde w(t) \vphi\dx 
		=\int_{\Om} \tilde w_0 \vphi \dx
		+\int_0^t \int_{\Om} H(\tilde v,\tilde w)\vphi \dx \ds
		\\ & \qquad\qquad\qquad\qquad
		+ \int_0^t \int_{\Om} \sigma(\tilde v) \vphi \dx \, d\tilde W^w(s),
		\end{split}
	\end{equation}
for all $\vphi_i,\vphi_e \in \tH(\Om)$ and $\vphi\in L^2(\Om)$. 
The laws of $\tilde v(0)=\tilde v_0$ and $\tilde w(0)=\tilde w_0$ 
are $\mu_{v_0}$ and $\mu_{w_0}$, respectively.
\end{lem}

\begin{proof}
We establish the first equation in \eqref{eq:weakform-tilde}. The  
remaining equations are treated in the same way. 
Let $Z\subset \tilde D\times [0,T]$ be a measurable set, 
and denote by 
\begin{equation}\label{eq:Zdef}
	\En_Z(\omega,t)\in L^\infty \left(\tilde D\times [0,T]; \tilde dP\times dt \right)
\end{equation}
the characteristic function of $Z$. Our aim is to show 
\begin{equation}\label{eq:limit-weak-form-tilde-tmp0}
	\begin{split}
		&\E \left[\int_0^T \En_Z(\omega,t) 
		\left(\, \int_{\Om} \tilde v(t) \vphi_i \dx\right) \dt \right]
		\\ & \quad +\E \left[\int_0^T \En_Z(\omega,t) 
		\left(\, \int_0^t  \int_{\Om} M_i\Grad \tilde u_i
		\cdot \Grad \vphi_i \dx\ds\right) \dt\right]
		\\ & \quad + \E \left[\int_0^T \En_Z(\omega,t)  
		\left (\, \int_0^t \int_{\Om} I(\tilde v,\tilde w) \vphi_i \dx\ds\right)\dt \right]
		\\ & \quad\quad 
		=\E \left[\int_0^T \En_Z(\omega,t) \left(\, \int_{\Om} 
		\tilde v_0 \vphi_i \dx\right) \dt \right]
		\\ & \quad\qquad\quad + \E \left[\int_0^T \En_Z(\omega,t) 
		\left(\, \int_0^t \int_{\Om} \eta(\tilde v) \vphi_i \dx\, d\tilde W^v (s)\right)\dt \right].
	\end{split}
\end{equation}
Then, since $Z$ is an arbitrary measurable set and 
the simple functions are dense in $L^2$, we  conclude that 
the first equation in \eqref{eq:weakform-tilde} holds 
for $d\tilde P\times dt$ almost every $(\omega,t)\in \tilde D \times [0,T]$ and 
any $\vphi_i \in \tH(\Om)$.

Fix $\vphi_i \in \tH(\Om)$, and note that \eqref{eq: approx-eqn-tilde} implies
\begin{equation}\label{eq:weakform-tilde-tmp1}
	\begin{split}
		& \int_{\Om} \tilde v^n(t) \vphi_i \dx
		+ \int_{\Om} \eps_n \tilde u_i^n(t) \vphi_i \dx 
		\\ & 
		\quad 
		+\int_0^t \int_{\Om} M_i\Grad \tilde u_i^n
		\cdot \Grad \Pi_n \vphi_i \dx\ds
		+ \int_0^t \int_{\Om}
		I(\tilde v^n,\tilde w^n) \Pi_n \vphi_i \dx\ds
		\\ & \quad\quad
		= \int_{\Om} \tilde v_0^n \, \vphi_i \dx 
		+\int_{\Om} \eps_n \tilde u_{i,0}^n \, \vphi_i \dx 
		+ \int_0^t \int_{\Om} \eta^n(\tilde v^n) \vphi_i 
		\dx\, d\tilde W^{v,(n)} (s),
	\end{split}
\end{equation}
using \eqref{eq:proj-prop1}. We multiply \eqref{eq:weakform-tilde-tmp1} 
by $\En_Z(\omega,t)$, cf.~\eqref{eq:Zdef}, 
integrate over $(\omega,t)$,  and then attempt to pass to the limit 
$n\to\infty$ in each term separately. 

We will make repeated use of the following simple fact: If 
$X_n\weak X$ in $L^p(\tilde D\times (0,T))$, for $p\in [1,\infty)$, then 
$\int_0^t X_n \ds \weak \int_0^t X\ds$ in $L^p(\tilde D\times (0,T))$ as well.

First, since
\begin{equation}\label{eq.test-fun}
	\En_Z(\omega,t) \vphi_i(x)\in
	 L^2\left(\tilde D,\tilde \cF,\tilde P;L^2((0,T);L^2(\Om))\right), 
\end{equation}
the weak convergence in $L^2_{\omega,t,x}$ 
of $\tilde v^n$, cf.~\eqref{eq:tilde-weakconv}, implies
$$
\tilde \E \left[\int_0^T \En_Z(\omega,t) 
\left(\, \int_{\Om}  \tilde v^n(t)\vphi_i \dx\right) \dt \right]
\to \tilde \E \left[\int_0^T \En_Z(\omega,t) 
\left(\, \int_{\Om}  \tilde v(t)\vphi_i \dx\right) \dt \right],
$$
as $n\to \infty$. Similarly, 
$$
\tilde \E \left[\int_0^T \En_Z(\omega,t) 
\left(\, \int_{\Om} \eps_n u_i^n \vphi_i \dx\right) \dt \right]\to 0, 
\quad \text{as $n\to \infty$.}
$$

The initial data terms on the right-hand side 
of \eqref{eq:weakform-tilde-tmp1} can be treated in the 
same way, using \eqref{eq:strong-conv2}. Recall also 
that the laws of $\tilde v_0^n$, $\tilde w_0^n$ coincide with 
the laws of $\Pi_n  v_0$, $\Pi_n  w_0$, respectively, 
and that $v_0\sim\mu_{v_0}$, $w_0\sim\mu_{w_0}$. 
Since $\Pi_n  v_0\to v_0$, $\Pi_n  w_0\to w_0$
in $L^2_{\omega,x}$ as $n\to \infty$, cf.~\eqref{eq:L2-conv-init} 
or  \eqref{eq:proj-prop2}, we conclude 
that $\tilde v(0)=\tilde v_0\sim \mu_{v_0}$, 
$\tilde w(0)=\tilde w_0\sim \mu_{w_0}$.

Next, note that $\Grad \Pi_n \vphi_i\to \Grad \vphi_i$ in 
$L^2(\Om)$ as $n\to \infty$, cf.~\eqref{eq:proj-prop2}. 
By weak convergence in $L^2_{\omega,t,x}$ of 
$\tilde \Grad u_i^n$, cf.~\eqref{eq:tilde-weakconv}, 
and \eqref{eq.test-fun}, it follows that
\begin{align*}
	&\tilde \E \left[\int_0^T \En_Z(\omega,t) 
	\left(\, \int_0^t \int_{\Om} M_i\Grad \tilde u_i^n
	\cdot \Grad \Pi_n \vphi_i \dx\ds \right)\dt \right]
	\\ & \qquad
	\to \tilde \E \left[\int_0^T \En_Z(\omega,t) 
	\left(\,\int_0^t \int_{\Om} M_i\Grad \tilde u_i
	\cdot \Grad \vphi_i \dx\ds\right)\dt \right]
	\to 0 \quad \text{as $n\to \infty$}. 
\end{align*}

To demonstrate convergence of the stochastic integral
$$
\int_0^t \int_{\Om} \eta^n(\tilde v^n) 
\vphi_i \dx\, d\tilde W^{v,(n)} (s)
= \int_{\Om} \left(\, \int_0^t \eta^n(\tilde v^n) 
\, d\tilde W^{v,(n)} (s)\right) \vphi_i \dx,
$$
we will use Lemma \ref{lem:stoch-conv} to infer that
\begin{equation}\label{eq:stoch-conv-tmp1}
	\int_0^t \eta^n(\tilde v^n) \, d\tilde W^{v,(n)} (s) \to 
	\int_0^t \eta(\tilde v) \, d\tilde W^v (s)
	\quad \text{in $L^2((0,T);L^2(\Om))$},
\end{equation}
in probability, as $n\to \infty$. 
Since $\tilde W^{v,(n)} \to \tilde W^v$ in $C([0,T];\U_0)$, 
$\tilde P$-a.s.~(and thus in probability), cf.~\eqref{eq:strong-conv1}, 
it remains to prove that 
\begin{equation}\label{eq:etan-conv}
	\eta^n(\tilde v^n) \to \eta(\tilde v) 
	\quad 
	\text{in $L^2\left((0,T); L_2(\U;L^2(\Om))\right)$, 
	$\tilde P$-almost surely.}
\end{equation}
Before we continue, recall that
$$
\int_0^t \eta^n(\tilde v^n)  \, d\tilde W^{v,(n)}
=\sum_{k=1}^n\int_0^t \eta^n_k(\tilde v^n)  \, d\tilde W^{v,n}_k,
$$
where $\eta_k^n(\tilde v^n)=\eta^n(\tilde v^n)
\psi_k \in L^2(\Om)$, $\Set{\psi_k}_{k\ge 1}$ is 
an orthonormal basis of $\U$,
$$
\eta_k^n(\tilde v^n) = \sum_{l=1}^n \eta_{k,l}(\tilde v^n) e_l, 
\quad 
\eta_{k,l}(\tilde v^n)= \left( \eta_k(\tilde v^n), e_l\right)_{L^2(\Om)},
$$ 
and $\Set{e_l}_{l=1}^\infty$ is an orthonormal basis of $L^2(\Om)$.
A similar decomposition holds for $\eta(\tilde v)$. 
Note that
\begin{equation}\label{est:vitali-start1}
	\begin{split}
		&\int_0^t \norm{\eta(\tilde v)-\eta^n(\tilde v^n)}_{L_2(\U;L^2(\Om))}^2\ds
		\\ & \quad 
		\leq \int_0^t \norm{\eta(\tilde v)-\eta(\tilde v^n)}_{L_2(\U;L^2(\Om))}^2 \ds
		+\int_0^t \norm{\eta(\tilde v)-\eta^n(\tilde v)}_{L_2(\U;L^2(\Om))}^2 \ds
		\\ & \quad =: I_1+I_2.
	\end{split}
\end{equation}
Exploiting \eqref{eq:noise-cond2} and \eqref{eq:strong-conv1}, we 
conclude easily that
\begin{equation}\label{eq:noise-conv-I1}
	I_1 \to 0, \quad \text{$\tilde P$-almost surely},
\end{equation}
as $n\to\infty$. We handle the $I_2$-term as follows: 
\begin{equation*}
	\begin{split}
		I_2 &= \int_0^t \sum_{k \geq 1}
		\norm{\eta_k(\tilde v)-\eta_k^n(\tilde v)}_{L^2(\Om)}^2 \ds
		\\ & 
		= \int_0^t \sum_{k \geq 1}
		\norm{\sum_{l \geq 1} \eta_{k,l}(\tilde v)e_l
		-\sum^n_{l=1} \eta_{k,l}(\tilde v)e_l}_{L^2(\Om)}^2 \ds
		\\ & = \int_0^t \sum_{k \geq 1}
		\norm{\eta_k(\tilde v)-\Pi_n\left(\eta_k(\tilde v)\right)}_{L^2(\Om)}^2 \ds.
	\end{split}
\end{equation*}
Observe that the integrand can be dominated 
by an $L^1(0,T)$ function, $\tilde P$-a.s.:
\begin{align*}
	&\sum_{k\ge 1}\norm{\eta_k(\tilde v(t))
	-\Pi_n\left(\eta_k(\tilde v(t))\right)}^2_{L^2(\Om)}
	\\ & \quad
	\overset{\eqref{eq:proj-prop0}}{\le} 
	4 \sum_{k \geq 1}\norm{\eta_k(\tilde v(t))}^2_{L^2(\Om)}
	= 4 \norm{\eta(\tilde v(t))}^2_{L_2(\U;L^2(\Om))}
	\overset{\eqref{eq:noise-cond2}}{\le} 
	C\left(1+\norm{\tilde v(t)}^2_{L^2(\Om)} \right),
\end{align*}
where we recall that $\tilde v \in L^2_\omega\left(L^\infty_t\left(L^2_x\right)\right)$
and thus $\tilde v\in L_t^2\left(L^2_x\right)$ $\tilde P$-a.s.~(cf.~Lemma \ref{lem:tlimits}).
Clearly, by \eqref{eq:proj-prop2}, $\Pi_n\left(\eta_k(\tilde v)\right)$ 
converges as $n\to \infty$ to $\eta_k(\tilde v)$ in $L^2(\Om)$, 
for a.e.~$t$, $\tilde P$-almost surely. 
Therefore, after an application of Lebesgue's dominated 
convergence theorem,
\begin{equation}\label{eq:noise-conv-I2}
	I_2 \overset{n\uparrow \infty}{\to} 0, \quad \text{$\tilde P$-almost surely}.
\end{equation}
Combining \eqref{eq:noise-conv-I1}, \eqref{est:vitali-start1}, 
\eqref{eq:noise-conv-I2} we arrive at \eqref{eq:etan-conv} (which 
implies \eqref{eq:stoch-conv-tmp1} via Lemma \ref{lem:stoch-conv}).

Passing to a subsequence (not relabeled), we can replace ``in probability" by 
``$\tilde P$-almost surely" in \eqref{eq:stoch-conv-tmp1}.  
Next, fixing any $q>2$, we verify that
\begin{align*}
	&\tilde \E\left[\norm{\int_0^t \eta^n(\tilde v^n)  
	\, d\tilde W^{v,(n)}}_{L^2((0,T);L^2(\Om))}^q \right]
	\\ & \quad =\tilde \E\left[\left(\int_0^T 
	\norm{\sum_{k=1}^n\int_0^t \eta^n_k(\tilde v^n)  
	\, d\tilde W^{v,n}_k}_{L^2(\Om)}^2\dt\right)^{\frac{q}{2}} \right]
	\\ & \quad 
	\le \bar C_T  \tilde \E\left[\, \sup_{t\in[0,T]}
	\norm{\sum_{k=1}^n\int_0^t \eta^n_k(\tilde v^n) 
	\, d\tilde W^{v,n}_k}_{L^2(\Om)}^q\,\right]
	\\ &\quad
	\le C_T\, \tilde \E\left[\left(\int_0^T 
	\sum_{k=1}^n\norm{\eta_k(\tilde v^n)}_{L^2(\Om)}^2 \dt\right)^{\frac{q}{2}}\right]
	\le C_{\eta,T},
\end{align*}
using the Burkholder-Davis-Gundy inequality \eqref{eq:bdg} and 
\eqref{eq:noise-cond}, \eqref{eq:apriori-est-tilde}. 
Accordingly, in light of Vitali's theorem, \eqref{eq:stoch-conv-tmp1} implies 
\begin{equation*}
	\int_0^t \eta^n(\tilde v^n) \, d\tilde W^{v,n} (s) 
	\overset{n\uparrow \infty}{\to}
	\int_0^t \eta(\tilde v) \, d\tilde W^v (s) \quad 
	\text{in $L^2\left(\tilde D,\tilde \cF,\tilde P;L^2((0,T);L^2(\Om))\right)$},
\end{equation*}
and hence
\begin{align*}
	&\tilde \E \left[\int_0^T \En_Z(\omega,t) \left(\,
	\int_0^t \int_{\Om} \eta^n(\tilde v^n)
	\vphi_i \dx\, d\tilde W^{v,n} (s)\right)\dt \right]
	\\ & \quad 
	= \tilde \E \left[\int_0^T 
	\int_{\Om} \left(\int_0^t \eta^n(\tilde v^n) \, d\tilde W^{v,n} (s)\right) 
	\left(\En_Z(\omega,t)\Pi_n \vphi_i(x)\right) \dx\dt\right]
	\\ & \quad
	\to \tilde \E \left[\int_0^T \En_Z(\omega,t) \left(\, 
	\int_0^t \int_{\Om} \eta(\tilde v) 
	\vphi_i \dx \, d\tilde W^v (s)\right)\dt\right]
	\quad \text{as $n\to\infty$}.
\end{align*}

With reference to the nonlinear term in \eqref{eq:weakform-tilde-tmp1}, 
according to condition {\rm (\textbf{GFHN})}, we have
$I(v,w)=I_1(v)+I_2(v)w$ with
$$
\abs{I_1(v)}\le c_{I,1}\left(1+ \abs{v}^3\right),
\quad 
I_2(v)=c_{I,3}+c_{I,4}v.
$$
By cause of the first part of \eqref{eq:strong-conv2}, passing 
to a subsequence if necessary, we may assume that  as $n\to \infty$,
$$
\tilde v^n \to \tilde v \quad
\text{for $d\tilde P\times dt\times dx$ almost every 
$(\omega,t,x)\in \tilde D \times [0,T]\times \Om$.}
$$ 
As a result of this, the boundedness of $\tilde v^n$ in $L^4_{\omega,t,x}$, 
cf.~\eqref{eq:apriori-est-tilde}, and Vitali's convergence theorem, we conclude
that as as $n\to \infty$,
\begin{equation}\label{eq:I1conv}
	\begin{split}
		& \tilde v^n \to \tilde v \quad
		\text{in $L^q(d \tilde P\times dt\times dx)$},
		\quad \text{for any $q\in [1,4)$,}
		\\ & I_1(\tilde v^n) \to I_1(\tilde v) \quad 
		\text{in $L^q(d \tilde P\times dt\times dx)$,}
		\quad \text{for any $q\in [1,4/3)$.}
	\end{split}
\end{equation}
Fix two numbers $q,q'$ such that
$$
\frac32 \le q<2, \quad 2<q'\le 3, \quad \quad \frac{1}{q}+\frac{1}{q'}=1,
$$
for example $q=3/2$ and $q'=3$. 
Then, by H\"older's inequality,
\begin{align*}
&\tilde \E \left[\int_0^T  \int_{\Om} 
	\abs{I_2(\tilde v^n) \Pi_n \vphi_i -I_2(\tilde v) \vphi_i}^2\dx\dt\right]
	\\ & \qquad
	\le  \tilde \E \left[\int_0^T  \int_{\Om}
	\abs{I_2(\tilde v^n)}^2
	\abs{\Pi_n \vphi_i -\vphi_i}^2\dx\dt\right]
	\\ & \qquad \qquad+ 
	 \tilde \E \left[\int_0^T \int_{\Om}
	 \abs{I_2(\tilde v^n)- I_2(\tilde v)}^2\abs{\vphi_i}^2\dx\dt\right]
 	\\ & \qquad \le \norm{I_2(\tilde v^n)}_{L^{2q}_{\omega,t,x}}^2
	\norm{\Pi_n \vphi_i -\vphi_i}_{L^{2q'}_{\omega,t,x}}^2
	 \\ & \qquad \qquad 
	 +\norm{I_2(\tilde v^n)-I_2(\tilde v)}_{L^{2q}_{\omega,t,x}}^2
	 \norm{\vphi_i}_{L^{2q'}_{\omega,t,x}}^2 \to 0 \quad \text{as $n\to \infty$},
\end{align*}
since $I_2(\tilde v^n)$ is bounded and converges strongly 
in $L^{2q}_{\omega,t,x}$ (with $2q<4$), consult \eqref{eq:I1conv}. 
Consequently, $I_2(\tilde v^n) \Pi_n \vphi_i  \to I_2(\tilde v)\vphi_i$ 
in $L^2(\tilde dP\times dt\times dx)$.
Besides, \eqref{eq:tilde-weakconv} implies $\tilde w^n \weak \tilde w$ 
in $L^2(\tilde dP\times dt\times dx)$. Hence,
\begin{equation}\label{eq:I2w-conv}
	I_2(\tilde v^n)\, \tilde w^n\, \Pi_n \vphi_i   
	\overset{n\uparrow \infty}{\weak} I_2(\tilde v)\, \vphi_i\, \tilde w  
	\quad \text{in $L^1(\tilde dP\times dt\times dx)$.}
\end{equation}

Regarding the $I_1$ term, fix two numbers $q,q'$ such that
$$
\frac65\le q<\frac43, \quad 3<q'\le 6, \quad \quad \frac{1}{q}+\frac{1}{q'}=1.
$$
Then, similar to the treatment of the $I_1$ term,
\begin{align*}
&\tilde \E \left[\int_0^T  \int_{\Om} 
	\abs{I_1(\tilde v^n) \Pi_n \vphi_i -I_1(\tilde v) \vphi_i}\dx\dt\right]
 	\\ & \qquad \le \norm{I_1(\tilde v^n)}_{L^q_{\omega,t,x}}
	\norm{\Pi_n \vphi_i -\vphi_i}_{L^{q'}_{\omega,t,x}}
	 \\ & \qquad \qquad 
	 +\norm{I_1(\tilde v^n)- I_1(\tilde v)}_{L^q_{\omega,t,x}}
	 \norm{\vphi_i}_{L^{q'}_{\omega,t,x}} \to 0 \quad \text{as $n\to \infty$},
\end{align*}
where we have used that $I_2(\tilde v^n)$ is bounded and 
converges strongly in $L^q_{\omega,t,x}$ ($q<4/3$), see \eqref{eq:I1conv},
and the Sobolev embedding theorem to control the $L^{q'}$ norm
of $\vphi_i$, $\Pi_n \vphi_i - \vphi_i$ in terms 
of the $\tH$ norm ($q'\le 6$). In other words, 
$$
I_1(\tilde v^n) \Pi_n \vphi_i \to I_1(\tilde v) \vphi_i
\quad \text{in $L^1\left(d\tilde P\times dt \times dx\right)$ as $n\to \infty$}. 
$$
Combining this and \eqref{eq:I2w-conv}, recalling 
$I(\tilde v^n,\tilde w^n)=I_1(\tilde v^n)+I_2(\tilde v^n)\tilde w^n$, we arrive finally at
\begin{align*}
	& \tilde \E \left[\int_0^T \En_Z(\omega,t) \left(\, \int_0^t \int_{\Om} 
	I(\tilde v^n,\tilde w^n) \Pi_n \vphi_i \ds \right)\dt \right]
	\\ & \qquad
	\to \tilde \E \left[\int_0^T \En_Z(\omega,t) \left(\, \int_0^t \int_{\Om} 
	I(\tilde v,\tilde w) \vphi_i \ds \right)\dt \right] 
	\quad \text{as $n\to \infty$}.
\end{align*}
This concludes the proof of \eqref{eq:limit-weak-form-tilde-tmp0} and thus the lemma.
\end{proof}

\subsection{Concluding the proof of Theorem \ref{thm:martingale}}\label{subsec:conclude}
As stated in Lemma \ref{eq:limit-eqs-tilde}, the Skorokhod-Jakubowski representations 
$\tilde U, \tilde W, \tilde v_0, \tilde w_0$ satisfy the weak form \eqref{eq:weakform-tilde} for 
a.e.~$t\in [0,T]$. Regarding the stochastic integrals in \eqref{eq:weakform-tilde}, the 
$(\tH(\Om))^*$ valued processes $\tilde v(t), \tilde w(t)$ are (by construction) 
$\tilde \cF_t$-measurable for each $t$. To upgrade \eqref{eq:weakform-tilde} to 
hold for ``every $t$", we will now prove that (cf.~also Remark \ref{rem:weak-L2-cont}) 
\begin{equation}\label{eq:weak-time-tilde}
	\tilde v(\omega), \tilde w(\omega) \in C([0,T];(\tH(\Om))^*), \quad 
	\text{for $\tilde P$-a.e.~$\omega\in \tilde D$}.
\end{equation}
This weak continuity property also ensures that $\tilde v$, $\tilde w$ 
are predictable in $(\tH(\Om))^*$. Hence, conditions
\eqref{eq:mart-adapt} and \eqref{eq:mart-weakform} 
in Definition \ref{def:martingale-sol} hold. Conditions 
\eqref{eq:mart-sto-basis} and \eqref{eq:mart-wiener} 
are covered by Lemma \ref{lem:wiener-tilde}, while 
Lemma \ref{lem:tlimits} validates 
conditions \eqref{eq:mart-uiue-reg} and \eqref{eq:mart-vreg}.
Lemma \ref{eq:limit-eqs-tilde} implies \eqref{eq:mart-data}.

To conclude the proof of Theorem \ref{thm:martingale},
it remains to verify \eqref{eq:weak-time-tilde}, which we 
do for $\tilde v$ (the case of $\tilde w$ is easier). 
Fix $\vphi\in \tH(\Om) \subset L^6(\Om)$, and 
consider the stochastic process
$$
\Psi_\vphi:\tilde D\times [0,T]\to \R,\quad
\Psi_\vphi(\omega,t):= \int_{\Om} \tilde v(\omega,t) \vphi \dx,
$$
relative to $\tilde \cS$, cf.~\eqref{eq:stoch-basis-new} and \eqref{eq:filtration-new}. 
To arrive at \eqref{eq:weak-time-tilde} it will be sufficient to prove that 
$\Psi_\vphi\in C([0,T])$ $\tilde P$-a.s., for any $\vphi$ in a countable dense subset 
$\Set{\vphi_\ell}_{\ell=1}^\infty\subset \tH(\Om)$. In what follows, 
let $\vphi$ denote an arbitrary function from $\Set{\vphi_\ell}_{\ell=1}^\infty$.

We are going to use the $L^{q_0}_{\omega}$ 
estimates in Corollary \ref{cor:Lq0-est}, 
with $q_0>\frac92$. Fix $t\in [0,T]$, $\vartheta>0$ (the 
case $\vartheta<0$ is treated similarly), and 
$q\in \left(3, \frac23 q_0\right]$. 
Then, using e.g.~the first equation in \eqref{eq:weakform-tilde},
\begin{align*}
	&\tilde \E\left[\abs{\Psi_\vphi(t+\vtheta)-\Psi_\vphi(t)}^q\right]
	\\ & \le \tilde \E \left[\abs{\int_t^{t+\vtheta} \int_{\Om}
	M_i\Grad \tilde u_i \cdot \Grad \vphi \dx\ds}^q\right]
	+\tilde \E \left[\abs{ \int_t^{t+\vtheta} 
	\int_{\Om}I(\tilde v,\tilde w) \vphi \dx\ds}^q\right]
	\\ & \qquad
	+ \tilde \E\left[\abs{\int_t^{t+\vtheta} \int_{\Om} 
	\eta(\tilde v) \vphi \dx\, d\tilde W^v (s)}^q\right]
	=: \Gamma_1 + \Gamma_2 + \Gamma_3.
\end{align*}
The $\Gamma_1$ term is estimated using the 
Cauchy-Schwarz inequality, the fact that 
$\Grad \tilde u_i \in L^{q_0}_\omega(L^2_{t,x})$, 
cf.~\eqref{eq:apriori-est-q0-tilde2}, and $q\le q_0$:
\begin{align*}
	\Gamma_1 & \le
	\tilde \E\left[\left(\, \int_t^{t+\vtheta} \int_{\Om} 
	\abs{\Grad \tilde u_i}^2\dx \ds \right)^{\frac{q}{2}}
	\left (\int_t^{t+\vtheta} \int_{\Om} 
	\abs{\Grad \vphi}^2 \dx \ds \right)^{\frac{q}{2}}\right]
	\\ & \le C_1 \abs{\vtheta}^{\frac{q}{2}}
	\norm{\Grad \vphi}_{L^2(\Om)}^q.
\end{align*}

Thanks to H\"older's inequality,
\begin{align*}
	\Gamma_2 & \le 
	\tilde \E\left[ \left(\, \int_t^{t+\vtheta} \int_{\Om} 
	\abs{I(\tilde v,\tilde w)}^{\frac43} \dx \ds \right)^{\frac{3q}{4}}
	\left ( \int_t^{t+\vtheta} \int_{\Om} \abs{\vphi}^3 \dx \ds 
	\right)^{\frac{q}{3}} \right]
	\\ & \le 
	\tilde C_2 \abs{\vtheta}^{\frac{q}{3}} 
	\tilde \E\left[ \left(\, \int_t^{t+\vtheta} \int_{\Om} 
	\left(\abs{\tilde v}^4 + \abs{\tilde w}^2 \right)
	\dx \ds \right)^{\frac{3q}{4}}\right]
	\norm{\vphi}_{L^3(\Om)}^{\frac{q}{3}}
	\\ & \le C_2 \abs{\vtheta}^{\frac{q}{3}}
	\norm{\vphi}_{L^3(\Om)}^q,
\end{align*}
using \eqref{eq:I43-est}, $\tilde v\in L^{2q_0}_\omega(L^4_{t,x})$, 
cf.~\eqref{eq:apriori-est-q0-tilde2}, $\tilde w\in L^{q_0}_\omega(L^\infty_t(L^2_x))$, 
cf.~\eqref{eq:Lq0-est}, and that the relevant exponents satisfy 
$3q\le 2q_0$, $3q/2\le q_0$.

Finally, we have
\begin{align*}
	\Gamma_3 
	& \le 
	\tilde \E \left[\norm{\sup_{\tau\in [0,\vtheta]} 
	\int_t^{t+\tau} \eta(\tilde v)  \, d\tilde W^v}_{L^2(\Om)}^q\right]
	\norm{\vphi}_{L^2(\Om)}^q
	\\ & 
	\overset{\eqref{eq:bdg}}{\le} \tilde C_3 \tilde \E\left[ \left(\, \int_t^{t+\vtheta} 
	\norm{\eta(\tilde v)}_{L_2(\U,L^2(\Om))}^2 \dt \right)^{\frac{q}{2}}\right]
	\norm{\vphi}_{L^2(\Om)}^q
	 \\ & \overset{\eqref{eq:noise-cond2}}{\le} 
	 \hat C_3 \abs{\vtheta}^{\frac{q}{2}}
	 \left(1+\tilde \E\left[\norm{\tilde v}_{L^\infty((0,T);L^2(\Om))}^q\right]\right) 
	 \norm{\vphi}_{L^2(\Om)}^q
	 \le C_3 \abs{\vtheta}^{\frac{q}{2}}
	 \norm{\vphi}_{L^2(\Om)}^q,
\end{align*}
since $\tilde v\in L^{q_0}_\omega(L^\infty_t(L^2_x))$ and $q\le q_0$.

Summarizing, with $t,t +\vtheta \in [0,T]$ 
and $\abs{\vtheta}<1$,  there exists a constant $C>0$ such that
$$
\tilde \E\left[\abs{\Psi_\vphi(t+\vtheta)-\Psi_\vphi(t)}^q\right] 
\le C \abs{\vtheta}^{\frac{q}{3}}\norm{\vphi}_{\tH(\Om)}^q
=C_\vphi \abs{\vtheta}^{1+\frac{q-3}{3}},
$$
where $C_\vphi:=C\norm{\vphi}_{\tH(\Om)}^q$.
Hence, Kolmogorov's continuity result 
(Theorem \ref{thm:kolm-cont}) can be applied with
$\kappa=q$, $\delta=\frac{q-3}{3}$,  
$\gamma=\frac{\delta}{\gamma}=\frac13 - \frac{1}{q}$ ($q>3$), 
securing the existence of a continuous modification of $\Psi_\vphi$. 
This concludes the proof of \eqref{eq:weak-time-tilde}.

\section{Uniqueness of weak (pathwise) solutions}\label{sec:uniq}
In this section we prove an $L^2$ stability estimate and consequently 
a pathwise uniqueness result. This result is used in the 
next section to conclude the existence of a unique weak solution to 
the stochastic bidomain model.

Let $\bigl(\cS, u_i, u_e,v, w\bigr)$ be a weak 
solution according to Definition \ref{def:martingale-sol}. 
We need a special case of the infinite dimensional version  
of It\^{o}'s formula \cite{DaPrato:2014aa,Krylov:2013aa,Prevot:2007aa}:
$$
d \norm{v(t)}_{L^2(\Om)}^2 
= 2\left(dv,v\right)_{(\tH(\Om))^*,\tH(\Om)}
+ 2\sum_{k\ge 1} \norm{\eta_k(v)}_{L^2(\Om)}^2\dt.
$$
To compute the first term on the right-hand side, multiply the first 
equation in \eqref{S1} by $u_i$, the second equation by $-u_e$, and 
sum the resulting equations. The outcome is
\begin{align*}	
	& v \, d v 
	- \sum_{j=i,e}\Div\bigl(M_j \Grad u_j\bigr) u_j \dt 
	+ v  I(v,w)\dt =  v\eta(v)\dW^v.
\end{align*}
Hence,
\begin{align*}
	&\left(dv,v\right)_{(\tH(\Om))^*,\tH(\Om)}
	= - \sum_{j=i,e} \left(M_j \Grad u_j , \Grad u_j\right)_{L^2(\Om)} \dt
	- \left(v,I(v,w)\right)_{L^2(\Om)} \dt
	\\ & \qquad\qquad  \qquad\qquad  \qquad\qquad 
	+  \sum_{k\ge 1}\left(v,\eta_k(v)\right)_{L^2(\Om)}\dW^v_k.
\end{align*}

Therefore, weak solutions of the stochastic bidomain model satisfy
the following  It\^{o} formula for the squared $L^2$ norm:
\begin{equation}\label{eq:Ito-L2-v}
	\begin{split}
		\norm{v(t)}_{L^2(\Om)}^2& =\norm{v(0)}_{L^2(\Om)}^2
		-2 \sum_{j=i,e} \int_0^t \int_{\Om} M_j \Grad u_j \cdot \Grad u_j\dx \ds
		\\ & \quad\quad  -2  \int_0^t \int_{\Om} v I(v,w)\dx \ds
		\\ & \quad\quad
		+2\sum_{k\ge 1}\int_0^t \int_{\Om}  \abs{\eta_k(v)}^2 \dx\ds
		+ 2\sum_{k\ge 1}\int_{\Om} v \,\eta_k(v)\dx \dW^v_k.
	\end{split}
\end{equation}
Additionally, from the (simpler) $w$-equation in \eqref{S1} 
we obtain
\begin{equation}\label{eq:Ito-L2-w}
	\begin{split}
		\norm{w(t)}_{L^2(\Om)}^2& =\norm{w(0)}_{L^2(\Om)}^2
		+ 2  \int_0^t \int_{\Om} w H(v,w)\dx \ds
		\\ & \qquad
		+2\sum_{k\ge 1}\int_0^t \int_{\Om}  \abs{\sigma_k(v)}^2 \dx\ds
		+ 2\sum_{k\ge 1}\int_{\Om} w \, \sigma_k(v)\dx \dW^w_k.
	\end{split}
\end{equation}

\begin{rem}\label{rem:krylov}
Krylov \cite{Krylov:2013aa} provides a rather simple proof of 
the It\^{o} formula for the squared norm $\norm{u_t}_H^2$, in the context of 
Hilbert space valued processes $u_t$ relative to  
a Gelfand triple $\Bbb{V} \subset \Bbb{H} 
\subset \Bbb{V}^*$, see also \cite[Theorem 4.2.5]{Prevot:2007aa}.
As part of the proof he also establishes the continuity 
of the process $t\to \norm{u_t}_H^2$ (more precisely, for a 
modification of $u_t$). The idea of his proof is to use an 
appropriate approximation procedure to lift the differential $du_t$ 
into $\Bbb{H}$, apply the It\^{o} formula for Hilbert space 
valued processes, and then pass to the limit in the approximations. 
Note that Definition \ref{def:martingale-sol} asks that 
the paths of $v(t)$ are weakly time continuous
but not that they belong to $C([0,T];L^2(\Om))$. 
However, as alluded to, this would actually be a consequence 
of the results in \cite{Krylov:2013aa}. 

Alternatively, we can argue for strong time continuity as follows:
Set 
$$
V(t):= v(t) - \int_0^t  \eta(v) \dW^v, \qquad 
W(t):= w(t) - \int_0^t  \sigma(w) \dW^w.
$$
Note that $P$-a.s., 
$$
V,W \in L^2((0,T);\tH(\Om)), 
\qquad
\partial_t V, \partial_t W \in L^2((0,T);(\tH(\Om))^*).
$$
Thus, via a classical result \cite{Temam:1977aa},
$V, W$ belong to $C([0,T];L^2(\Om))$. 
Since (by standard arguments \cite{DaPrato:2014aa})
$$
t\mapsto\int_0^t  \beta(v) \dW 
\in C([0,T];L^2(\Om)), \quad \text{$P$-a.s.},
$$
for $(\beta,W)=(\eta,W^v), (\sigma,W^w)$, we 
conclude that $P$-a.s.~$v,w\in C([0,T];L^2(\Om))$.
\end{rem}

We are now in a position to prove the stability result.

\begin{thm}\label{thm:uniq}
Suppose conditions {\rm (\textbf{GFHN})}, 
\eqref{matrix}, and \eqref{eq:noise-cond} hold. 
Let 
$$
\bar U=\bigl(\cS, \bar u_i,\bar u_e,\bar v, \bar w\bigr), \quad 
\hat U=\bigl(\cS, \hat u_i, \hat u_e,\hat v, \hat w\bigr)
$$ 
be two weak solutions (according to Definition \ref{def:martingale-sol}) 
relative to the same stochastic basis $\cS$, cf.~\eqref{eq:stochbasis}, with 
initial data $\bar v(0)=\bar v_0$, $\hat v(0)=\hat v_0$, 
$\bar w(0)=w_0$, $\hat w(0)=\hat w(0)$, where 
$$
\bar v_0, \hat v_0, \bar w_0,\hat w_0\in 
L^2\left(D,\cF,P;L^2(\Om)\right).
$$ 
There exists a positive constant $C\ge 1$ such that
\begin{equation}\label{eq:L2-stability}
	\begin{split}
		& \E \left [ \sup_{t\in [0,T]}\norm{\bar v(t)-\hat v(t)}_{L^2(\Om)}^2\right]
		+\sum_{j=i,e} \E\left[\norm{\bar u_j- \hat u_j}_{L^2(\Om_T)}^2\right]
		\\ & \quad 
		+\E \left [\sup_{t\in [0,T]} \norm{\bar w(t)-\hat w(t)}_{L^2(\Om)}^2\right]
		\\ & \quad \quad \quad
		\le C \left(
		\E \left [ \norm{\bar v_0-\hat v_0}_{L^2(\Om)}^2\right]
		+
		\E \left [\norm{\bar w_0-\hat w_0}_{L^2(\Om)}^2\right]
		\right).
	\end{split}
\end{equation}
With $\bar v_0=\hat v_0$, $\bar w_0=\hat w_0$, it follows 
that weak (pathwise) solutions are unique, cf.~\eqref{eq:path-uniq}.
\end{thm}

\begin{proof}[Proof of Theorem \ref{thm:uniq}]
Set $v:=\bar v- \hat v$, $u_i:=\bar u_i- \hat u_i$, $u_e:=\bar u_e- \hat u_e$, and 
$w:= \bar w - \hat w$. Note that $v= u_i - u_e$. We have $P$-a.s.,
\begin{align*}
	&u_i, \bar u_i, \hat u_i ,u_e, \bar u_e, \hat u_e, 
	v, \bar v, \hat v \in L^2((0,T);\tH(\Om)),
	\\ & 
	v, \bar v, \hat v, w,\bar w, \hat w \in L^\infty((0,T);L^2(\Om)) 
	\cap C([0,T];(\tH(\Om))^*).
\end{align*}
Actually, we can replace $C([0,T];(\tH(\Om))^*)$ by $C([0,T];L^2(\Om))$,  
see Remark \ref{rem:krylov}.

Subtracting the $(\tH(\Om))^*$ valued 
equations for $\bar U, \hat U$, cf.~ \eqref{S1}, we obtain 
\begin{equation}\label{eq:uniq-v}
	\begin{split}
		& d v  -\Div\bigl(M_i \Grad u_i \bigr) \dt
		+ \left(I(\bar v,\bar w)-I(\hat v,\hat w)\right)\dt = 
		\left(\eta(\bar v)-\eta(\hat v)\right) \dW^v,
		\\ &
		d v +\Div\bigl(M_e\Grad u_e\bigr) \dt
	 	+ \left(I(\bar v,\bar w)-I(\hat v,\hat w)\right) \dt  =  
	 	\left(\eta(\bar v)-\eta(\hat v)\right)  \dW^v,
		\\ &
		d w=\left(H(\bar v,\bar w)-H(\hat v, \hat w)\right)\dt 
		+\left(\sigma(\bar v)-\sigma(\hat v)\right) \dW^w.
	\end{split}
\end{equation}
Given the equations in \eqref{eq:uniq-v}, we apply 
the It\^{o} formula, cf.~\eqref{eq:Ito-L2-v} and \eqref{eq:Ito-L2-w}.
Using \eqref{matrix} and adding the resulting equations, we obtain 
in the end the following inequality: 
\begin{equation}\label{eq:Ito-L2-v-diff}
	\begin{split}
		&\frac12\norm{v(t)}_{L^2(\Om)}^2
		+\frac12\norm{w(t)}_{L^2(\Om)}^2
		+m \sum_{j=i,e} \int_0^t \int_{\Om} \abs{\Grad u_j}^2 \dx \ds
		\\ &  
		\, \le \frac12\norm{v(0)}_{L^2(\Om)}^2
		+\frac12\norm{w(0)}_{L^2(\Om)}^2
		\\ & \quad
		+ \int_0^t \int_{\Om} 
		\Bigl(w\left(H(\bar v,\bar w)-H(\hat v, \hat w)\right)
		- v \left(I(\bar v,\bar w)-I(\hat v,\hat w)\right)\Bigr)\dx \ds
		\\ &  \quad
		+\sum_{k\ge 1}\int_0^t \int_{\Om}  \abs{\eta_k(\bar v)-\eta_k(\hat v)}^2 \dx\ds
		+\sum_{k\ge 1}\int_0^t \int_{\Om}  \abs{\sigma_k(\bar v)-\sigma_k(\hat v)}^2 \dx\ds
		\\ &  \quad
		+\sum_{k\ge 1}\int_{\Om} v \left(\eta_k(\bar v)-\eta_k(\hat v)\right)\dx \dW^v_k
		+\sum_{k\ge 1}\int_{\Om} w \left(\sigma_k(\bar v)-\sigma_k(\hat v)\right)\dx \dW^w_k.
	\end{split}
\end{equation}
As in \cite[page 479]{Bourgault:2009aa}, we can 
use {\rm (\textbf{GFHN})} to bound the the third term on 
right-hand side by a constant times
$$
\int_0^t\left(  \norm{v(s)}_{L^2(\Om)}^2 + \norm{w(s)}_{L^2(\Om)}^2\right)\ds.
$$
Using \eqref{eq:noise-cond}, we can bound the fourth term on right-hand
of \eqref{eq:Ito-L2-v-diff} by a constant times 
$$
\int_0^t \norm{v(s)}_{L^2(\Om)}^2\ds.
$$
The stochastic integrals in \eqref{eq:Ito-L2-v-diff} 
are square-integrable, zero mean martingales. Moreover, using
the Poincar\'e inequality, we have
$$
\int_0^t \int_{\Om} \abs{ u_e}^2 \dx \ds 
\leq \tilde C \int_0^t \int_{\Om} \abs{\Grad u_e}^2 \dx \ds,
$$
for some constant $\tilde C>0$. Since $u_i=v+u_e$, we control 
$u_i$ as well. As a result of all this, there is a constant $C>0$ such that 
\begin{equation*}
	\begin{split}
		&\E\left[\norm{v(t)}_{L^2(\Om)}^2\right]
		+\sum_{j=i,e} \E\left[\norm{u_j}_{L^2(\Om_T)}^2\right]
		+\E \left[\norm{w(t)}_{L^2(\Om)}^2\right]
		\\ & \qquad 
		\le  \E\left[\norm{v(0)}_{L^2(\Om)}^2\right]
		+\E \left[\norm{w(0)}_{L^2(\Om)}^2\right]
		\\ & \qquad \qquad 
		+ C \int_0^t\left(  \E\left[\norm{v(s)}_{L^2(\Om}^2\right] 
		+ \E\left [\norm{w(s)}_{L^2(\Om}^2\right]\right)\ds,
	\end{split}
\end{equation*}
for any $t\in [0,T]$. The Gr\"onwall inequality 
delivers \eqref{eq:L2-stability} ``without $\sup$". The refinement 
\eqref{eq:L2-stability} (``with $\sup$") comes from the 
application of a martingale inequality \eqref{eq:bdg}, 
see \eqref{Esup1} for a similar situation.

\end{proof}

\section{Existence of weak (pathwise) solution}\label{sec:pathwise}
In this section we prove that the stochastic bidomain model possesses a unique weak 
(pathwise) solution in the sense of Definition \ref{def:pathwise-sol}, thereby 
proving Theorem \ref{thm:pathwise}. The proof 
follows the traditional Yamada-Watanabe approach (see for example 
\cite{Debussche:2011aa,Glatt-Holtz:2014aa,Ikeda:1981aa,Kurtz:2014aa,Prevot:2007aa}), 
combining the existence of at least one weak martingale 
solution (Theorem  \ref{thm:martingale}) with a pathwise 
uniqueness result (Theorem \ref{thm:uniq}), relying on 
the Gy\"ongy-Krylov characterization of convergence 
in probability (Lemma \ref{lem:krylov}).

Referring to Section \ref{sec:approx-sol} for details, recall that 
$$
U^n = \bigl(u_i^n,u_e^n,v^n,w^n\bigr),
\, 
W^n=\bigl(W^{v,n}, W^{w,n}\bigr),
\, 
U_0^n =\bigl(u_{i,0}^n, u_{e,0}^n, v_0^n, w_0^n\bigr)
$$
is the Faedo-Galerkin solution, defined on some  fixed stochastic basis 
$$
\cS=\bigl(D,\cF,\Set{\cFt}_{t\in [0,T]},P,W\bigr), \quad W=\bigl(W^v,W^w\bigr),
$$ 
where $W^v=\Set{W_k^v}_{k\ge 1}$, $W^w=\Set{W_k^w}_{k\ge 1}$ 
are cylindrical Wiener processes.  Moreover, recall that $\cL_n$ is 
the probability law of the random variable
$$
\Phi_n: D \to \cX=\cX_U\times \cX_W\times \cX_{U_0}, \quad
\Phi_n (\omega) = \left(U^n(\omega), W^n(\omega),U_0^n(\omega)\right).
$$

We intend to show that the approximate solutions $U^n$ 
converge in probability (in $\cX_U$) to a random variable 
$U = \left(u_i,u_e,v,w\right)$ (defined on $\cS$). 
Passing to a subsequence if necessary, we may as well replace 
convergence in probability by a.s.~convergence. We 
then argue as in Subsection \ref{subsec:limit} 
to arrive at the conclusion that the limit $U$ of $\Set{U^n}_{n\ge 1}$ is a 
weak (pathwise) solution of the stochastic bidomain model.

It remains to prove that $\Set{U^n}_{n\ge 1}$ converges in probability. 
To this end, we will use the Gy\"ongy-Krylov lemma along with 
pathwise uniqueness. By Lemma \ref{lem:tight}, the sequence 
$\Set{\cL_n}_{n\ge1}$ is tight on $\cX$.
For $n,m\ge 1$, denote by $\cL_{n,m}$ the law of the random variable
$$
\Phi_{n,m}(\omega) = \left(U^n(\omega),U^m(\omega), W^n(\omega),
U_0^n(\omega),U_0^m(\omega)\right),
$$
defined on the extended path space
\begin{align*}
	\cX^E & := \cX_U \times \cX_U \times \cX_W
	\times \cX_{U_0}\times \cX_{U_0}.
\end{align*}
Clearly, we have $\cL_{n,m}=\cL_{U^n }\times \cL_{U^m}
\times \cL_{W^n}\times \cL_{U_0^n}\times \cL_{U_0^m}$ (see 
Subsection \ref{subsec:tight} for the notation). 
With obvious modifications to the proof of Lemma \ref{lem:tight}, we 
conclude that $\Set{\cL_{n,m}}_{n,m\ge 1}$ is tight on $\cX^E$. 
Let us now fix an arbitrary subsequence $\Set{\cL_{n_k,m_k}}_{k\ge 1}$ of 
$\Set{\cL_{n,m}}_{n,m\ge 1}$, which obviously is also tight on $\cX^E$.

Passing to a further subsequence if needed (without relabelling as usual), the 
Skorokhod-Jakubowski representation theorem provides a new probability 
space $(\tilde{D},\tilde{\cF}, \tilde{P})$ and new $\cX^E$-valued random variables
\begin{equation}\label{eq:uniq-as-conv}
	\left(\bar U^{n_k}, \hat U^{m_k}, \tilde W^{n_k}, \bar U_0^{n_k},\hat U_0^{m_k}\right),
	\quad
	\left(\bar U, \hat U, \tilde W, \bar U_0,\hat U_0\right) 
\end{equation}
on $(\tilde{D},\tilde{\cF}, \tilde{P})$, such that the law of 
$$
\left(\bar U^{n_k}, \hat U^{m_k}, \tilde W^{n_k},
\bar U_0^{n_k},\hat U_0^{m_k}\right)
$$
is $\cL_{n_k,m_k}$ and as $k\to \infty$,
$$
\left(\bar U^{n_k}, \hat U^{m_k}, \tilde W^{n_k},\bar U_0^{n_k},\hat U_0^{m_k}\right) 
\to \left(\bar U, \hat U, \tilde W,\bar U_0,\hat U_0\right) \quad  
\text{$\tilde P$-almost surely (in $\cX^E$).}
$$
Observe that
$$
\tilde P \left(\Set{\omega\in \tilde D: \bar U_0(\omega)=\hat U_0(\omega)}\right)=1.
$$
Indeed,
$U_0^{n_k}=\Pi_{n_k} U_0$ and $U_0^{m_k}=\Pi_{m_k} U_0$, 
and so for any $\ell\le \min (n_k,m_k)$,
\begin{align*}
	&\tilde P \left(\Set{\omega\in \tilde D: 
	\Pi_\ell \bar U_0^{n_k}
	=\Pi_\ell \hat U_0^{m_k}}\right)
	\\ & \qquad 
	= P \left(\Set{\omega\in D: 
	\Pi_\ell U_0^{n_k}
	=\Pi_\ell U_0^{m_k}}\right)=1,
\end{align*}
by equality of the laws. This proves the claim.

Applying the arguments in Subsection \ref{subsec:limit} separately to 
$$
\bar U^{n_k}, \tilde W^{n_k}, \bar U_0^{n_k},  
\bar U, \tilde W, \bar U_0
\quad \text{and} \quad
\hat U^{m_k}, \tilde W^{n_k}, \hat U_0^{m_k}, 
\hat U, \tilde W, \hat U_0,
$$
it follows that $\left(\bar U, \tilde W,\bar U_0\right)$ and 
$\left(\hat U, \tilde W,\hat U_0\right)$ are both 
weak martingale solutions, relative 
to the same stochastic basis  
$$
\tilde \cS=\left( \tilde D,\tilde \cF,\Set{\tcFt}_{t\in [0,T]}, \tilde P, \tilde W\right), 
\quad \tilde W=\tilde W^v, \tilde W^w,
$$
where
$$
\tcFt= \sigma\left(
\sigma \bigl(
\bar U\big |_{[0,t]},
\hat U\big |_{[0,t]},
\tilde W\big |_{[0,t]},
\bar U_0
\bigr)\bigcup  
\bigl\{N \in \tilde \cF: \tilde P (N)=0\bigr\}\right), 
\quad t\in [0,T].
$$
Denote by $\mu_{n_k,m_k}$ and $\mu$ the joint laws of 
$\left(\bar U^{n_k}, \hat U^{m_k}\right)$ and 
$\left(\bar U, \hat U\right)$, respectively. 
Then, in view of \eqref{eq:uniq-as-conv}, $\mu_{n_k,m_k}\weak \mu$ 
as $k\to \infty$. Since $\bar U_0=\hat U_0$ $\tilde P$-a.s., Theorem \ref{thm:uniq} ensures 
that $\bar U = \hat U$ $\tilde P$-a.s.~(in $\cX_U$). Hence, since this implies
$$
\mu \left(\Set{(X, Y)\in \cX_U\times \cX_U: X=Y}\right)
=\tilde P\left(\Set{(\omega\in \tilde D: \bar U (\omega)=\hat U (\omega)}\right)=1,
$$
we can appeal to Lemma \ref{lem:krylov}, cf.~Remark \ref{rem:GyKr}, to complete the proof.

\end{document}